\documentclass[11pt]{article}
\usepackage[utf8]{inputenc}

\usepackage[nottoc]{tocbibind}

\usepackage[abbrev,lite,alphabetic]{amsrefs}
\usepackage[colorlinks,plainpages,urlcolor=blue,linkcolor=reference,citecolor=citation,colorlinks = true, hyperfootnotes=false]{hyperref}
\hypersetup{citecolor=gray, linkcolor = darkgray, urlcolor= gray}

\usepackage[T1]{fontenc}
\usepackage{etex}
\usepackage[shortlabels]{enumitem}
\usepackage{moreenum}
\usepackage{subfiles}
 \usepackage{lmodern}
\usepackage{titlesec}
\usepackage{wrapfig}
\usepackage{csquotes}

\usepackage{stmaryrd}

\usepackage{tikz-cd}
\usepackage{mathrsfs}
\usepackage{sseq}
\usepackage{mathdots}
\usepackage{thm-restate}
\usepackage{amsmath}
\usepackage{amsxtra}
\usepackage{amscd}
\usepackage{amsthm}
\usepackage{amsfonts}
\usepackage{amssymb}
\usepackage{mathtools}
\usepackage[scr=rsfs]{mathalfa}
\usepackage{eucal}
\usepackage{bbm}
\usepackage[all,cmtip]{xy}

\usepackage{graphicx}
\usepackage{tikz}
\usepackage{morefloats}
\usepackage{pdfpages}
\usetikzlibrary{matrix,arrows,decorations.pathmorphing,quotes,angles,decorations.markings,shapes.misc}
\RequirePackage{color}

\definecolor{darkpastelpurple}{rgb}{0.59, 0.44, 0.84}
\usepackage{microtype}
\usepackage{epigraph}
\usepackage{tocloft}
\usepackage[colorinlistoftodos, textsize=tiny]{todonotes}

\usepackage[headsep=1cm,headheight=2cm,margin=2.7cm]{geometry}
\usepackage[framemethod=TikZ]{mdframed}
\usepackage{setspace}
\titleformat{\subsection}[hang]{\large\bfseries}{\thesubsection\hsp\textcolor{gray!75}{|}\hsp}{0pt}{\large\bfseries}
\titleformat{\subsubsection}[hang]{\bfseries}{\thesubsubsection\hsp\textcolor{gray!75}{|}\hsp}{0pt}{\bfseries}
\titleformat{\part}[display]{\Huge\bfseries}{\partname~\thepart:}{20pt}{}{}
\newcommand{\hsp}{\hspace{20pt}}
\titleformat{\section}[hang]{\Large\bfseries}{\thesection\hsp\textcolor{gray!75}{|}\hsp}{0pt}{\Large\bfseries}

\makeatletter
\let\pprevious@stem\@empty

\def\fooo#1#2\global\let\previous@stem\current@stem#3\zz{%
   \def#1{#2%
\global\let\pprevious@stem\previous@stem
\global\let\previous@stem\current@stem
#3}}

\def\foo#1{\expandafter\fooo\expandafter#1#1\zz}

\foo\generate@alphalabel

\let\generate@label\generate@alphalabel

\def\calc@alpha@suffix{%
    \@tempswafalse
    \compare@stems\previous@stem\current@stem
    \ifsame@stems
        \ifx\previous@year\current@year
            \@tempswatrue
        \fi
    \else
    \compare@stems\pprevious@stem\current@stem
    \ifsame@stems
        \ifx\previous@year\current@year
            \@tempswatrue
        \fi
    \fi
    \fi
    \if@tempswa
        \global\advance\alpha@suffix\@ne
        \edef\alpha@label@suffix{\@suffix@format\alpha@suffix}%
        \ifnum\alpha@suffix=\tw@
            \immediate\write\@auxout{%
                \string\ModifyBibLabel{\prev@citekey}%
            }%
        \fi
    \else
        \let\alpha@label@suffix\@empty
        \global\alpha@suffix\@ne
        \@xp\ifx \csname b@\current@citekey @suffix\endcsname \relax
        \else
            \edef\alpha@label@suffix{\@suffix@format\alpha@suffix}%
        \fi
    \fi
}

\makeatother

\usepackage{cleveref}
\definecolor{reference}{rgb}{0.20,0.36,0.74}
\definecolor{citation}{rgb}{0,.40,.80}

\usepackage{makeidx}

\makeindex

\usepackage{lipsum}
\usepackage{cleveref}
\crefname{section}{\S \!\!}{\S\S \!\!}
\crefname{equation}{}{}
\crefname{enumi}{}{}

\setenumerate[1]{label={(\arabic*)}}

\allowdisplaybreaks

\theoremstyle{plain}
\newtheorem{theorem}{Theorem}[section]
\newtheorem{mainthm}{Theorem}

\newtheorem{proposition}[theorem]{Proposition}
\newtheorem{lemma}[theorem]{Lemma}
\newtheorem{corollary}[theorem]{Corollary}
\theoremstyle{definition}

\newtheorem{definition}[theorem]{Definition}
\newtheorem{construction}[theorem]{Construction}
\newtheorem{notation}[theorem]{Notation}
\newtheorem{observation}[theorem]{Observation}
\newtheorem{remark}[theorem]{Remark}
\newtheorem{example}[theorem]{Example}
\newtheorem{warning}[theorem]{Warning}

\def\cA{\mathcal A}\def\cB{\mathcal B}\def\cC{\mathcal C}\def\cD{\mathcal D}
\def\cE{\mathcal E}\def\cF{\mathcal F}
\def\cI{\mathcal I}
\def\cM{\mathcal M}\def\cO{\mathcal O}
\def\cS{\mathcal S}
\def\cU{\mathcal U}\def\cV{\mathcal V}\def\cX{\mathcal X}
\def\cY{\mathcal Y}\def\cZ{\mathcal Z}

\newcommand{\Spectra}{{\cS}{\sf p}}
\newcommand{\Sp}{{\cS}{\sf p}}
\newcommand{\CycSp}{\mathsf{CycSp}}
\newcommand{\Spaces}{{\cS}}

\newcommand{\ra}{\rightarrow}

\newcommand{\xra}{\xrightarrow}

\newcommand{\xla}{\xleftarrow}

\newcommand{\xhookra}{\xhookrightarrow}

\newcommand{\longla}{\longleftarrow}

\newcommand{\longra}{\longrightarrow}

\newcommand{\xlongra}[1]{\stackrel{#1}{\longra}}
\newcommand{\xlongla}[1]{\stackrel{#1}{\longla}}

\newcommand{\longhookra}{\lhook\joinrel\longrightarrow}

\makeatletter
\providecommand{\leftsquigarrow}{%
  \mathrel{\mathpalette\reflect@squig\relax}%
}
\newcommand{\reflect@squig}[2]{%
  \reflectbox{$\m@th#1\rightsquigarrow$}%
}
\makeatother

\newcommand{\hookra}{\hookrightarrow}
\newcommand{\hookla}{\hookleftarrow}

\newcommand{\Shv}{{\sf Shv}}

\newcommand{\HomforStrat}{\Yleft}
\newcommand{\HomforCAlg}{\Yright}

\newcommand{\adj}{\dashv}

\DeclareMathOperator{\uno}{\mathbbm{1}}
\newcommand{\ev}{{\sf ev}}
\newcommand{\pr}{{\sf pr}}

\newcommand{\Fun}{{\sf Fun}}
\newcommand{\Hom}{{\sf Hom}}

\newcommand{\Fin}{{\sf Fin}}
\newcommand{\Bord}{{\sf Bord}}
\newcommand{\BBord}{\mathfrak{B}{\sf ord}}
\newcommand{\Strat}{{\sf Strat}}
\newcommand{\Stratlax}{\Strat^\lax}
\newcommand{\Cat}{{\sf Cat}}

\newcommand{\PrLSt}{{\sf Pr}^L_{{\sf st}}}

\newcommand{\es}{\varnothing}

\newcommand{\Spec}{{\sf Spec}}
\newcommand{\QC}{{\sf QC}}

\newcommand{\Mod}{{\sf Mod}}

\newcommand{\Pic}{{\sf Pic}}

\newcommand{\fgt}{{\sf fgt}}

\newcommand{\dzbl}{{\sf dzbl}}

\renewcommand{\lim}{{\sf lim}}
\newcommand{\lax}{{\sf lax}}
\newcommand{\Sect}{{\sf Sect}}

\newcommand{\const}{{\sf const}}
\newcommand{\pt}{{\sf pt}}

\newcommand{\id}{{\sf id}}
\newcommand{\Ar}{{\sf Ar}}
\newcommand{\Map}{{\sf Hom}}

\newcommand{\pos}{{P}}

\newcommand{\bit}[1]{\textbf{\textit{#1}}}

\newcommand{\can}{{\sf can}}

\newcommand{\Bifib}{{\sf Bifib}}
\newcommand{\llax}{{\sf llax}}
\newcommand{\rlax}{{\sf rlax}}
\newcommand{\lstr}{{\sf lstrict}}
\newcommand{\rstr}{{\sf rstrict}}

\newcommand{\cocart}{{\sf cocart}}

\newcommand{\op}{{\sf op}}

\newcommand{\Cocart}{{\sf Cocart}}
\newcommand{\Cart}{{\sf Cart}}
\newcommand{\locCocart}{{\sf locCocart}}

\newcommand{\Mon}{{\sf Mon}}

\newcommand{\CAlg}{{\sf CAlg}}
\newcommand{\Alg}{{\sf Alg}}
\newcommand{\Com}{{\sf Com}}

\newcommand{\adjunct}[4]{ \begin{tikzcd}[column sep=2cm, ampersand replacement=\&] {#2} \arrow[yshift=0.9ex]{r}{{#1}} \arrow[leftarrow, yshift=-0.9ex]{r}[yshift=-0.2ex]{\bot}[swap]{{#4}} \& {#3} \end{tikzcd} }

\newcommand{\adjunctbig}[4]{ \begin{tikzcd}[column sep=4cm, ampersand replacement=\&] {#2} \arrow[yshift=0.9ex]{r}{{#1}} \arrow[leftarrow, yshift=-0.9ex]{r}[yshift=-0.2ex]{\bot}[swap]{{#4}} \& {#3} \end{tikzcd} }

\newcommand{\adjunctbigger}[4]{ \begin{tikzcd}[column sep=5cm, ampersand replacement=\&] {#2} \arrow[yshift=0.9ex]{r}{{#1}} \arrow[leftarrow, yshift=-0.9ex]{r}[yshift=-0.2ex]{\bot}[swap]{{#4}} \& {#3} \end{tikzcd} }

\newcommand{\categ}[1]{\textbf{\textup{#1}}}
\newcommand{\sd}{\mathrm{sd}}
\newcommand{\Tw}{\mathrm{Tw}}
\newcommand{\sSet}{\categ{sSet}}
\newcommand{\ob}{\mathrm{ob}}
\newcommand{\res}{\mathrm{res}}

\newcommand{\Yo}{\mathrm{Yo}}

\renewcommand{\max}{\mathrm{max}}
\renewcommand{\min}{\mathrm{min}}

\newcommand{\leftnat}[1]{\vphantom{#1}_{\natural}\mskip-1mu{#1}}

\newcommand{\colim}{\mathrm{colim}}

\newcommand{\EE}{\mathbb{E}}
\newcommand{\NN}{\mathbb{N}}

\newcommand{\SSS}{\mathbb{S}}
\newcommand{\QQ}{\mathbb{Q}}
\newcommand{\ZZ}{\mathbb{Z}}

\newcommand{\SMC}{\mathsf{SMC}}
\newcommand{\Comm}{{\sf Comm}}

\newcommand{\Env}{{\sf Env}}

\newcommand{\End}{\mathsf{End}}
\newcommand{\LEq}{\mathsf{LEq}}

\newcommand{\THH}{\mathrm{THH}}

\newcommand{\triv}{\mathrm{triv}}

\newcommand{\Span}{\mathsf{Span}}

\title{Dualizable objects in stratified categories and the 1-dimensional bordism hypothesis for recollements}
\author{Grigory Kondyrev, Aaron Mazel-Gee, and Jay Shah}
\date{\today}

\begin{document}

\maketitle

\begin{abstract}
Given a monoidal $\infty$-category $\cC$ equipped with a monoidal recollement, we give a simple criterion for an object in $\cC$ to be dualizable in terms of the dualizability of each of its factors and a projection formula relating them. Predicated on this, we then characterize dualizability in any monoidally stratified $\infty$-category in terms of stratumwise dualizability and a projection formula for the links.

Using our criterion, we prove a $1$-dimensional bordism hypothesis for symmetric monoidal recollements. Namely, we provide an algebraic enhancement of the 1-dimensional framed bordism $\infty$-category that corepresents dualizable objects in symmetric monoidal recollements.

We also give a number of examples and applications of our criterion drawn from algebra and homotopy theory, including equivariant and cyclotomic spectra and a multiplicative form of the Thom isomorphism.
\end{abstract}

% Given a monoidal $\infty$-category equipped with a monoidal stratification (in the sense of [ref]), we give a simple criterion for its dualizable objects in terms of a projection formula. we unpack the examples of [some genuine G-spectra, maybe $K(1)$-local homotopy theory, ...]. we also use our result to prove a 1-dimensional bordism hypothesis for symmetric monoidal recollements.
%we describe a stratified bordism category that corepresents dualizable objects in symmetric monoidal recollements. the main ingredient is a surprisingly simple characterization of dualizable objects in symmetric monoidal recollements in terms of a projection formula, which appears to be new.

\setcounter{tocdepth}{2}
\tableofcontents
 
\setcounter{section}{-1}

\section{Introduction}

\subsection{Overview}
\label{subsection.overview}

% \footnote{Here and below, we implicitly take sheaves in any pointed category with finite limits.}
Broadly speaking, a \textit{stratification} of a geometric object is a decomposition into smaller pieces, called the \textit{strata} of the stratification; the set of strata form a poset $\pos$ according to closure relations. In favorable situations, one can reconstruct the original object from the strata along with appropriate gluing data: its \textit{links}, indexed by morphisms $p_0 < p_1$ in $\pos$, and more generally its \textit{higher links}, indexed by strings of morphisms $p_0 < \cdots < p_n$ in $\pos$. This pattern persists upon passing to categories of sheaves, and leads to the general notion of a \bit{stratified category} along with its corresponding notions of strata, links, and higher links \cite{AMGR}; the notion of a \bit{recollement} is recovered in the case of the poset $P = [1] \coloneqq \{ 0 < 1 \}$ \cite{BBD}.\footnote{Throughout this paper, we take the ``implicit $\infty$'' convention: ``category'' means ``$\infty$-category'', and all notions (e.g., co/limits) are implicitly $\infty$-categorical. In particular, our results apply equally well to ordinary categories.} Here are two simple examples (see \Cref{example.many.stratns} for more).
\begin{itemize}
\item A stratification of the scheme $\Spec(\ZZ)$ leads 
%As a simple example, a stratification of a topological space over the poset $[1]$ is simply a closed-open decomposition, and this determines a recollement of its category of sheaves (as we recall in \Cref{subsubsection.intro.dzbl.in.rec}). Likewise, stratifications of schemes determine stratifications of their categories of quasicoherent sheaves. For example, in the simple case of $\Spec(\ZZ)$, this leads
to the \textit{arithmetic fracture square}, a pullback square
\[ \begin{tikzcd}
M
\arrow{r}
\arrow{d}
&
\QQ \otimes_\ZZ M
\arrow{d}
\\
{\displaystyle \prod_{p \text{ prime}} M^\wedge_p}
\arrow{r}
&
\QQ \otimes_\ZZ \left(
{\displaystyle \prod_{p \text{ prime}} M^\wedge_p}
\right)
\end{tikzcd} \]
expressing any quasicoherent sheaf $M \in \QC(\Spec(\ZZ)) \simeq \Mod_\ZZ$ in terms of its rationalization, $p$-completions, and gluing data thereamong.
\item A stratification of a topological space $X$ over $[1]$ is simply a closed-open decomposition. This determines a recollement of its category $\Shv(X)$ of sheaves (as we recall in \Cref{subsubsection.intro.dzbl.in.rec}), which affords an analogous pullback square expressing any sheaf $\cF \in \Shv(X)$ in terms of its restrictions and gluing data.
\end{itemize}
Of course, for stratifications over larger posets, the corresponding reconstruction formulas for objects are more elaborate.

% More broadly, stratified categories are ubiquitous throughout algebraic geometry and homotopy theory (cf. \Cref{example.many.stratns} below).

% More broadly, stratified categories appear in {\color{orange} [give list and refs: modules over rings; qcoh over schemes (and prestacks?); equivariant homotopy theory; Goodwillie calculus; chromatic homotopy theory; D-modules? [ref Kashiwara's theorem?]; topoi?]} {\color{blue} [ref \Cref{example.many.stratns} below]}

% {\color{orange} Many/all of the above-listed stratifications are symmetric monoidal.}

The notion of a stratified category can be upgraded to account for symmetric monoidal structures, and in fact stratifications of geometric objects determine \textit{symmetric monoidal} stratifications of their categories of sheaves. However, the notion of a symmetric monoidal stratification is somewhat subtle, because in general the gluing functors between strata are only \textit{laxly} symmetric monoidal.\footnote{Our conventions (e.g., regarding the term `laxly symmetric monoidal') are laid out in \Cref{subsection.intro.notation}.} As a result, it is generally nontrivial to study questions of a global nature in a symmetric monoidally stratified category that relate to its symmetric monoidal structure.

Our first main result answers a very basic such global question: namely, we give a characterization of the dualizable objects in a symmetric monoidally stratified category. This is nontrivial even for symmetric monoidal recollements (i.e., stratified over the poset $[1]$), where it reduces to stratumwise dualizability and a projection formula for the link. Perhaps surprisingly, it turns out that dualizability over an arbitrary poset is likewise detected merely by stratumwise dualizability and a projection formula for the links: one need not check conditions for higher links. We state these results in \Cref{subsection.intro.dzbl.in.strat}, after reviewing the relevant definitions.

Using this, we prove our second main result: a $[1]$-stratified generalization of the 1-dimensional bordism hypothesis. Namely, we construct a \bit{$[1]$-stratified 1-dimensional bordism category} and prove that it corepresents dualizable objects among symmetric monoidal recollements. We state this result in \Cref{subsection.intro.bord}.

We describe a number of examples and applications of our results in \Cref{section.applications.and.examples}, related to: modules over complete local rings; genuine $G$-spectra; cyclotomic spectra; derived Mackey functors; Picard groups; and Thom spectra and orientations.

\subsection{Dualizable objects in stratified categories}
\label{subsection.intro.dzbl.in.strat}

Here we explain our characterization of the dualizable objects in a symmetric monoidally stratified category. For expository purposes, we begin with the special case of a symmetric monoidal recollement.\footnote{In fact, in the main body of the paper we establish a slightly more general result, namely a characterization of the dualizable objects in a (not necessarily symmetric) monoidal recollement. We omit this generalization here; the only difference is that one must keep track of the handedness of duals.}

\subsubsection{Dualizable objects in recollements}
\label{subsubsection.intro.dzbl.in.rec}

We define a \bit{$[1]$-stratified symmetric monoidal category} to be a laxly symmetric monoidal functor
\[
\cU
\xlongra{\varphi}
\cZ
~,
\]
and we define its \bit{underlying symmetric monoidal category} to be the pullback
\[ \begin{tikzcd}
\cX
\coloneqq
&[-1.2cm]
\lim^\lax(\varphi)
\arrow{r}
\arrow{d}
&
\Ar(\cZ)
\arrow{d}{\ev_1}
\\
&
\cU
\arrow{r}[swap]{\varphi}
&
\cZ
\end{tikzcd}
~.
\]
We depict a typical object of $\cX$ by $[u \mapsto \varphi(u) \xleftarrow{\alpha} z]$, where $u \in \cU$ and $z \in \cZ$; in these terms, its symmetric monoidal structure is given by the formula
\[
[u \mapsto \varphi(u) \xleftarrow{\alpha} z]
\otimes
[u' \mapsto \varphi(u') \xleftarrow{\alpha'} z']
\coloneqq
[ u \otimes u'
\mapsto
\varphi(u \otimes u')
\xla{\can}
\varphi(u) \otimes \varphi(u')
\xla{\alpha \otimes \alpha'}
z \otimes z'
]
\]
(using the laxness of $\varphi$). In this situation, we always have a solid diagram
\begin{equation}
\label{intro.smrec}
\begin{tikzcd}[column sep=2cm]
\cU
\arrow[leftarrow]{r}[yshift=1.5ex]{\bot}[swap, yshift=-1.5ex]{\bot}[description]{j^*}
\arrow[hook, bend left=50, dashed]{r}{j_!}
\arrow[hook, bend right=50]{r}[swap]{j_*}
&
\cX
\arrow[dashed, hookleftarrow]{r}[yshift=1.5ex]{\bot}[swap, yshift=-1.5ex]{\bot}[description]{i_*}
\arrow[bend left=50]{r}{i^*}
\arrow[dashed, bend right=50]{r}[swap]{i^!}
&
\cZ
\end{tikzcd}
~,
\end{equation}
in which $j^*$ and $i^*$ are symmetric monoidal, $j_*$ is laxly symmetric monoidal, and $\varphi \simeq i^* j_*$; the dashed adjoints exist under mild hypotheses,\footnote{Specifically: $j_!$ exists when $\cZ$ has an initial object; $i_*$ exists when $\cU$ and $\cZ$ have terminal objects and $\varphi(\pt_\cU) \simeq \pt_\cZ$ (or more generally if $\cU$ has a terminal object and the functor $\cZ \xra{\varphi(\pt_\cU) \times (-)} \cZ$ exists); and $i^!$ exists when ($i_*$ exists and moreover) $\cU$ is pointed and $\cZ$ admits fibers.} making diagram \Cref{intro.smrec} into a \bit{symmetric monoidal recollement}.\footnote{Recollements were originally introduced by Beilinson--Bernstein--Deligne \cite{BBD}*{\S 1.4.3}. In \cite{HA}*{\S A.8.1}, Lurie gives the following more general definition of a recollement: $\cU$, $\cZ$, and $\cX$ all admit finite limits and are related by the pair of adjunctions $j^* \adj j_*$ and $i^* \adj i_*$, subject to the conditions that $j^*$, $i^*$ are left-exact, $j_*$, $i_*$ are fully faithful, $j^* i_* \simeq \const_{\pt_\cU}$, and $j^*$, $i^*$ are jointly conservative. This is convenient, as such recollements are equivalent to left-exact functors $\cU \xra{\varphi} \cZ$ between categories with finite limits.} Our notation is motivated by the following fundamental example: a closed-open decomposition
\[ \begin{tikzcd}
U
\arrow[hook]{r}{j}[swap]{{\sf open}}
&
X
\arrow[hookleftarrow]{r}{i}[swap]{{\sf closed}}
&
Z
\end{tikzcd} \]
among topological spaces determines a symmetric monoidal recollement \Cref{intro.smrec} in which $\cU = \Shv(U)$, $\cX = \Shv(X)$, and $\cZ = \Shv(Z)$. Similarly, a closed-open decomposition of a scheme determines a recollement of its category of quasicoherent sheaves.\footnote{Note that this has the opposite handedness of the situation with topological spaces.}

Let us fix a $[1]$-stratified symmetric monoidal category $\cX \coloneqq \lim^\lax(\cU \xra{\varphi} \cZ)$ and an object $x \coloneqq [u \mapsto \varphi(u) \xla{\alpha} z ] \in \cX$. Observe that for any object $w \in \cU$, we have a canonical morphism
\[
z \otimes \varphi(w)
\xlongra{\alpha}
\varphi(u) \otimes \varphi(w)
\xra{\can}
\varphi(u \otimes w)
\]
(where the map $\can$ comes from laxness of $\varphi$). We say that $x$ \bit{satisfies the projection formula at $w$} if this morphism is an equivalence.

\begin{mainthm}[\Cref{theorem.algebraic_descr}]
\label{theorem.intro.alg}
The following are equivalent.
\begin{enumerate}
    \item The object $x \in \cX$ is dualizable.
    \item The objects $u \in \cU$ and $z \in \cZ$ are dualizable, and $x$ satisfies the projection formula at every object $w \in \cU$.
    \item The objects $u \in \cU$ and $z \in \cZ$ are dualizable, and $x$ satisfies the projection formula at the objects $\uno_\cU \in \cU$ and $u^\vee \in \cU$.
\end{enumerate}
Moreover, if these conditions hold, then we have $x^\vee \simeq [u^\vee \mapsto \varphi(u^\vee) \xla{\beta} z^\vee]$ where $\beta$ is adjunct to the composite
\[
\uno_\cZ
\xra{\can}
\varphi(\uno_\cU)
\xra{\varphi(\eta)}
\varphi(u \otimes u^\vee)
\xlongla{\sim}
z \otimes \varphi(u^\vee)
\]
(using the laxness of $\varphi$, the duality data for $u$ and $u^\vee$, and the projection formula for $x$ at $u^\vee$).
\end{mainthm}

\noindent We give a somewhat more conceptual (but slightly weaker) reformulation of \Cref{theorem.intro.alg} as \Cref{remark.strictness.on.subcat.gend.by.u.and.udual}.

\subsubsection{Dualizable objects in stratified categories}
\label{subsubsection.intro.dzbl.objs.in.strat}

As our notation suggests, given a $[1]$-stratified symmetric monoidal category $\cU \xra{\varphi} \cZ$, its underlying symmetric monoidal category $\cX$ is its \bit{lax limit}: the terminal lax cone
\[ \begin{tikzcd}[column sep=0.5cm]
&
\cX
\arrow{ld}
\arrow{rd}
\\
\cU
\arrow{rr}[swap]{\varphi}[yshift=0.3cm]{\rotatebox{30}{$\Leftarrow$}}
&
&
\cZ
\end{tikzcd} \]
in the 2-category $\SMC^\lax$ of symmetric monoidal categories and laxly symmetric monoidal functors (a full sub-2-category of the 2-category of operads). We generalize this as follows: for any category $\cB$, we define a \bit{$\cB$-stratified symmetric monoidal category} to be a (unital) oplax functor 
\[
\cB
\xra[{\sf oplax}]{\cX_\bullet}
\SMC^\lax
~,\footnote{Note that an oplax functor out of $[1]$ is automatically strict, because $[1]$ has no nontrivial composites.}
\]
and we define its \bit{underlying symmetric monoidal category} to be its lax limit
\[
\cX
\coloneqq
\lim^\lax \left(
\cB
\xra[{\sf oplax}]{\cX_\bullet}
\SMC^\lax
\right)
~.\footnote{The standard terms ``lax'' and ``oplax'' that we use here are systematically replaced by the terms ``right-lax'' and ``left-lax'' in \cite{AMGR} (which follows the conventions of \cite{GR}). Since issues of handedness do not play a major role for us here, we employ the shorter terms.}^,\footnote{This definition strictly generalizes that introduced in \cite[\S 1.3]{AMGR} in multiple ways: we do not require $\cB$ to be a poset, we do not require presentability or stability, and we require the existence of fewer adjoints. (However, we are implicitly enforcing \textit{convergence}.)}
\]
We illustrate these notions shortly. However, we first note the clash in handedness: we are taking the \textit{lax} limit of an \textit{oplax} functor. This leads to substantial complexity, especially when accounting for symmetric monoidal structures. Nevertheless, it turns out that these notions are important ones: natural examples of stratified symmetric monoidal categories abound.

\begin{example}
\label{example.many.stratns}

\begin{enumerate}

\item[]

\item The category of $\pos$-constructible sheaves on a $\pos$-stratified topological space (or more generally a $\pos$-stratified topos, e.g.,  the \'etale topos of a $\pos$-stratified scheme \cite{BGH}) is symmetric monoidally $\pos$-stratified \cite[\S 1.8]{AMGR}.

\item\label{example.qcoh} The category of quasicoherent sheaves on a $\pos$-stratified scheme is symmetric monoidally $\pos$-stratified \cite[Examples 1.3.6 and 1.5.1]{AMGR}.

\end{enumerate}

\noindent (In the case that $\pos = [1]$, the above two examples specialize to the fact that a closed-open decomposition determines a symmetric monoidal recollement.)

\begin{enumerate}

\setcounter{enumi}{2}

\item For any compact Lie group $G$, the category $\Spectra^G$ of genuine $G$-spectra is symmetric monoidally stratified over the poset $\pos_G$ of closed subgroups of $G$ and subconjugacies \cite[\S 1.7]{AMGR}.

\item Any rigidly compactly generated presentably symmetric monoidal stable category (e.g., the category of spectra or genuine $G$-spectra) is symmetric monoidally stratified over its Balmer spectrum \cite[\S 1.6]{AMGR}.

\begin{enumerate}[label=(\alph*)]

\item In the case of quasicoherent sheaves over a scheme, we obtain an \textit{adelic stratification} over its Zariski poset -- a special case of \Cref{example.qcoh}.

\item In the case of $\Spectra$, we obtain a stratification organizing the fundamental objects of chromatic homotopy theory \cite[Example 4.6.14]{AMGR}.

\end{enumerate}

\item For any category $\cI$ admitting finite colimits and a terminal object and any presentable stable category $\cY$, the category $\Fun(\cI,\cY)$ is $(\NN^\leq)^\op$-stratified via Goodwillie--Taylor approximations, and this is a symmetric monoidal stratification if $\cY$ is presentably symmetric monoidal \cite[Example 1.3.11]{AMGR}.

\item Given any morphism $\cA \ra \cB$ of presentably symmetric monoidal stable categories, a $\pos$-stratification of $\cA$ determines a $\pos$-stratification of $\cB$ \cite[Remark 1.5.6]{AMGR}.\footnote{More generally, for any $\EE_n$-algebra $\cA$ in $\PrLSt$ for $1 \leq n \leq \infty$, an $\EE_n$-monoidal stratification of $\cA$ determines a stratification of any $\cA$-module in $\PrLSt$.}

% \item {\color{blue} \cite{BGH}, stratified etale topos}

\end{enumerate}
\end{example}

We now illustrate the above notions in the case that $\cB = [2]$. First of all, a $[2]$-stratified symmetric monoidal category consists of a diagram
\[ \begin{tikzcd}[column sep=0.5cm]
&
\cX_1
\arrow{rd}[sloped]{G}
\\
\cX_0
\arrow{rr}[yshift=0.3cm, xshift=-0.1cm]{\theta \Uparrow}[swap]{H}
\arrow{ru}[sloped]{F}
&
&
\cX_2
\end{tikzcd} \]
in $\SMC^\lax$. Its underlying symmetric monoidal category $\cX$ has objects given by tuples
\begin{equation}
\label{typical.object.of.brax.two.stratd.cat}
\left[
x_0
~,~
\begin{tikzcd}
x_1
\arrow{d}[swap]{\gamma_{01}}
\\
F(x_0)
\end{tikzcd}
~,~
\begin{tikzcd}
x_2
\arrow{r}{\gamma_{12}}
\arrow{d}[swap]{\gamma_{02}}
&
G(x_1)
\arrow{d}{G(\gamma_{01})}
\\
H(x_0)
\arrow{r}[swap]{\theta_{x_0}}
&
GF(x_0)
\end{tikzcd}
\right]
~,
\end{equation}
where $x_i \in \cX_i$ and the square commutes.\footnote{The commutative square in $\cX_2$ may be thought of as `higher gluing data', and may be denoted by $\gamma_{012}$. In general, an object of $\lim^\lax_\cB(\cX_\bullet)$ consists of a compatible family of functors $\sd(\cB)_{\max = b} \ra \cX_b$, where $\sd(\cB)$ denotes the \textit{subdivision} of $\cB$ (whose objects are conservative functors $[n] \ra \cB$ and whose morphisms are (necessarily injective) factorizations).} The symmetric monoidal structure of $\cX$ can be described informally as follows: the tensor product $\Cref{typical.object.of.brax.two.stratd.cat}
\otimes
\Cref{typical.object.of.brax.two.stratd.cat}'$ is the object
\[
% \hspace{-0.5cm}
%\Cref{typical.object.of.brax.two.stratd.cat}
%\otimes
%\Cref{typical.object.of.brax.two.stratd.cat}'
%\coloneqq
\left[
x_0 \otimes x_0',
\begin{tikzcd}
x_1 \otimes x_1'
\arrow{d}[swap]{\gamma_{01} \otimes \gamma_{01}'}
\\
F(x_0) \otimes F(x_0')
\arrow{d}[swap]{\can_F}
\\
F(x_0 \otimes x_0')
\end{tikzcd} ,
\begin{tikzcd}
x_2 \otimes x_2'
\arrow{r}{\gamma_{12} \otimes \gamma_{12}'}
\arrow{d}[swap]{\gamma_{02} \otimes \gamma_{02}'}
&
G(x_1) \otimes G(x_1')
\arrow{r}{\can_G}
\arrow{d}{G(\gamma_{01}) \otimes G(\gamma_{01}')}
&
G(x_1 \otimes x_1')
\arrow{d}{G( \gamma_{01} \otimes \gamma_{01}')}
\\
H(x_0) \otimes H(x_0')
\arrow{r}[swap]{\theta_{x_0} \otimes \theta_{x_0'}}
\arrow{d}[swap]{\can_H}
&
GF(x_0) \otimes GF(x_0')
\arrow{r}{\can_G}
\arrow{rd}[sloped, swap]{\can_{GF}}
&
G(F(x_0) \otimes F(x_0'))
\arrow{d}{G(\can_F)}
\\
H(x_0 \otimes x_0')
\arrow{rr}[swap]{\theta_{x_0 \otimes x_0'}}
&
&
GF(x_0 \otimes x_0')
\end{tikzcd}
\right]
\]
(using the laxness of all four of the data $F$, $G$, $H$, and $\theta$).

We have the following characterization of dualizable objects in a $\cB$-stratified symmetric monoidal category
\[
\cX
\simeq
\lim^\lax \left(
\cB
\xra[{\sf oplax}]{\cX_\bullet}
\SMC^\lax
\right)
~,
\]
which we find strikingly simple in comparison with the above complexity.

\begin{mainthm}[\Cref{thm:GeneralDualizability}]
\label{theorem.intro.alg2}
An object $x \in \cX$ is dualizable if and only if for every morphism $[1] \xra{f} \cB$ the restriction
\[
f^*(x)
\in
\lim^\lax \left( \cX_{f(0)} \longra \cX_{f(1)} \right)
\]
is dualizable.
\end{mainthm}

\noindent Of course, the power of \Cref{theorem.intro.alg2} lies in its combination with \Cref{theorem.intro.alg}. Namely, we find that $x \in \cX$ is dualizable if and only if the following two conditions are satisfied.
\begin{enumerate}

\item For every object $b \in \cB$, the restriction $b^*(x) \in \cX_b$ is dualizable.

\item For every nondegenerate morphism $[1] \xra{f} \cB$, the object 
\[
f^*(x)
\in
\lim^\lax \left( \cX_{f(0)} \longra \cX_{f(1)} \right)
\]
satisfies the projection formula at the objects $\uno_{\cX_{f(0)}} \in \cX_{f(0)}$ and $(f(0)^*(x))^\vee \in \cX_{f(0)}$.

\end{enumerate}

\subsection{The \texorpdfstring{$[1]$}{[1]}-stratified \texorpdfstring{$1$}{1}-dimensional bordism hypothesis}
\label{subsection.intro.bord}

Let us recall the ordinary 1-dimensional bordism hypothesis (following \cite{Baez-Dolan,TQFT}). Consider the ($1$-dimensional framed) bordism category $\Bord$. Its objects are compact signed $0$-manifolds, its morphisms are compact framed (equivalently, oriented) $1$-dimensional bordisms, and its symmetric monoidal structure is given by disjoint union. The positively-signed point $+ \in \Bord$ is a dualizable object. The \bit{bordism hypothesis} asserts that $+$ is the free symmetric monoidal category containing a dualizable object. That is, for any symmetric monoidal category $\cC \in \SMC$, evaluation determines an equivalence
\[
%\ev_+:
\hom_\SMC(\Bord,\cC)
\xra[\sim]{\ev_+}
\cC^\dzbl
\]
from the space of symmetric monoidal functors $\Bord \ra \cC$ to the space of dualizable objects in $\cC$. Note that this necessarily refers to \textit{strictly} symmetric monoidal functors, as (op)laxly symmetric monoidal functors will not generally preserve dualizable objects.

Our goal is as follows: given a $[1]$-stratified symmetric monoidal category $\cX \simeq \lim^\lax ( \cU \xra{\varphi} \cZ)$ (e.g., a symmetric monoidal recollement), we would like to corepresent the data of a dualizable object in $\cX$ within the context of $[1]$-stratified symmetric monoidal categories. More precisely, let $\Strat \coloneqq \Strat_{[1]}$ be the category whose objects are $[1]$-stratified symmetric monoidal categories and in which a morphism $\varphi \ra \varphi'$ is a commutative square
\[ \begin{tikzcd}
\cU
\arrow{r}{\varphi}
\arrow{d}[swap]{f}
&
\cZ
\arrow{d}{g}
\\
\cU'
\arrow{r}[swap]{\varphi'}
&
\cZ'
\end{tikzcd} \]
in $\SMC^\lax$ in which both $f$ and $g$ lie in $\SMC$ (i.e., they are strictly symmetric monoidal).\footnote{The optimality of this definition for our purposes is explained in \Cref{rmk.why.defined.Strat.as.such}. (In light of the fact that (op)laxly symmetric monoidal functors do not generally preserve dualizable objects, an obvious first desideratum is that the formation of lax limits defines a functor $\Strat \xra{\lim^\lax} \SMC$.)} Then, we seek an object $\BBord \coloneqq \BBord_{[1]} \in \Strat$ -- a \bit{$[1]$-stratified bordism category} -- along with a natural equivalence
\[
\hom_{\Strat}(\BBord , \varphi)
\xlongra{\sim}
\cX^\dzbl
~.
\]
Said differently, $\BBord$ should be the free $[1]$-stratified symmetric monoidal category containing a dualizable object.

We define
\[
\BBord
\coloneqq
\left(
\Bord
\xrightarrow{(\id_\Bord , \const_{\pt})}
\Bord \times \Fin
\right) \in \Strat
~.
\]
It is not hard to check that its underlying symmetric monoidal category is
\[
\lim^\lax(\BBord)
\simeq
\Ar(\Bord) \times \Fin
~,
\]
so an object is a pair of a bordism and a finite set. Thereafter, it is not hard to check that the object
$
\tau
\coloneqq
( \id_+ , \es )
\in
\lim^\lax(\BBord)
$
is dualizable, and hence is classified by a morphism
$
\Bord
\xra{\tau}
\lim^\lax(\BBord)
$ in $\SMC$.\footnote{In fact, it is not hard to check that $\tau$ factors as an equivalence $\Bord \xra[\sim]{\tau} \lim^\lax(\BBord)^\dzbl$ (as should be expected in light of \Cref{theorem.intro.bord}).} Using \Cref{theorem.intro.alg}, we prove the following result.

\begin{mainthm}[\Cref{thm.stratified.bordism}]
\label{theorem.intro.bord}
The stratified bordism category $\BBord \in \Strat$ is the free $[1]$-stratified symmetric monoidal category containing a dualizable object. That is, for any $[1]$-stratified symmetric monoidal category $\cX \simeq \lim^\lax ( \cU \xra{\varphi} \cZ)$, the composite
\[
\Hom_{\Strat}(\BBord,\varphi)
\xra{\lim^\lax}
\Hom_\SMC(\lim^\lax(\BBord),\cX)
\xlongra{\tau^*}
\Hom_\SMC(\Bord,\cX)
\simeq
\cX^\dzbl
\]
is an equivalence.
\end{mainthm}

\begin{remark}
Note that $\BBord \in \Strat$ itself does not underlie a symmetric monoidal recollement, as $\Bord$ does not admit finite limits. This is one of the reasons why we work in the more general context of $\Strat$.
\end{remark}

\begin{remark}
We would be interested to see a generalization of \Cref{theorem.intro.bord} to larger posets, as well as ($\cO$-)monoidal and higher-categorical versions.
\end{remark}

%We would also be interested to see a manifestly differentio-topological description of $\BBord$.

\begin{remark}
It is not hard to verify that $\Strat$ is presentable, and thereafter to formally deduce the existence of a left adjoint
\[ \begin{tikzcd}[column sep=1.5cm]
\SMC
\arrow[dashed, yshift=0.9ex]{r}{L}
\arrow[leftarrow, yshift=-0.9ex]{r}[yshift=-0.2ex]{\bot}[swap]{\lim^\lax}
&
\Strat
\end{tikzcd}~. \]
One can then interpret \Cref{theorem.intro.bord} as a concrete description of the object $L(\Bord) \in \Strat$.

In fact, the left adjoint $L$ admits an explicit formula: it carries a symmetric monoidal category $\cC \in \SMC$ to the $[1]$-stratified symmetric monoidal category given by the composite
\[
\cC
\xlongra{\eta}
\Env(\cC)
\xra{\{1\}}
\colim^\SMC \left(
\begin{tikzcd}
\Env(\cC)
\arrow{r}{\epsilon}
\arrow{d}[swap]{\{0\}}
&
\cC
\\
{[1] \odot \Env(\cC)}
\end{tikzcd}
\right)
~,
\]
where $\SMC^\lax \xra{\Env} \SMC$ denotes the \textit{symmetric monoidal envelope} functor (the left adjoint to the forgetful functor \cite[\S 2.2.4]{HA}) and $\Cat \times \SMC \xra{\odot} \SMC$ denotes the tensoring. The fact that our description of $L(\Bord)$ in \Cref{theorem.intro.bord} is substantially simpler than this general formula is due to our usage of \Cref{theorem.intro.alg}.
\end{remark}

%We do not address the monoidal version of \Cref{theorem.intro.bord}. We would be very interested to see an answer to this. It seems that the resulting universal object would be somewhat less geometric. Firstly, it would involve the free monoidal $\infty$-category on an object equipped with a right dual, which is \textit{not} the tangle $\infty$-category (because in the tangle $\infty$-category all objects have all duals, which is even stronger than asking for merely an object with both duals). Secondly, it would involve coproduct in the $\infty$-category ${\sf Alg}(\Cat)$ of monoidal $\infty$-categories which is, contrary to coproduct in $\CAlg(\Cat)$, quite complicated. 

\subsection{Outline}
\label{subsection.intro.outline}

%After explaining how the theory of recollements is subsumed by that of $[1]$-indexed lax limits, w

The main body of this paper is organized as follows.

\begin{itemize}

\item[\Cref{sec:1}:] We prove \Cref{theorem.intro.alg} after some brief recollections on recollements and dualizability. We also discuss a number of closely related ideas (e.g., the Beauville--Laszlo theorem).

\item[\Cref{sec:2}:] We prove \Cref{theorem.intro.alg2} in \Cref{sec:2.4}, after some preliminary work. Namely, in \Cref{sec:lax.limit} we discuss the notion of a $\cO$-monoidal stratified category, and in \S\S\ref{sec:2.2}-\ref{sec:2.3} we study some aspects of the theory of lax right Kan extensions.

\item[\Cref{section:strat}:] We prove \Cref{theorem.intro.bord} using \Cref{theorem.intro.alg}. We note here that \Cref{section:strat} can be read independently of \Cref{sec:2} (which is somewhat technical).

\item[\Cref{section.applications.and.examples}:] We give some examples and applications of Theorems \ref{theorem.intro.alg} and \ref{theorem.intro.alg2}.

\end{itemize}

%We then introduce the concept of a $\cO$-monoidal stratified category in \Cref{sec:lax.limit} and prove \Cref{theorem.intro.alg2} in \Cref{sec:2.4}; the inductive method behind its proof rests on a fragment of the theory of lax right Kan extensions, which we carefully explain in \Cref{sec:2.2}--\ref{sec:2.3}. We next apply \Cref{theorem.intro.alg} to prove our $[1]$-stratified version of the $1$-dimensional bordism hypothesis in \Cref{section:strat}; this may be read independently of \Cref{sec:2} apart from a few parts in \Cref{sec:lax.limit} that we recall at the outset. Finally, in \Cref{section.applications.and.examples} we give our applications and examples of \Cref{theorem.intro.alg} and \Cref{theorem.intro.alg2}.

% Finally, in \Cref{section.applications.and.examples} we give a number of applications and examples of \Cref{theorem.intro.alg} and \Cref{theorem.intro.alg2}, including to orientation theory, $G$-spectra, and cyclotomic spectra.

% {\color{blue} describe sections in turn. in particular, note that \S 2 and \S 3 are logically unrelated, although both rest on \S 1. actually, \S 3 uses a few basic notions from \S 2, but we've written \S 3 to be readable independently (since \S 2 is quite a bit more technical).}

\subsection{Notation and conventions}
\label{subsection.intro.notation}

We generally use the notation and terminology introduced by Lurie in his books \cite{HA,HTT}, with which we assume a passing familiarity.

%However, as indicated in \Cref{subsubsection.intro.dzbl.objs.in.strat} we write $\SMC^\lax$ for the 2-category of symmetric monoidal categories and laxly symmetric monoidal functors 

In general, we use the standard 2-categorical terms `lax' and `oplax': for instance, locally cocartesian (resp.\! cartesian) fibrations over a base category $\cB$ encode oplax (resp.\! lax) functors from $\cB$ to $\Cat$. These are respectively also called `right-lax' and 'left-lax' (e.g., in \cite{AMGR}), and we sometimes use this disambiguating terminology as well (e.g., in \Cref{sec:lax.limit} when discussing right-lax morphisms among left-lax diagrams). In fact, we adopt a fibrational perspective and \textit{define} left-lax functors from $\cB$ to $\SMC^\lax$ to be commutative monoids in $\locCocart^{\lax}_{\cB}$ (although we show that the two possible definitions are equivalent when $\cB = [1]$ in \Cref{prop:StratID}). Also, we use the term `laxly (symmetric) monoidal' instead of `lax (symmetric) monoidal' to avoid conflation with other notions of laxness.\footnote{Beware that a laxly symmetric monoidal functor $\cC^\otimes \ra \cD^\otimes$ is a morphism between cocartesian fibrations over $\Fin_*$, and is therefore an instance of an \textit{oplax} morphism.} We refer the reader to \cite{AMGR}*{Appendix~A} for a more thorough discussion of these notions.

%In $2$-category theory, there are two possible directions that have to be distinguished when working with lax concepts. For instance, locally cocartesian and locally cartesian fibrations over a base category $\cB$ encode differently handed lax functors from $\cB$ to $\Cat$; in \cite{AMGR}, these are called `left-lax' and `right-lax', respectively, and are commonly called `oplax' and `lax' elsewhere in the literature. We systematically abbreviate `right-lax' as `lax' throughout this paper (e.g., `lax limit' means `right-lax limit'), only adopting the disambiguating terminology of left vs. right in \Cref{sec:lax.limit} since we need to discuss right-lax morphisms of left-lax diagrams there. See also \cite{AMGR}*{Appendix~A}.

We use the word `strict' in place of the word `strong' (e.g., we may say that a functor is `strictly monoidal'). Of course, this always refers to the homotopy-coherent notion.

We write $(-)^\dagger$ to denote the passage from a morphism to its adjunct (i.e., the corresponding morphism via an (implicit) adjunction).

\subsection{Acknowledgments}
\label{subsection.intro.ack}

It is our pleasure to thank Clark Barwick, Tim Campion, Denis-Charles Cisinski, David Gepner, Peter Haine, Achim Krause, Thomas Nikolaus, and Tomer Schlank for helpful discussions related to this paper. All three authors gratefully acknowledge the superb working conditions provided by the Mathematical Sciences Research Institute (which is supported by NSF award 1440140), where they were in residence during the Spring 2020 semester. The first author was supported by the HSE University Basic Research Program. The third author was funded by the Deutsche Forschungsgemeinschaft (DFG, German Research Foundation) under Germany's
Excellence Strategy EXC 2044–390685587, Mathematics Münster: Dynamics–Geometry–Structure.

\section{Dualizable objects in recollements} \label{sec:1}

In this section we prove \Cref{theorem.intro.alg}, which provides an algebraic description of dualizable objects in a symmetric monoidal recollement. In fact, we prove a more general result characterizing left and right-dualizable objects in the lax limit of a laxly monoidal functor (\Cref{theorem.algebraic_descr}). This is accomplished in \Cref{subsection.dualizable.objects.in.recs}. Thereafter, we provide an extended discussion of it and related ideas in \Cref{subsection.complements.to.algdescr}.

\subsection{Dualizable objects in recollements}
\label{subsection.dualizable.objects.in.recs}

We begin by laying out our conventions for duality data.

\begin{definition} \label{def:RightDualizable} Let $\cC$ be a monoidal category, with unit object $\uno \coloneqq \uno_\cC$. An object $x \in \cC$ is \bit{right-dualizable} if there exist a \bit{right dual} object $x^\vee \in \cC$ along with structure morphisms
\[
x^\vee \otimes x
\xlongra{\epsilon}
\uno
\qquad
\text{and}
\qquad
\uno
\xlongra{\eta}
x \otimes x^\vee
\]
satisfying the triangle identities.\footnote{The duality data furnish natural equivalences
\[ \Map_\cC(a,b \otimes x^{\vee}) \simeq \Map_\cC(a \otimes x, b) \qquad \text{and} \qquad \Map_\cC(a, x \otimes b) \simeq \Map_\cC(x^{\vee} \otimes a, b) \]
for any $a,b \in \cC$.}
We refer to $x$ as the \bit{left dual} of $x^\vee$.
%An object $x \in \cC$ is \bit{left-dualizable} if it is right-dualizable when $\cC$ is equipped with the reversed monoidal structure.
\end{definition}

\begin{warning}
Given an endomorphism in a 2-category, a \textit{right} dual is a \textit{left} adjoint. This unfortunate state of affairs is ultimately a consequence of the historical mistake of composing functions (and thereafter, functors) from right to left.
\end{warning}

We now recall how the notion of a monoidal recollement is subsumed by that of a laxly monoidal functor.

\begin{definition} \label{def:MonoidalRecollement}
Let $\cO^\otimes$ be a reduced operad (i.e., an operad with contractible underlying category), and let $\cX$ be an $\cO$-monoidal category that admits finite limits. Let $\cU, \cZ \subseteq \cX$ be two full subcategories that together constitute a \textit{recollement} of $\cX$ (in the sense of \cite[\S A.8.1]{HA}): in other words, their inclusions $j_*$ and $i_*$ admit left adjoints
\[ \begin{tikzcd}[column sep=2cm]
\cU
\arrow[dashed, leftarrow]{r}[swap, yshift=-1.2ex]{\bot}{j^*}
\arrow[hook, bend right=50]{r}[swap]{j_*}
&
\cX
\arrow[hookleftarrow]{r}[yshift=1.2ex]{\bot}[swap]{i_*}
\arrow[dashed, bend left=50]{r}{i^*}
&
\cZ
\end{tikzcd} \]
that are left exact and jointly conservative, and moreover the composite $j^*i_*$ is constant at the terminal object of $\cU$. Then, we say the recollement is an \bit{$\cO$-monoidal recollement} if the localization functors $j_\ast j^\ast$ and $i_\ast i^\ast$ are compatible with the $\cO$-monoidal structure in the sense of \cite[Definition 2.2.1.6]{HA}. In this case, $\cU$ and $\cZ$ canonically inherit $\cO$-monoidal structures, in such a way that $j^\ast$ and $i^\ast$ are strictly $\cO$-monoidal while $j_\ast$ and $i_\ast$ are laxly $\cO$-monoidal \cite[Proposition 2.2.1.9]{HA}.
\end{definition}

In this paper, we will only be concerned with the case where $\cO$ is either the commutative or the associative operad.

\begin{observation} \label{obs:RecollementsAreLaxLimits}
Given a recollement as in \Cref{def:MonoidalRecollement}, we can reconstruct $\cX$ from the datum of the composite functor $\varphi \coloneqq i^* j_*$: namely, we have a canonical pullback square
\[ \begin{tikzcd}[column sep=2cm]
\cX
\arrow{r}{i^* \xra{\eta} i^* j_* j^*}
\arrow{d}[swap]{j^*}
&
\Ar(\cZ)
\arrow{d}{\ev_1}
\\
\cU
\arrow{r}[swap]{\varphi}
&
\cZ
\end{tikzcd} \]
\cite[Corollary 1.10]{QS}.
%$(\cU, \cZ)$ on $\cX$, we may reconstruct $\cX$ from the datum of the gluing functor $\varphi = i^\ast j_\ast: \cU \to \cZ$, in the sense that we have a canonical equivalence of categories \cite[Corollary 1.10]{QS}
%$$\cX \xrightarrow{\simeq} \lim^\lax ( \cU \xrightarrow{\varphi} \cZ ) = \Ar(\cZ) \times_{\ev_1,\cZ, \varphi} \cU.$$
Moreover, if the recollement is $\cO$-monoidal so that $\varphi$ is a laxly $\cO$-monoidal functor, then by the same reasoning as in \cite[Proposition 1.26]{QS} (which dealt with the case where $\cO$ is the commutative operad), this canonically upgrades to a pullback square of operads
\[ \begin{tikzcd}
\cX^\otimes \ar{r} \ar{d} & \Ar(\cZ)^\otimes \ar{d}{(\ev_1)^\otimes} \\
\cU^\otimes \ar{r}[swap]{\varphi^\otimes} & \cZ^\otimes
\end{tikzcd} \]
(where $\Ar(\cZ)$ is endowed with the pointwise $\cO$-monoidal structure). The datum of an $\cO$-monoidal recollement is thus equivalent to that of a left-exact laxly $\cO$-monoidal functor between $\cO$-monoidal categories that admit finite limits.

More generally, if $\cU \xra{\varphi} \cZ$ is any laxly $\cO$-monoidal functor between $\cO$-monoidal categories, then $\Ar(\cZ) \times_{\cZ} \cU$ comes endowed with its \emph{canonical} $\cO$-monoidal structure via the above pullback \cite[Definition 1.21]{QS}. %, using that $(\ev_1)^\otimes$ is a cocartesian fibration of operads over $\cO^\otimes$.
% Note also that although $\cU \times_{\cZ} \Ar(\cZ)$ no longer decomposes as a recollement, the projection $j^\ast$ to $\cU$ still admits a fully faithful right adjoint given by $j_\ast(u) = [u, \id_{\varphi(u)}]$.
\end{observation}

% \begin{example}
% Suppose $\cO = \Ass$ is the associative operad and let $\varphi: \cU \to \cZ$ be a laxly monoidal functor. Then we may unwind the definition of the monoidal structure on $\Ar(\cZ) \times_{\cZ} \cU$ at the level of objects as follows: suppose given objects $x_i = [u_i \mapsto \varphi(u_i) \xla{\alpha_i} z_i]$, $i = 0,1$. Then the tensor product $x_0 \otimes x_1$ is given by
% $$ [u_0 \otimes u_1 \mapsto \varphi(u_0 \otimes u_1) \xla{\can} \varphi(u_0) \otimes \varphi(u_1) \xla{\alpha_0 \otimes \alpha_1} z_0 \otimes z_1]. $$
% \end{example}

By \Cref{obs:RecollementsAreLaxLimits}, in order to understand left- or right-dualizable objects in a monoidal recollement, it suffices to understand those in the lax limit of a laxly monoidal functor. We now state the main theorem of this section, which specializes to \Cref{theorem.intro.alg}.

\begin{theorem}
\label{theorem.algebraic_descr}
Let $\cU \xra{\varphi} \cZ$ be a laxly monoidal functor of monoidal categories and let
\[
\cX \coloneqq
\lim^\lax ( \cU \xlongra{\varphi} \cZ ) \coloneqq \cU \times_\cZ \Ar(\cZ)
\]
be its lax limit. For any object $x=[u \mapsto \varphi(u) \xla{\alpha} z]$ of $\cX$, the following conditions are equivalent.
\begin{enumerate}
    \item The object $x \in \cX$ is right-dualizable.
    \item\label{contn.alg.descr.all.objects} The objects $u \in \cU$ and $z \in \cZ$ are right-dualizable, and for any $w \in \cU$ the composite morphism
    \[
    z \otimes \varphi(w) \xlongra{\alpha} \varphi(u) \otimes \varphi(w) \xrightarrow{\can} \varphi(u \otimes w)
    \]
    in $\cZ$ is an equivalence.
    \item\label{condtn.alg.descr.two.objects.only} The objects $u \in \cU$ and $z \in \cZ$ are right-dualizable, and the composite morphisms
    \[
    \gamma: z \otimes \varphi(\uno_{\cU}) \xlongra{\alpha} \varphi(u) \otimes \varphi(\uno_{\cU}) \xrightarrow{\can} \varphi(u)
    \]
    \[
    g: z \otimes \varphi(u^{\vee}) \xlongra{\alpha} \varphi(u) \otimes \varphi(u^{\vee}) \xrightarrow{\can} \varphi(u \otimes u^{\vee})
    \]
    in $\cZ$ are equivalences.
\end{enumerate}
Moreover, if $x$ is right-dualizable, then its right dual $x^{\vee}$ is given by $[u^{\vee} \mapsto \varphi(u^{\vee}) \xla{\beta} z^{\vee}]$, where $\beta$ is adjunct to the composite
$$\uno_{\cZ} \xrightarrow{\can} \varphi(\uno_{\cU}) \xrightarrow{\varphi(\eta)} \varphi(u \otimes u^{\vee}) \underset{\sim}{\xlongla{g}} z \otimes \varphi(u^{\vee})~.$$
\end{theorem}

\begin{proof} First, suppose that $x$ is right-dualizable. Since the functors $\cX \xrightarrow{j^*} \cU$ and $\cX \xrightarrow{i^*} \cZ$ are strictly monoidal, the objects $u = j^*(x)$ and $z = i^*(x)$ are right-dualizable. Now, for any $w \in \cU$, the composite map
\begin{equation}
\label{composite.map.in.X.for.projection.formula}
x \otimes j_{\ast}(w)
\longra
j_\ast j^\ast (x) \otimes j_\ast (w)
\xra{\can}
j_{\ast} ((j^{\ast} x) \otimes w)
\end{equation}
is an equivalence, because for any $t \in \cX$ we have the sequence of equivalences
\begin{align*}
\Map_{\cX}(t,x \otimes j_{\ast}(w)) &\simeq \Map_{\cX}(x^{\vee} \otimes t, j_{\ast} w) \\
 &\simeq \Map_{\cU}(j^{\ast} (x^{\vee} \otimes t), w) \\
 &\simeq \Map_{\cU}((j^{\ast} x)^{\vee} \otimes j^{\ast} t, w) \\
 &\simeq \Map_{\cU}(j^{\ast} t, (j^{\ast} x) \otimes w) \\
 &\simeq \Map_{\cX}(t, j_{\ast} ((j^{\ast} x) \otimes w))
 ~.
\end{align*}
Applying $i^{\ast}$ to the map \Cref{composite.map.in.X.for.projection.formula} yields the composite morphism of (2), which is therefore an equivalence. This proves the implication $(1) \Rightarrow (2)$.

The implication $(2) \Rightarrow (3)$ is trivial.

We prove the implication $(3) \Rightarrow (1)$ by explicitly constructing duality data $x$. Let
$$
\left(u^{\vee}, \: u^{\vee} \otimes u \xra{\epsilon_u} \uno_\cU, \: \uno_\cU \xra{\eta_u} u \otimes u^{\vee} \right)
\qquad
\text{and}
\qquad
\left( z^{\vee}, \: z^{\vee} \otimes z \xra{\epsilon_z} \uno_\cZ , \: \uno_\cZ \xra{\eta_z} z \otimes z^{\vee} \right)
$$
be duality data for $u$ and $z$. Let $\varphi(u \otimes u^{\vee}) \xra{f} z \otimes \varphi(u^{\vee})$ denote the inverse of the equivalence $g$, and let us define the morphism $z^{\vee} \xra{\beta} \varphi(u^{\vee})$ to be the adjunct of the composite
\[ 
\begin{tikzcd}
\beta^{\dagger}: \uno_\cZ \ar{r}{\can} & \varphi(\uno_\cU) \ar{r}{\varphi(\eta_u)} & \varphi(u \otimes u^{\vee}) \ar{r}{f} & z \otimes \varphi(u^{\vee}).
\end{tikzcd} 
\]
Let $x^{\vee} = [u^{\vee} \mapsto \varphi(u^{\vee}) \xla{\beta} z^{\vee}]$. We wish to construct an evaluation map $x^{\vee} \otimes x \xra{\epsilon_x} \uno_\cX$ that restricts to $\epsilon_u$ and $\epsilon_z$ on each component. To do this, we must exhibit a homotopy that makes the diagram
\[ \begin{tikzcd} z^{\vee} \otimes z \ar{r}{\beta \otimes \alpha} \ar{d}[swap]{\epsilon_z} & \varphi(u^{\vee}) \otimes \varphi(u) \ar{r}{\can} & \varphi(u^{\vee} \otimes u) \ar{d}{\varphi(\epsilon_u)} \\
\uno_\cZ \ar{rr}[swap]{\can} & & \varphi(\uno_\cU)
\end{tikzcd}
\]
commute. Equivalently, we must exhibit a homotopy between the two adjunct maps $z \to z \otimes \varphi(\uno_\cU)$. Invoking assumption (3), we may compare the two maps after postcomposition by the equivalence $z \otimes \varphi(\uno_\cU) \xra[\sim]{\gamma} \varphi(u)$. The map $\can \circ \epsilon_z$ is then adjunct to the map $\alpha: z \to z \otimes \varphi(\uno_\cU) \xra[\sim]{\gamma} \varphi(u)$. On the other hand, the map $\varphi(\epsilon_u) \circ \can \circ (\beta \otimes \alpha)$ is adjunct to the upper horizontal composite in the commutative diagram
\[ \begin{tikzcd}[column sep=1.5cm]
\uno \otimes z \ar{r}{\beta^{\dagger} \otimes \alpha } \ar{dd}{\can \otimes \alpha} \ar[bend right=60]{ddd}[swap]{\alpha} & z \otimes \varphi(u^{\vee}) \otimes \varphi(u) \ar{r}{\id \otimes \can} \ar{d}{\alpha \otimes \id \otimes \id} \ar[bend right=75]{dd}[sloped]{\sim} & z \otimes \varphi(u^{\vee} \otimes u) \ar{r}{\id \otimes \varphi(\epsilon_u)} \ar{d}{\alpha \otimes \id} & z \otimes \varphi(\uno) \ar{d}[swap]{\alpha \otimes \id} \ar[bend left=65]{dd}[sloped, swap]{\sim} \\
 & \varphi(u) \otimes \varphi(u^{\vee}) \otimes \varphi(u) \ar{r}{\id \otimes \can} \ar{d}{\can \otimes \id} & \varphi(u) \otimes \varphi(u^{\vee}\otimes u) \ar{r}{\id \otimes \varphi(\epsilon_u)} \ar{d}{\can} & \varphi(u) \otimes \varphi(\uno) \ar{d}[swap]{\can} \\
\varphi(\uno) \otimes \varphi(u) \ar{r}{\varphi(\eta_u) \otimes \id} \ar{d}{\can} & \varphi(u \otimes u^\vee) \otimes \varphi(u) \ar{r}{\can} & \varphi(u \otimes u^{\vee} \otimes u) \ar{r}{\varphi(\id \otimes \epsilon_u)} & \varphi(u) \\
\varphi(u) \ar{urr}[swap]{\varphi(\eta_u \otimes \id)} \ar[bend right=12]{urrr}[swap]{\id}
\end{tikzcd} \]
in $\cZ$, which then yields the desired homotopy and thus $\epsilon_x$.
% \footnote{\color{red} Grigory: maybe should add that the left square commutes by the construction of $\beta$}

Next, using the commutative diagram
\[ \begin{tikzcd}
\uno \ar{rr}{\can} \ar{d}[swap]{\eta_z} \ar{rd}[sloped]{\beta^\dagger} & & \varphi(\uno) \ar{d}{\varphi(\eta_u)} \\
z \otimes z^\vee \ar{r}{\id \otimes \beta} \ar{rd}[swap, sloped]{\alpha \otimes \beta} & z \otimes \varphi(u^\vee) \ar{r}{g}[swap]{\sim} \ar{d}{\alpha \otimes \id} & \varphi(u \otimes u^\vee) \\
& \varphi(u) \otimes \varphi(u^\vee) \ar{ru}[swap, sloped]{\can}
\end{tikzcd} \]
in $\cZ$ we define a coevaluation map $\uno \xra{\eta_x} x \otimes x^{\vee}$ that restricts to $\eta_u$ and $\eta_z$ on each component.

We now verify the triangle identity $\id_x \simeq (\id \otimes \epsilon_x) \circ (\eta_x \otimes \id)$. Since $(\id \otimes \epsilon_x) \circ (\eta_x \otimes \id)$ is homotopic to the identity on the components $u$ and $z$ by construction, it only remains to show that that pasting the two homotopies given above yields an outer square in the commutative diagram
\[ \begin{tikzcd}[column sep=6em]
\uno \otimes z \ar{r}{\can \otimes \alpha} \ar{d}{\eta \otimes \id} & \varphi(\uno) \otimes \varphi(u) \ar{r}{\can} \ar{d}{\varphi(\eta) \otimes \id} & \varphi(u) \ar{d}{\varphi(\eta \otimes \id)} \\
z \otimes z^\vee \otimes z \ar{r}{(\can \circ (\alpha \otimes \beta)) \otimes \alpha} \ar{dd}{\id \otimes \epsilon} \ar[bend right=15]{rd}[sloped, swap]{\alpha \otimes (\can \circ (\beta \otimes \alpha))} & \varphi(u \otimes u^{\vee}) \otimes \varphi(u) \ar{r}{\can} & \varphi(u \otimes u^{\vee} \otimes u) \ar{dd}{\varphi(\id \otimes \epsilon)} \\
 & \varphi(u) \otimes \varphi(u^\vee \otimes u) \ar{d}{\id \otimes \varphi(\epsilon)} \ar[bend right=15]{ur}[sloped, swap]{\can} & \\
z \otimes \uno \ar{r}{\alpha \otimes \can} & \varphi(u) \otimes \varphi(\uno) \ar{r}{\can} & \varphi(u)
\end{tikzcd} \]
in $\cZ$ that is homotopic to the identity on $\alpha$. A chase of the definitions shows that the homotopy witnessing the commutativity of the upper rectangle is itself homotopic to
\[ \begin{tikzcd}[column sep=4em]
% & z \otimes \varphi(u^{\vee}) \otimes \varphi(u)
z \ar{r}{\alpha} \ar{d}[swap]{\eta \otimes \id} & \varphi(u) \ar{rd}[sloped]{\beta^{\dagger} \otimes \id} \ar[bend left=10]{rrd}[sloped]{\varphi(\eta \otimes \id)}  \ar{d}[swap]{\eta \otimes \beta^{\dagger} \otimes \id} \\
z \otimes z^{\vee} \otimes z \ar{r}[swap]{\id \otimes \beta^{\dagger} \otimes \alpha} & z \otimes z^{\vee} \otimes z \otimes \varphi(u^{\vee}) \otimes \varphi(u) \ar{r}[swap]{\id \otimes \epsilon \otimes \id} & z \otimes \varphi(u^{\vee}) \otimes \varphi(u) \ar{r}[swap]{\can \circ (g \otimes \id)} & \varphi(u \otimes u^{\vee} \otimes u)
\end{tikzcd}~. \]
In particular, the factorization through $\id \otimes \beta \otimes \alpha$ permits us to reduce the assertion to checking that in the commutative diagram
\[ \begin{tikzcd}
z \ar{r}{\id \otimes \can} \ar{d}[swap]{\eta \otimes \id} & z \otimes \varphi(\uno) \ar{r}{\gamma}[swap]{\sim} \ar{d}{\eta \otimes \id} & \varphi(u) \ar{d}[swap]{\eta \otimes \beta^\dagger \otimes \id} \ar{rd}[sloped]{\beta^\dagger \otimes \id} \\
z \otimes z^{\vee} \otimes z \ar{r}[swap, outer sep=2pt]{\id \otimes \can} \ar{ddd}[swap]{\id \otimes \epsilon} & z \otimes z^{\vee} \otimes z \otimes \varphi(\uno) \ar{r}[swap, outer sep=2pt]{\id \otimes \beta^\dagger \otimes \gamma} \ar[bend right=30]{rdd}[swap, sloped]{\id} & z \otimes z^{\vee} \otimes z \otimes \varphi(u^{\vee}) \otimes \varphi(u) \ar{r}[swap, outer sep=2pt]{\id \otimes \epsilon \otimes \id} \ar{d}[swap]{\id \otimes \can} & z \otimes \varphi(u^{\vee}) \otimes \varphi(u) \ar{d}{\id \otimes \can} \\
& & z \otimes z^{\vee} \otimes z \otimes \varphi(u^{\vee} \otimes u) \ar{r}[swap]{\id \otimes \epsilon \otimes \id} \ar{d}[swap]{\id \otimes \varphi(\epsilon)} & z \otimes \varphi(u^{\vee} \otimes u) \ar{d}{\id \otimes \varphi(\epsilon)} \\
 & & z \otimes z^{\vee} \otimes z \otimes \varphi(\uno) \ar{d}[swap]{\id \otimes \epsilon \otimes \id} \ar{r}[swap]{\id \otimes \epsilon \otimes \id} & z \otimes \varphi(\uno) \ar{d}{\gamma}[sloped, anchor=north]{\sim} \\
z \ar{rr}[swap]{\id \otimes \can} & & z \otimes \varphi(\uno) \ar{r}[swap]{\gamma}{\sim} & \varphi(u)
\end{tikzcd} \]
the homotopy witnessing the commutativity of the outer square is homotopic to the identity on $\alpha$ (note the commutativity of the inner triangle derives from the diagram that defines $\epsilon_x$). But since $\gamma$ is an equivalence, this holds in view of the homotopy between the left inner diagram and the identity on $\id \otimes \can$ furnished by the triangle identity $\id_z \simeq (\id \otimes \epsilon_z) \circ (\eta_z \otimes \id)$.

Finally, the other triangle identity $\id_{x^{\vee}} \simeq (\epsilon_x \otimes \id) \circ (\id \otimes \eta_x)$ may be established by similar reasoning, the details of which we leave to the reader.
\end{proof}

\subsection{Complements to \texorpdfstring{\Cref{theorem.algebraic_descr}}{Theorem~\ref{theorem.algebraic_descr}}}
\label{subsection.complements.to.algdescr}

In this subsection, we collect a few remarks that serve to illustrate \Cref{theorem.algebraic_descr} and connect it with other parts of the literature.

\begin{remark} \label{rem:inducedRecollementOnModules}
Suppose that $\cU \xra{\varphi} \cZ$ is a laxly symmetric monoidal functor with lax limit $\cX \coloneqq \lim^\lax ( \cU \xra{\varphi} \cZ)$, and suppose that $A = [B \mapsto \varphi(B) \xla{\alpha} C] \in \CAlg(\cX)$ is a commutative algebra object. Then, we have an identification
\begin{equation}
\label{identify.A.modules.in.X.as.a.lax.limit}
\Mod_A(\cX)
\simeq
\lim^\lax \left( \Mod_B(\cU) \xlongra{\tilde{\varphi}} \Mod_C(\cZ) \right)
~,
\end{equation}
where $\tilde{\varphi}$ denotes the composite
\[
\Mod_B(\cU)
\xlongra{\varphi}
\Mod_{\varphi(B)}(\cZ)
\xlongra{\alpha^*}
\Mod_C(\cZ)
\]
(using the fact that $\varphi$ is laxly symmetric monoidal).\footnote{To see this, observe that the pullback square
\[ \begin{tikzcd}[ampersand replacement=\&]
(\cX^\otimes, A) \ar{r} \ar{d} \& (\Ar(\cZ)^\otimes, \alpha) \ar{d}{(\ev_1)^\otimes} \\
(\cU^\otimes, B) \ar{r}[swap]{\varphi^\otimes} \& (\cZ^\otimes, \varphi(B))
\end{tikzcd} \]
of operads equipped with a distinguished commutative algebra object is sent to a pullback square of operads under the functor $(\cC^\otimes , R) \mapsto \Mod_R(\cC)^\otimes$, and combine this with the pullback square
\[ \begin{tikzcd}[ampersand replacement=\&]
\Mod_\alpha(\Ar(\cZ))^\otimes
\arrow{r}
\arrow{d}[swap]{(\ev_1)^\otimes}
\&
\Ar(\Mod_C(\cZ))^\otimes
\arrow{d}{(\ev_1)^\otimes}
\\
\Mod_{\varphi(B)}(\cZ)^\otimes
\arrow{r}[swap]{(\alpha^*)^\otimes}
\&
\Mod_C(\cZ)^\otimes
\end{tikzcd} \]
of operads.} Let us suppose that the relative tensor products exist making $\Mod_B(\cU)$ and $\Mod_C(\cZ)$ (and hence also $\Mod_A(\cX)$) into symmetric monoidal categories. Then, $\tilde{\varphi}$ is also laxly symmetric monoidal and the equivalence \Cref{identify.A.modules.in.X.as.a.lax.limit} is symmetric monoidal. Applying \Cref{theorem.algebraic_descr}, the projection formula takes the following form: for a dualizable object $M = [N \mapsto \varphi(N) \xla{\beta} P] \in \Mod_A(\cX)$ and any $W \in \Mod_B(\cU)$, the canonical composite morphism
\[
P \otimes_C \varphi(W)
\xlongra{\beta}
\varphi(N) \otimes_C \varphi(W)
\xlongra{\can}
\varphi(N \otimes_B W)
\]
is an equivalence.
\end{remark}

% and let $A = [B \mapsto \varphi(B) \xla{\alpha} C] \in \cX$ be a commutative algebra. Then $\Mod_{A}(\cX)$ is the lax limit of the functor $\widetilde{\varphi}: \Mod_{B}(\cU) \to \Mod_{C}(\cZ)$ given by $\varphi$ at the level of underlying objects, which for a $B$-module $N$ regards $\varphi(N)$ as a $C$-module via restriction along $\alpha$. Indeed, the pullback square

% of operads with a distinguished algebra object is sent to a pullback square of operads under the functor $(\cC^\otimes, R) \mapsto \Mod_{R}(\cC)^\otimes$, and this together with the pullback of operads
% $$\Mod_{\alpha}(\Ar(\cZ))^\otimes \simeq \Mod_{\varphi(B)}(\cZ)^\otimes \times_{(\alpha^*)^\otimes, \Mod_{C}(\cZ)^\otimes, (\ev_1)^\otimes} \Ar(\Mod_C(\cZ))^\otimes$$
% implies the claim. Now, suppose that the relative tensor products exist to make $\Mod_{B}(\cU)$ and $\Mod_{C}(\cZ)$ (and hence also $\Mod_{A}(\cX)$) into symmetric monoidal categories. Then $\varphi$ is laxly symmetric monoidal and the canonical symmetric monoidal structure on $\Mod_{A}(\cX)$ agrees with that given by $(M,M') \mapsto M \otimes_A M'$.\footnote{More generally, we could suppose that $A$ is a $\EE_2$-algebra and work with $\LMod_{A}(\cX)$ as a monoidal category.} Applying \Cref{theorem.algebraic_descr} to $\Mod_{A}(\cX)$, we see that the projection formula takes on the following form: for a dualizable object $M = [N \mapsto \varphi(N) \xla{\alpha} P] \in \Mod_A(\cX)$ and $W \in \Mod_B(\cU)$, we have that
% $$ P \otimes_C \varphi(W) \xra{\simeq} \varphi(N \otimes_B W).$$

\begin{remark}
\label{rmk.strictly.s.m.case}
If the functor $\cU \xra{\varphi} \cZ$ is strictly monoidal, then \Cref{theorem.algebraic_descr} shows that an object $[u \mapsto \varphi(u) \xla{\alpha} z] \in \cX$ is right-dualizable if and only if $u$ is right-dualizable and $\alpha$ is an equivalence (which implies that $z$ is right-dualizable). Said differently, in this case the functor
\begin{equation}
\label{j.upper.star.on.dzbls}
\cU^\dzbl
\xlongla{j^*}
\cX^\dzbl
\end{equation}
is an equivalence, which generalizes the easy observation that $\Ar(\cZ)^\dzbl \xra{\ev_1} \cZ^{\dzbl}$ is an equivalence.\footnote{Indeed, a more direct way to see this is to use monoidal equivalences $\cX \simeq \cU \times_{\cZ} \Ar(\cZ)$ and $\cZ^{\dzbl} \simeq \Ar(\cZ)^{\dzbl}$.}
% (In this special case, the fiber product may be taken in $\CAlg \subseteq \CAlg^\lax$.)

On the other hand, in general the fully faithful inclusion
\[
\lim^\lax(\cU^\dzbl \longhookra \cU \xlongra{\varphi} \cZ)
\longhookra
\lim^\lax(\cU \xlongra{\varphi} \cZ)
\]
clearly restricts to an equivalence on dualizable objects. So in fact, the functor \Cref{j.upper.star.on.dzbls} is an equivalence as soon as the composite $\cU^\dzbl \hookra \cU \xra{\varphi} \cZ$ is strictly monoidal. See \Cref{example.modules.over.complete.local.cring} for an instance of this.
\end{remark}

\begin{remark}
\label{remark.strictness.on.subcat.gend.by.u.and.udual}
Here is a more conceptual (but slightly weaker) reformulation of \Cref{theorem.algebraic_descr} in the symmetric monoidal case.

First of all, note that the laxly symmetric monoidal functor $\cU \xra{\varphi} \cZ$ admits a canonical factorization
\[ \begin{tikzcd}
&
\Mod_{\varphi(\uno_\cU)}(\cZ)
\arrow{d}{\fgt}
\\
\cU
\arrow{r}[swap]{\varphi}
\arrow[dashed]{ru}[sloped]{\varphi}
&
\cZ
\end{tikzcd}
\]
through a strictly unital laxly symmetric monoidal functor.\footnote{If $\cZ$ does not admit all geometric realizations then $\Mod_{\varphi(\uno_\cU)}(\cZ)$ may not be a symmetric monoidal category, but we may still consider it as an operad.}

We claim that an object $[u \mapsto \varphi(u) \xla{\alpha} z] \in \cX$ is dualizable if and only if both $u \in \cU$ and $z \in \cZ$ are dualizable and moreover the following conditions hold.

\begin{enumerate}[label=(\alph*)]

\item\label{condn.4a} The composite functor
\begin{equation}
\label{composite.functor.from.Bord.to.modules.in.Z}
\Bord
\xlongra{u}
\cU
\xlongra{\varphi}
\Mod_{\varphi(\uno_\cU)}(\cZ)
\end{equation}
is strictly symmetric monoidal.\footnote{Writing $\langle u,u^\vee \rangle \subseteq \cU$ for the full symmetric monoidal subcategory containing $u$ and $u^\vee$, this is equivalent to the condition that the composite $\langle u,u^\vee \rangle \hookra \cU \xra{\varphi} \Mod_{\varphi(\uno_\cU)}(\cZ)$ is strictly symmetric monoidal.}

\item\label{condn.4b} The adjunct $z \otimes \varphi(\uno_\cU) \xra{\alpha^\dagger} \varphi(u)$ in $\Mod_{\varphi(\uno_\cU)}(\cZ)$ of the map $z \xra{\alpha} \varphi(u)$ in $\cZ$ is an equivalence.

\end{enumerate}
On the one hand, these conditions imply condition \Cref{condtn.alg.descr.two.objects.only} of \Cref{theorem.algebraic_descr}: the morphism $\alpha^\dagger$ in condition \Cref{condn.4b} is precisely the morphism $\gamma$, and thereafter condition \Cref{condn.4a} implies that the morphism $g$ is an equivalence. On the other hand, condition \Cref{contn.alg.descr.all.objects} of \Cref{theorem.algebraic_descr} evidently implies condition \Cref{condn.4a}.

Informally, this may be summarized as saying that dualizable objects of $\cX$ are all given by the following procedure: choose a dualizable object $u \in \cU$ such that the composite \Cref{composite.functor.from.Bord.to.modules.in.Z} is strictly symmetric monoidal, and then choose a dualizable object $z \in \cZ$ and a $\varphi(\uno_\cU)$-linear equivalence $z \otimes \varphi(\uno_\cU) \xra{\sim} \varphi(u)$.

We can articulate this alternatively as follows. Let us write
\[ \begin{tikzcd}
\cX^\dzbl
\arrow{rr}{j^*}
\arrow[dashed, two heads]{rd}
&
&
\cU^\dzbl
\\
&
\cU^\dzbl_0
\arrow[hook]{ru}
\end{tikzcd} \]
for the image, and let us write $\cU^\dzbl_1 \subseteq \cU^\dzbl$ for the full subcategory on those dualizable objects $u \in \cU$ such that the resulting composite morphism \Cref{composite.functor.from.Bord.to.modules.in.Z} is strictly symmetric monoidal. Then, we have a commutative diagram
\[ \begin{tikzcd}[row sep=1.5cm, column sep=2cm]
\cX^\dzbl
\arrow{r}{j^*}
\arrow{d}[swap]{i^*}
&
\cU^\dzbl_0
\arrow{d}{\varphi}
\arrow[hook]{r}
&
\cU^\dzbl_1
\arrow{d}{\varphi}
\\
\cZ^\dzbl
\arrow{r}[swap]{(-) \otimes \varphi(\uno_\cU)}
&
\Mod_{\varphi(\uno_\cU)}(\cZ)^\dzbl
\arrow{r}[swap]{\id}
&
\Mod_{\varphi(\uno_\cU)}(\cZ)^\dzbl
\end{tikzcd} \]
in which both the left square and the composite rectangle are pullbacks, and moreover the left square consists entirely of strictly symmetric monoidal functors.\footnote{One can see that the middle vertical functor is strictly symmetric monoidal as follows. Given a pair of objects $u,v \in \cU^\dzbl_0$, let $x = [u \mapsto \varphi(u) \xla{\alpha} z] \in \cX^\dzbl$ be a lift of $u$. Then, the projection formula for $x$ at $v$ guarantees that the composite
\[
z \otimes \varphi(v)
\xlongra{\alpha}
\varphi(u) \otimes \varphi(v)
\xra{\can}
\varphi(u \otimes v)
\]
is an equivalence. On the other hand, applying condition \Cref{condn.4b} we can rewrite the source as
\[
z \otimes \varphi(v) \simeq
(z \otimes \varphi(\uno_\cU)) \otimes_{\varphi(\uno_\cU)} \varphi(v)
\simeq
\varphi(u) \otimes_{\varphi(\uno_\cU)} \varphi(v)
~.
\]
}
\end{remark}

\begin{remark}
\Cref{remark.strictness.on.subcat.gend.by.u.and.udual} has the following curious consequence. Let us take $\cU = \Bord$ and $u \coloneqq + \in \Bord$ to be the universal dualizable object. Then, a laxly symmetric monoidal functor $\Bord \xra{\varphi} \cZ$ is strict if and only if the following conditions hold.
\begin{enumerate}[(i)]

\item The object $z \coloneqq \varphi(+) \in \cZ$ is dualizable.

\item The laxly symmetric monoidal functor $\varphi$ is unital, i.e., the canonical morphism $\uno_\cZ \ra \varphi(\es)$ is an equivalence.

\item The laxly symmetric monoidal functor $\varphi$ is strict on the pair $(+,-) \in \Bord^{\times 2}$, i.e., the canonical morphism $\varphi(+) \otimes \varphi(-) \ra \varphi(+ \sqcup -)$ is an equivalence.

\end{enumerate}
Indeed, these are equivalent to condition \Cref{condtn.alg.descr.two.objects.only} of \Cref{theorem.algebraic_descr}, and our conclusion is condition \Cref{condn.4a} of \Cref{remark.strictness.on.subcat.gend.by.u.and.udual}.
\end{remark}

\begin{remark}[Beauville--Laszlo] \label{rem:BeauvilleLaszlo}
Suppose for simplicity that $\varphi$ is a laxly symmetric monoidal functor, so that $j_*(\uno_{\cU})$ and $\varphi(\uno_{\cU})$ are commutative algebras, and suppose moreover that the relative tensor products exist to make $\Mod_{j_* (\uno_\cU)}(\cX)$ and $\Mod_{\varphi(\uno_\cU)}(\cZ)$ symmetric monoidal categories. We then have a pullback square
\begin{equation}
\label{before.dzbl.in.beauville.laszlo}
\begin{tikzcd}
\cX \ar{r}{i^*} \ar{d}[swap]{- \otimes j_*(\uno_{\cU})} & \cZ \ar{d}{- \otimes \varphi(\uno_\cU)} \\
\Mod_{j_* (\uno_\cU)}(\cX) \ar{r}[swap]{i^*} & \Mod_{\varphi(\uno_\cU)}(\cZ)
\end{tikzcd}
\end{equation}
in $\SMC$. Indeed, it is not hard to check that we have an equivalence
\begin{equation}
\label{recollement.of.Mod.j.lower.star.uno.U.in.X}
\Mod_{j_*(\uno_\cU)}(\cX) \simeq \lim^\lax ( \cU \xlongra{\varphi} \Mod_{\varphi(\uno_\cU)}(\cZ))
\end{equation}
in $\SMC$, and thereafter we may define an inverse functor
\[
\Mod_{j_* (\uno_\cU)}(\cX) \times_{\Mod_{\varphi(\uno_\cU)}(\cZ)} \cZ \longra \cX
\]
by the formula
$$([u \mapsto \varphi(u) \xla{\alpha'} z'], \, z', \, z' \xla{\sim}  z \otimes \varphi(\uno_\cU)) \: \longmapsto \: [u \mapsto \varphi(u) \xla{\alpha'} z' \xla{\sim} z \otimes \varphi(\uno_{\cU}) \xla{\eta} z]~.$$
This result and variants thereof are sometimes known as the \textit{Beauville--Laszlo theorem} (after \cite{BL}; e.g., see \cite[Proposition~7.4.1.1]{SAG}).

Let us pass to dualizable objects in the pullback square \Cref{before.dzbl.in.beauville.laszlo}, to obtain a pullback square
\begin{equation}
\label{beauville.laszlo.pullback.square}
\begin{tikzcd}
\cX^\dzbl
\ar{r}{i^*} \ar{d}[swap]{- \otimes j_*(\uno_{\cU})} & \cZ^\dzbl \ar{d}{- \otimes \varphi(\uno_\cU)} \\
\Mod_{j_* (\uno_\cU)}(\cX)^\dzbl \ar{r}[swap]{i^*} & \Mod_{\varphi(\uno_\cU)}(\cZ)^\dzbl
\end{tikzcd}~.
\end{equation}
Note that this description of the dualizable objects in $\cX$ is a priori very different from that conferred by \Cref{theorem.algebraic_descr}: here the connection with dualizable objects in $\cU$ is not apparent, and indeed this pullback square yields no information if $\varphi(\uno_{\cU}) \simeq \uno_{\cZ}$ (which implies that $j_*(\uno_\cU) \simeq \uno_\cX$). Nonetheless, we may relate it to the pullback squares of \Cref{remark.strictness.on.subcat.gend.by.u.and.udual} in the following manner. We always have a factorization of the adjunction $j^* \dashv j_*$ as
\[ \begin{tikzcd}[column sep=2cm]
j^*
:
\cX
\arrow[yshift=0.9ex]{r}{- \otimes j_*(\uno_\cU)}
\arrow[leftarrow, yshift=-0.9ex]{r}[yshift=-0.2ex]{\bot}[swap]{\fgt}
&
\Mod_{j_*(\uno_\cU)}(\cX)
\arrow[yshift=0.9ex]{r}{\overline{j}^*}
\arrow[leftarrow, yshift=-0.9ex]{r}[yshift=-0.2ex]{\bot}[swap]{\overline{j}_*}
&
\cU
:
j_*
\end{tikzcd} \]
with both left adjoints strictly symmetric monoidal. Furthermore, note that if $x \in \cX^{\dzbl}$, then the natural map $x \otimes j_* (\uno_{\cU}) \to j_* j^* (x)$ is an equivalence (as may be checked on representable functors), and similarly if $x \in \Mod_{j_*(\uno_\cU)}(\cX)^{\dzbl}$, then $x \ra \overline{j}_* \overline{j}^* (x)$ is an equivalence (since now $\overline{j}_*(\uno_{\cU}) \in \Mod_{j_*(\uno_\cU)}(\cX)$ is the unit object). Therefore, upon passage to dualizable objects, the functor $\overline{j}^*$ is fully faithful. Thereafter, due to the commutative triangle
\[ \begin{tikzcd}
\Mod_{j_*(\uno_\cU)}(\cX)^{\dzbl}
\arrow[hook]{r}{\overline{j}^*}
\arrow{rd}[sloped, swap]{i^*}
&
\cU^\dzbl
\arrow{d}{\varphi}
\\
&
\Mod_{\varphi(\uno_\cU)}(\cZ)
\end{tikzcd}~,
\]
we see that objects of $\Mod_{j_*(\uno_\cU)}(\cX)^{\dzbl}$ satisfy condition \Cref{condn.4a} in \Cref{remark.strictness.on.subcat.gend.by.u.and.udual}, so we obtain an inclusion $\Mod_{j_*(\uno_\cU)}(\cX)^{\dzbl} \subseteq \cU_1^{\dzbl}$ (using the notation of \Cref{remark.strictness.on.subcat.gend.by.u.and.udual}). On the other hand, the factorization shows that $\cU_0^{\dzbl} \subseteq \Mod_{j_*(\uno_\cU)}(\cX)^{\dzbl}$. We have thus shown that the Beauville--Laszlo pullback square \Cref{beauville.laszlo.pullback.square} fits in between the two pullback squares of \Cref{remark.strictness.on.subcat.gend.by.u.and.udual}:
\[ \begin{tikzcd}
\cX^\dzbl \arrow{r} \arrow{d} & \cU^\dzbl_0 \arrow{d} \arrow[hook]{r} & \Mod_{j_*(\uno_\cU)}(\cX)^{\dzbl} \arrow[hook]{r} \arrow{d} & \cU^\dzbl_1 \arrow{d} \\
\cZ^\dzbl \arrow{r} & \Mod_{\varphi(\uno_\cU)}(\cZ)^\dzbl \arrow{r}[swap]{\id} &
\Mod_{\varphi(\uno_\cU)}(\cZ)^\dzbl \arrow{r}[swap]{\id} & \Mod_{\varphi(\uno_\cU)}(\cZ)^\dzbl \:.
\end{tikzcd} \]

Suppose we additionally assume that $\cU \xra{\varphi} \cZ$ is an exact accessible functor of stable presentable categories (ensuring that $\cX$ is stable and presentably symmetric monoidal) and that $\cX$ is rigidly compactly generated (i.e., $\cX^{\omega} = \cX^{\dzbl}$ and $\cX \simeq \mathrm{Ind}(\cX^{\omega})$). Then, $\Mod_{j_*(\uno_\cU)}(\cX)^{\dzbl}$ is the thick subcategory generated by the essential image of $(- \otimes j_*(\uno_\cU))|_{\cX^{\dzbl}}$. Consequently, in this case if we want a pullback square of symmetric monoidal idempotent complete stable categories, the Beauville--Laszlo square is the `minimal' such option.
\end{remark}

% alt method: Using this and the projection formula $j_!(u) \otimes x \simeq j_!(u \otimes j^* x)$,
\begin{remark}
We digress to explain the relationship between the Beauville--Laszlo pullback square and similar results of Achim Krause in \cite{Krause}. Let us suppose that the hypotheses in the last paragraph of \Cref{rem:BeauvilleLaszlo} are in effect. First note that if $u \in \cU$ is compact, then $u$ is dualizable: indeed, since $j_!: \cU \to \cX$ preserves compact objects and $\cX^{\omega} = \cX^{\dzbl}$, we get that $j_!(u)$ is dualizable and hence $j^* j_!(u) \simeq u$ is dualizable. Note as well that since $j_!(u) = [u \mapsto \varphi(u) \leftarrow 0]$, it follows from the projection formula of \Cref{theorem.algebraic_descr} that if $u$ is compact, then $\varphi(u) = 0$ and hence the natural map $j_!(u) \to j_*(u)$ is an equivalence. We deduce that for all $x \in \cX$ and $u \in \cU^{\omega}$, both $j^*: \Map(j_! u, x) \to \Map(u, j^* x)$ and $j^*: \Map(x, j_! u) \to \Map(j^* x, u)$ are equivalences. The hypotheses of \cite[Lemma~3.9]{Krause} are thus satisfied, so that we get a pullback square of stable categories
\[ \begin{tikzcd}
\cX^{\omega} \ar{r} \ar{d} & \cX^{\omega} / \cU^{\omega} \ar{d} \\
\cU^{\dzbl} \ar{r} & \cU^{\dzbl} / \cU^{\omega}
\end{tikzcd} \]
where we form Verdier quotients in the righthand column. Moreover, it is easily checked that $\cU^{\omega} \subseteq \cX^{\dzbl}$ and $\cU^{\omega} \subseteq \cU^{\dzbl}$ are thick tensor ideals, so this is in fact a pullback square of stable symmetric monoidal categories.

We next bring the bottom row of the Beauville--Laszlo pullback square \Cref{beauville.laszlo.pullback.square} into the picture.  By the same reasoning now applied to the recollement $(\cU, \Mod_{\varphi(\uno_\cU)}(\cZ))$ on $\Mod_{j_*(\uno_\cU)}(\cX)$ (recall the equivalence \Cref{recollement.of.Mod.j.lower.star.uno.U.in.X}), we obtain the pullback square of stable symmetric monoidal categories
\[ \begin{tikzcd}
\Mod_{j_*(\uno_\cU)}(\cX)^{\omega} \ar{r} \ar[hook]{d} & \Mod_{j_*(\uno_\cU)}(\cX)^{\omega} / \cU^{\omega} \ar[hook]{d} \\
\cU^{\dzbl} \ar{r} & \cU^{\dzbl} / \cU^{\omega}
\end{tikzcd} \]
where the left vertical functor is fully faithful by \Cref{rem:BeauvilleLaszlo} and the right vertical functor is then fully faithful using the formula for mapping spaces in a Verdier quotient \cite[Lemma~3.3]{Krause}. One can sometimes show that the right vertical functor is an equivalence, at least upon idempotent completion; for example, the description of the $\ZZ$-linear stable module category as modules over the Tate construction falls into this paradigm \cite[Proposition~5.7]{Krause}.

Now suppose also that $\cU$ is compactly generated. Then the idempotent completions of $\cX^{\omega} / \cU^{\omega}$ and $\Mod_{j_*(\uno_\cU)}(\cX)^{\omega} / \cU^{\omega}$ are $\cZ^{\omega}$ and $\Mod_{\varphi(\uno_{\cU})}(\cZ)^{\omega}$, respectively. Since the idempotent completion functor (as an endofunctor of the category of small stable categories) preserves all limits \cite[Proposition~A.3.3]{nine2}, we obtain the diagram
\[ \begin{tikzcd}
\cX^{\omega} \ar{r} \ar{d} & \Mod_{j_*(\uno_\cU)}(\cX)^{\omega} \ar[hook]{r} \ar{d} & \cU^{\dzbl} \ar{d} \\
\cZ^{\omega} \ar{r} & \Mod_{\varphi(\uno_{\cU})}(\cZ)^{\omega} \ar[hook]{r} & (\cU^{\dzbl} / \cU^{\omega})^{\mathrm{ic}} \:.
\end{tikzcd}  \]
in which both squares are pullbacks (where $(-)^{\mathrm{ic}}$ denotes idempotent completion).
\end{remark}

To conclude this subsection, we record the following consequence of \Cref{theorem.algebraic_descr} in the setting of stable categories for later use.

\begin{corollary}
\label{stable_case}
Let $\cU$ and $\cZ$ be stably monoidal categories and let $\cU \xra{\varphi} \cZ$ be an exact laxly monoidal functor. Let $u \in \cU$ be an object that is contained in the thick subcategory of $\cU$ generated by $\uno_{\cU}$ (i.e., $u$ is \emph{perfect}). Then an object $[u \mapsto \varphi(u) \xla{\alpha} z] \in \cX$ is dualizable if and only if $z \in \cZ$ is dualizable and moreover the composite morphism
%\begin{equation}
%\label{3a_in_stable_case}
\[
\gamma: z \otimes \varphi(\uno_{\cU}) \xlongra{\alpha} \varphi(u) \otimes \varphi(\uno_{\cU}) \xrightarrow{\can} \varphi(u)
\]
%\end{equation}
is an equivalence.
\end{corollary}
\begin{proof}
Note that $u$ being perfect implies that it is dualizable. So by condition \Cref{condtn.alg.descr.two.objects.only} of \Cref{theorem.algebraic_descr}, it remains to show that under the stated hypotheses the morphism
\[
 g: z \otimes \varphi(u^{\vee}) \xrightarrow{\alpha} \varphi(u) \otimes\varphi(u^{\vee}) \xrightarrow{\can} \varphi(u \otimes u^{\vee})
\]
is an equivalence. Let $\cU' \subseteq \cU$ be the subcategory consisting of those $w \in \cU$ such that the morphism 
\[
 g_w: z \otimes \varphi(w) \xrightarrow{\alpha} \varphi(u) \otimes\varphi(w) \xrightarrow{\can} \varphi(u \otimes w)
\]
is an equivalence. Note that $\cU'$ is thick and contains $\uno_{\cU}$ by assumption, so it contains $u$. Hence it also contains $u^{\vee}$, i.e., the morphism $g \coloneqq g_{u^{\vee}}$ is an equivalence.
\end{proof}

\section{Stratified categories and their dualizable objects} \label{sec:2}

In this section, we introduce the concept of an $\cO$-monoidal category stratified over a base category $\cB$, which specializes to a laxly $\cO$-monoidal functor of $\cO$-monoidal categories when $\cB = [1]$. We then prove that dualizable objects in the lax limit of a symmetric monoidal $\cB$-stratified category are detected by restriction to the set of individual strata and links.

\subsection{Monoidal structures on stratified categories and their lax limits} \label{sec:lax.limit}

Given a category $\cB$, we write $\locCocart_{\cB}$ for the 2-category whose objects are locally cocartesian fibrations $\cC \to \cB$, whose morphisms are those functors over $\cB$ that preserve locally cocartesian morphisms, and whose 2-morphisms are arbitrary natural transformations. We write $\Fun^\cocart_{/\cB}(-,-)$ for its functor categories.

\begin{definition}\cite[\S A.5]{AMGR}
We write
\[
\locCocart_{\cB}
\xra{\Yo^\lax}
\Fun(\Delta_{/\cB}^\op, \Cat)
\]
for the (restricted) \bit{lax Yoneda embedding}, given by the assignment
\[ (\cC \to \cB) \longmapsto \left( \left( [n] \xra{\varphi} \cB \right) \mapsto \Fun^{\cocart}_{/[n]}(\sd[n], \varphi^{\ast} \cC) \right). \]
By \cite[Observation~A.7.4]{AMGR}, $\Yo^\lax$ is faithful. We write $\locCocart^{\lax}_{\cB}$ for its essential image.
\end{definition}

\begin{notation}
Given an operad $\cO^\otimes$ and a category $\cC$ with finite products, we let $\Mon_{\cO}(\cC)$ denote the category of \bit{$\cO$-monoids in $\cC$}, namely the full subcategory of $\Fun(\cO^\otimes, \cC)$ spanned by those functors $M$ which satisfy the Segal condition: for every object $x \in \cO^\otimes_{\langle n \rangle}$ and every collection of inert morphisms $\{  x \xra{\rho^i} x_i: 1 \leq i \leq n \}$, the induced morphism $M(x) \ra \prod_{1 \leq i \leq n} M(x_i)$ is an equivalence.

By \cite[Proposition 2.1.2.12]{HA}, under unstraightening $\Mon_{\cO}(\Cat)$ is equivalent to the category of $\cO$-monoidal categories and $\cO$-monoidal functors thereof. We then let $\Mon^\lax_{\cO}(\Cat)$ denote the larger category with the same objects but whose morphisms only required to preserve inert morphisms.
\end{notation}

\begin{definition} \label{def:StratifiedGeneral}
An \bit{$\cO$-monoidally $\cB$-stratified category} is an $\cO$-monoid in $\locCocart^{\lax}_{\cB}$.
\end{definition}

Observe that the lax limit functor admits a factorization
\[
\lim^\lax_\cB
\coloneqq
\Fun^\cocart_{/\cB} ( \sd(\cB) , -)
:
\locCocart_\cB
\xra{\Yo^\lax}
\Fun(\Delta_{/\cB}^\op, \Cat) \xrightarrow{\lim} \Cat
\]
through the lax Yoneda embedding, yielding an evident extension to a functor $\locCocart_\cB \xra{\lim^\lax_\cB} \Cat$. By \cite[Observation~A.5.8]{AMGR}, this extension participates in an adjunction
\[ \begin{tikzcd}[column sep=2cm]
\Cat
\arrow[yshift=0.9ex]{r}{\const}
\arrow[leftarrow, yshift=-0.9ex]{r}[yshift=-0.2ex]{\bot}[swap]{\lim^\lax_\cB}
&
\locCocart_\cB
\end{tikzcd} \]
It follows that $\lim^\lax_{\cB}$ sends $\cO$-monoids in $\locCocart^{\lax}_{\cB}$ to $\cO$-monoids in $\Cat$. In fact, since $\const$ also preserves products, we obtain an adjunction
\begin{equation}
\label{lim.lax.adjunction.monoidal}
\adjunct{\const}{\Mon_{\cO}(\Cat)}{\Mon_{\cO}(\locCocart^{\lax}_{\cB})}{\lim^\lax_{\cB}}.
\end{equation}
This generalizes the canonical $\cO$-monoidal structure on the lax limit of a laxly $\cO$-monoidal functor, as we now explain.

First, we show how \Cref{def:StratifiedGeneral} specializes to a laxly $\cO$-monoidal functor of $\cO$-monoidal categories when $\cB = [1]$. For this, we will exploit the observation that $\Cocart^{\lax}_{[1]}$ may be alternatively described as the full subcategory of $\Cat_{/[1]^\op}$ on the \emph{cartesian} fibrations over $[1]^\op$. For greater clarity, we now temporarily reintroduce the distinction between right-lax and left-lax morphisms, so that $\locCocart^{\rlax}_{\cB} \coloneqq  \locCocart^{\lax}_{\cB}$.

\begin{notation} For a category $\cC$, we respectively let $\Cart^{\rlax}_{\cC}, \Cocart^{\llax}_{\cC} \subseteq \Cat_{/\cC}$ be the full subcategories on the cartesian and cocartesian fibrations over $\cC$.
\end{notation}

\noindent Then by \cite[Observation~A.5.7]{AMGR}, we have that $\Cocart^{\rlax}_{[1]} \simeq \Cart^{\rlax}_{[1]^\op}$ (with the equivalence implemented by sending a cocartesian fibration to its dual cartesian fibration), so we are entitled to make the following definition.

\begin{definition}\label{def_strat} Let $\Stratlax_{\cO} \coloneqq \Mon_{\cO}(\Cart^{\rlax}_{[1]^\op})$ be the category of $\cO$-monoidal $[1]$-stratified categories and $\cO$-monoidal right-lax morphisms thereof. Let $\Strat_{\cO} \subseteq \Stratlax_{\cO}$ be the wide subcategory on those $\cO$-monoidal right-lax morphisms that strictly commute (i.e., preserve cartesian morphisms when viewed as morphisms of the underlying objects).

In the case that $\cO = \Comm$ is the commutative operad (i.e., $\Comm^\otimes = \Fin_*$), we simply write $\Stratlax \coloneqq \Stratlax_\Comm$ and $\Strat \coloneqq \Strat_\Comm$.
\end{definition}

\begin{proposition} \label{prop:StratID}
The category $\Strat_{\cO}$ is equivalent to the wide subcategory of $\Ar( \Mon^{\lax}_{\cO}(\Cat))$ spanned by those morphisms whose restrictions to $\{ 0 \}$ and $\{ 1 \}$ are strictly $\cO$-monoidal.
\end{proposition}

\begin{remark}
\label{rmk.describe.Strat.informally}
We may summarize \Cref{prop:StratID} informally as follows: $\Strat_\cO$ is the category whose objects are laxly $\cO$-monoidal functors $\cU \xra{\varphi} \cZ$ and whose morphisms are commutative squares
\[ \begin{tikzcd}
\cU
\arrow{r}{\varphi}
\arrow{d}[swap]{f}
&
\cZ
\arrow{d}{g}
\\
\cU'
\arrow{r}[swap]{\varphi'}
&
\cZ'
\end{tikzcd} \]
in $\Mon^\lax_\cO(\Cat)$ such that the morphisms $f$ and $g$ lie in the subcategory $\Mon_\cO(\Cat) \subseteq \Mon_\cO^\lax(\Cat)$.
\end{remark}

\begin{remark} \label{rem:MonoidalFunctoriality}
The proof of \Cref{prop:StratID} will additionally show that morphisms in the larger category $\Stratlax_{\cO}$ are given by lax commutative squares
\[ \begin{tikzcd}
\cU \ar{r}{\varphi}[swap, xshift=-0.1cm, yshift=-0.4cm]{\eta \rotatebox{-40}{$\Downarrow$}} \ar{d}[swap]{f}
%\ar[phantom]{rd}{\eta \SWarrow}
& \cZ \ar{d}{g} \\
\cU' \ar{r}[swap]{\varphi'} & \cZ'
\end{tikzcd} \]
in which $f$ and $g$ are strictly $\cO$-monoidal, $\varphi$ and $\varphi'$ are laxly $\cO$-monoidal, and $g \varphi' \xRightarrow{\eta} \varphi' f$ is a laxly $\cO$-monoidal natural transformation.
% (Here the objects in $\Stratlax_{\cO}$ are displayed horizontally.)

Moreover, \emph{laxly} $\cO$-monoidal right-lax morphisms of $\cO$-monoids are given by such lax commutative squares, but where we only require that the vertical morphisms $f$ and $g$ be laxly $\cO$-monoidal. As we will not need this assertion, we leave the details of its proof to the reader.
\end{remark}

To prove \Cref{prop:StratID}, we will employ the formalism of \emph{bifibrations fibered in categories}.\footnote{This differs from the notion studied in \cite[\S 2.4.7]{HTT}, which discusses bifibrations fibered in groupoids.}

\begin{definition} Let $\cA$ and $\cB$ be two categories. We say that an isofibration $\cX \xra{\pi} \cA \times \cB$ is a \bit{bifibration} if the following two conditions hold.

\begin{enumerate}
\item The composite $\pi_\cA \coloneqq \pr_\cA \circ \pi$ is a cartesian fibration and $\pi$ is a morphism of cartesian fibrations over $\cA$, so that all $\pi_\cA$-cartesian morphisms project to equivalences in $\cB$.
\item The composite $\pi_\cB \coloneqq \pr_\cB \circ \pi$ is a cocartesian fibration and $\pi$ is a morphism of cocartesian fibrations over $\cB$, so that all $\pi_\cB$-cocartesian morphisms project to equivalences in $\cA$.
\end{enumerate}

A (\bit{bilax}) \bit{morphism of bifibrations} is simply a morphism in $\Cat_{/\cA \times \cB}$. We say that a morphism
\[ \begin{tikzcd}
\cX
\arrow{rr}{F}
\arrow{rd}[sloped, swap]{\pi}
&
&
\cY
\arrow{ld}[sloped, swap]{\rho}
\\
&
\cA \times \cB
\end{tikzcd}~, \]
is \bit{right-strict} (resp.\! \bit{left-strict}) if it carries $\pi_\cA$-cartesian morphisms to $\rho_\cA$-cartesian morphisms (resp.\! $\pi_\cB$-cocartesian morphisms to $\rho_\cB$-cocartesian morphisms),\footnote{In other words, right-strict morphisms of bifibrations are left-lax, and vice-versa.} and \bit{strict} if it is both left and right strict.

We write $\Bifib_{\cA \times \cB} \subseteq \Cat_{/ \cA \times \cB}$ for the full subcategory spanned by the bifibrations, and we write $\Bifib^{\lstr}_{\cA \times \cB}$ and $\Bifib^{\rstr}_{\cA \times \cB}$ for the wide subcategories of $\Bifib_{\cA \times \cB}$ on left-strict and right-strict morphisms.
\end{definition}

\begin{proposition} \label{prop:BifibStraighteningEquivalences} We have equivalences of categories
$$ \Bifib^{\lstr}_{\cA \times \cB} \simeq \Fun(\cB, \Cart^{\rlax}_{\cA}) \qquad \text{and} \qquad \Bifib^{\rstr}_{\cA \times \cB} \simeq \Fun(\cA^\op, \Cocart^{\llax}_{\cB}) \:.$$
\end{proposition}
\begin{proof}
We will prove the second equivalence and then deduce the first by taking opposites. First note that by passage to overcategories, the usual straightening equivalence $\Cart_{\cA} \simeq \Fun(\cA^\op, \Cat)$ becomes
$$\Theta: (\Cart_{\cA})_{/(\cA \times \cB)} \simeq \Fun(\cA^\op, \Cat)_{/\const_{\cB}} \simeq \Fun(\cA^\op, \Cat_{/\cB}).$$
Given a functor $\cA^\op \xra{F} \Cat_{/\cB}$, let $\cX \xra{\pi} \cA \times \cB$ be the corresponding object under the equivalence $\Theta$. We then see that $F$ is valued in cocartesian fibrations over $\cB$ if and only if the pullback $\cX_{a} \xra{\pi_a} \{ a \} \times \cB$ is a cocartesian fibration for all $a \in \cA$. We claim that this latter condition holds if and only if $\pi$ is a bifibration. For the `if' direction, supposing that $\cX \xra{\pi_\cB} \cB$ is a cocartesian fibration whose cocartesian morphisms project to equivalences in $\cA$ ensures that the restriction $\cX_a \subseteq \cX \to \cB$ is a cocartesian fibration for all $a \in \cA$. % (since $\pi$ is always assumed to be an isofibration).
For the `only if' direction, given a morphism $b \xra{e} b'$ in $\cB$ and an object $x \in \cX$ over $b$, we need to produce a $\pi_\cB$-cocartesian lift of $e$, i.e., we need to extend $x$ to a $\pi_\cB$-colimit covering $e$. Let $a = \pi_\cA(x)$ and consider the commutative diagram
\[ \begin{tikzcd} \{0\} \ar{r}{ x }  \ar{d} & \cX_a \ar{r} \ar{d}{\pi|_a} & \cX \ar{d}{\pi} \\
{[1]} \ar{r}[swap]{ e } \ar[dashed]{ru}[sloped]{f} & \{ a \} \times \cB \ar{r}  & \cA \times \cB
\end{tikzcd}~. \]
Since $\pi|_a$ is a cocartesian fibration by assumption, there exists a map $f$ which is a $\pi|_a$-colimit of $x$. By \cite[Corollary 4.3.1.15]{HTT}, $f$ viewed as a map into $\cX$ is then a $\pi$-colimit. By \cite[Proposition 4.3.1.5(2)]{HTT}, since $\pi \circ f$ is a $\pr_\cB$-cocartesian morphism (as a degenerate morphism in the $\cA$-factor), it follows that $f$ is a $\pi_\cB$-colimit. We conclude that $\pi_\cB$ is a cocartesian fibration. In fact, since we exhibited cocartesian lifts in $\cX$ which all project to degeneracies in $\cA$, by the stability of cocartesian morphisms under equivalence it follows that all $\pi_\cB$-cocartesian morphisms in $\cX$ project to equivalences in $\cA$. This shows that $\pi$ is a bifibration.
The category $\Bifib^{\rstr}_{\cA \times \cB}$ viewed as the full subcategory of $(\Cart_{\cA})_{/(\cA \times \cB)}$ on the bifibrations is thus seen to be equivalent under $\Theta$ to $\Fun(\cA^\op, \Cocart^{\llax}_{\cB})$.
\end{proof}

\begin{proof}[Proof of \Cref{prop:StratID}] Using the equivalences of \Cref{prop:BifibStraighteningEquivalences} we get equivalences of groupoids
$$
\Fun(\cB, \Cart^{\rlax}_{\cA})^{\simeq} \simeq (\Bifib^{\lstr}_{\cA \times \cB})^{\simeq} \simeq (\Bifib^{\rstr}_{\cA \times \cB})^{\simeq} \simeq \Fun(\cA^\op, \Cocart^{\llax}_{\cB})^{\simeq} \:.
$$
Then if we let $\cB = \cO^\otimes$ and $\cA = [1]^\op$, under this equivalence we see that $\cO$-monoids in $\Cart^{\rlax}_{[1]^\op}$ correspond to functors $[1] \to \Mon_{\cO}^\lax(\Cat)$. 

We then note that under the first equivalence of \Cref{prop:BifibStraighteningEquivalences}, $\Strat_{\cO}$ identifies as the wide subcategory of $\Stratlax_{\cO}$ on the strict morphisms of bifibrations over $[1]^\op \times \cO^\otimes$. Thus, using the second equivalence of \Cref{prop:BifibStraighteningEquivalences}, we see that $\Strat$ is equivalent to the wide subcategory of $\Ar(\Mon_{\cO}^\lax(\Cat))$ spanned by those morphisms whose restrictions to $\{0\}$ and $\{1\}$ are strictly $\cO$-monoidal, as claimed.
\end{proof}

Under the equivalence $\Cocart^{\rlax}_{[1]} \simeq \Cart^{\rlax}_{[1]^\op}$, the functor $\lim^{\lax}_{[1]}$ identifies with $\Sect(-) \coloneqq  \Fun_{/[1]^\op}([1]^\op,-)$. We next identify the induced functor on $\cO$-monoids
\[ \Stratlax_{\cO} = \Mon_{\cO}(\Cart^{\rlax}_{[1]^\op}) \xra{\Sect} \Mon_{\cO}(\Cat) \]
in more familiar terms. For convenience, we will suppose the operad $\cO^\otimes$ is reduced.

We also begin to use quasicategories per se. Namely, in the proof of the next proposition, we use Lurie's pairing construction \cite[Corollary 3.2.2.13]{HTT}: given $K, L \in \sSet_{/S}$, we write $\widetilde{\Fun}_S(K,L) \in \sSet_{/S}$ for the simplicial set over $S$ defined by the formula
\[ \Hom_{/S}(A, \widetilde{\Fun}_S(K,L)) \cong \Hom_{/S}(A \times_S K, L)~. \]
If $K \to S$ is a cartesian fibration and $L \to S$ is a cocartesian fibration, then $\widetilde{\Fun}_S(K,L) \to S$ is a cocartesian fibration with fibers $\Fun(K_s,L_s)$, such that for a morphism $s \xra{f} t$ in $S$, the pushforward functoriality $f_!$ is such that the diagram
\[ \begin{tikzcd}
K_s \ar{r}{F} & L_s \ar{d}{f_!} \\
K_t \ar{r}[swap]{f_! F} \ar{u}{f^\ast} & L_t
\end{tikzcd} \]
commutes.

\begin{proposition} \label{prop:MonoidalLaxLimitID}
Suppose $\cO^\otimes$ is a reduced operad, let $\cU \xra{\varphi} \cZ$ be a laxly $\cO$-monoidal functor, and let $\cM \to [1]^\op$ be its unstraightening to a cartesian fibration. Endow $\cM$ with the structure of an $\cO$-monoid in $\Cart^{\rlax}_{[1]^\op}$ under the equivalence of \Cref{prop:StratID}, and let $\Sect(\cM)^\otimes \to \cO^\otimes$ be the resulting $\cO$-monoidal structure on its category of sections. Then we have a pullback square of operads fibered over $\cO^\otimes$
\[ \begin{tikzcd}
\Sect(\cM)^\otimes \ar{r} \ar{d} & (\cZ^\otimes)^{\Delta^1} \ar{d}{\ev_1} \\
\cU^\otimes \ar{r}[swap]{\varphi} & \cZ^\otimes
\end{tikzcd} \:, \]
identifying $\Sect(\cM)^\otimes$ with the canonical $\cO$-monoidal structure of \Cref{obs:RecollementsAreLaxLimits}.
\end{proposition}
\begin{proof} In the statement, $(\cZ^\otimes)^{\Delta^1}$ denotes the cotensor in operads over $\cO^\otimes$, which defines the pointwise $\cO$-monoidal structure on $\Ar(\cZ)$. First, let $\cM^\otimes \to [1]^\op \times \cO^\otimes$ be the bifibration corresponding to $\varphi$ under the second equivalence of \Cref{prop:BifibStraighteningEquivalences} with $\cA = [1]^\op$ and $\cB = \cO^\otimes$. Let $[1]^\op \times \cO^\otimes \xra{p} \cO^\otimes$ be the projection. We observe that the formation of relative sections
$$p_{\ast} \cM^\otimes \coloneqq  \widetilde{\Fun}_{\cO^\otimes}([1]^\op \times \cO^\otimes, \cM^\otimes)$$
computes the induced $\cO$-monoidal structure on $\Sect(\cM)$.\footnote{A routine application of \cite[Theorem~B.4.2]{HA} shows that $p_{\ast} \cM^\otimes$ is indeed an operad over $\cO^\otimes$.}

We then have a right-strict morphism $\cM^\otimes \to [1]^\op \times \cZ^\otimes$ of bifibrations over $[1]^\op \times \cO^\otimes$ determined under the second equivalence of \Cref{prop:BifibStraighteningEquivalences} by the commutative square of operads
\[ \begin{tikzcd}
\cU^\otimes \ar{r}{\varphi} \ar{d}[swap]{\varphi} & \cZ^\otimes \ar{d}{\id} \\
\cZ^\otimes \ar{r}[swap]{\id} & \cZ^\otimes.
\end{tikzcd} \]
This specifies a functor $\cM^\otimes \to \cZ^\otimes$ over $\cO^\otimes$ and hence a functor $p_{\ast} \cM^\otimes \xra{F} (\cZ^\otimes)^{\Delta^1}$ over $\cO^\otimes$ by postcomposition, using the identification $\Delta^1 \cong [1]^\op$. Over $\cO \simeq \ast$, this restricts to the functor $\Sect(\cM) \to \Ar(\cZ)$ that sends a section $[z \to u]$ to $[z \to \varphi(u)]$. It's also straightforward to check that $F$ is a morphism of operads and $p_{\ast} \cM^\otimes \xra{i^* = \ev_0 \circ F} \cZ^\otimes$ is $\cO$-monoidal (though $F$ isn't $\cO$-monoidal unless $\varphi$ is).

On the other hand, the inclusion $\{1\} \times \cO^\otimes \subseteq [1]^\op \times \cO^\otimes$ defines an $\cO$-monoidal functor $p_{\ast} \cM^\otimes \xra{j^*} \cU^\otimes$ that restricts over $\cO$ to the functor $\Sect(\cM) \to \cU$ given by evaluation at $1$. We now have our desired commutative square of operads, which is also a pullback square of underlying categories. Furthermore, both $\Sect(\cM)^\otimes$ and $(\cZ^\otimes)^{\Delta^1} \times_{\cZ^\otimes} \cU^\otimes$ are $\cO$-monoidal categories (in fact, the projections to $\cU^\otimes$ are cocartesian fibrations in both cases), and using joint conservativity of the two $\cO$-monoidal functors $i^*, j^*$ to $\cU^\otimes$ and $\cZ^\otimes$, we see that the comparison map $\Sect(\cM)^\otimes \to (\cZ^\otimes)^{\Delta^1} \times_{\cZ^\otimes} \cU^\otimes$ is $\cO$-monoidal and hence an equivalence as such.
\end{proof}

\begin{remark} \label{rem:StrongerUniversalProperty}
\Cref{prop:MonoidalLaxLimitID} shows that the lax limit of a laxly $\cO$-monoidal functor satisfies a universal property with respect to \emph{laxly} $\cO$-monoidal functors, which is stronger than that encoded by the adjunction \Cref{lim.lax.adjunction.monoidal}. This universal property corresponds to that of the lax limit when taken in the $2$-category $\Mon_{\cO}^{\lax}(\Cat)$, and in fact one always has an adjunction
\[ \adjunct{\const}{\Mon_{\cO}^{\lax}(\Cat)}{\Mon_{\cO}^{\lax}(\locCocart^{\lax}_{\cB})}{\lim^{\lax}_{\cB}} \:. \]
In conjunction with the second half of \Cref{rem:MonoidalFunctoriality}, we then see that a lax commutative square of laxly $\cO$-monoidal functors of $\cO$-monoidal categories
\[ \begin{tikzcd}
\cU \ar{r}{\varphi}[swap, xshift=-0.1cm, yshift=-0.4cm]{\eta \rotatebox{-40}{$\Downarrow$}} \ar{d}[swap]{f}
& \cZ \ar{d}{g} \\
\cU' \ar{r}[swap]{\varphi'} & \cZ'
\end{tikzcd} \]
furnishes a laxly $\cO$-monoidal functor $\cX \to \cX'$ between lax limits. With a bit more work, one may understand this functoriality in terms of the explicit formulas of \cite[\S 2.4]{ES} in the case $\cO = \Comm$ (cf. the discussion prior to \cite[Theorem~2.64]{ES}).
\end{remark}

\subsection{Understanding monoidal structures on lax limits in general} \label{sec:2.2}

To understand $\cO$-monoidal structures on lax limits over $\cB$, we would like to reduce to considering lax limits over $[1]$ and full subcategories $\cB_0, \cB_1 \subseteq \cB$ forming a sieve-cosieve decomposition of $\cB$ by means of the following strategy. Namely, suppose we have a functor $\cB \xra{\pi} [1]$ (which is equivalent data to such a decomposition of $\cB$). Then we have a factorization
\[ \const: \Cat \xrightarrow{\const} \Cocart^{\lax}_{[1]} \xrightarrow{\pi^{\ast}} \locCocart^{\lax}_{\cB} \]
(cf. \Cref{obs:strategy} below), and we would like to factor $\lim^{\lax}_{\cB}$ as a composition of corresponding right adjoints
\[ \begin{tikzcd}
\locCocart^{\lax}_{\cB} \ar[dotted]{r}{\pi^{\lax}_{\ast}} & \Cocart^{\lax}_{[1]} \ar{r}{\lim^\lax_{[1]}} & \Cat 
\end{tikzcd}~, \]
where we first take ``(pointwise) lax right Kan extension along $\pi$'' and then the lax limit over $[1]$. Furthermore, $\pi^{\lax}_{\ast}$ should have the base-change property
\[ \begin{tikzcd}
\locCocart^{\lax}_{\cB_0} \ar{d}[swap]{\lim^{\lax}_{\cB_0}} & \locCocart^{\lax}_{\cB} \ar{r}{\iota_1^\ast} \ar{l}[swap]{\iota_0^\ast} \ar[dotted]{d}{\pi^{\lax}_{\ast}} & \locCocart^{\lax}_{\cB_1} \ar{d}{\lim^{\lax}_{\cB_1}} \\
\Cat & \Cocart^{\lax}_{[1]} \ar{r}[swap]{\iota^\ast_1} \ar{l}{\iota^\ast_0} & \Cat
\end{tikzcd} \]
However, in this generality one only expects $\pi^{\lax}_{\ast}$ to be partially defined, since one needs conditions on $\cC \to \cB$ for the monodromy functor $\lim^\lax_{\cB_0}(\cC|_{\cB_0}) \xra{\varphi} \lim^\lax_{\cB_1}(\cC|_{\cB_1})$ to exist.\footnote{Even if we attempted to invoke the adjoint functor theorem to construct $\pi_*^{\lax}$, we would then need to impose additional conditions so as to validate the base-change property. This is analogous to the situation for ordinary pointwise Kan extensions.}

In \cite[\S 2.2]{QS}, the third author determined the relevant conditions for existence and gave a direct construction of $\pi_{\ast}^{\lax}$ at the level of objects in terms of its monodromy functor \cite[Proposition 2.27]{QS}. For example (cf. \cite[Remark 2.28]{QS}), if $\cB = P$ is a poset, then for every $x \in P_1$, if we let $J_x$ be the poset whose objects are strings $[a_0 < ... < a_n < x]$, $n \geq 0$ with $a_i \in P_0$ and whose morphisms are string inclusions, then the value of $\varphi(\sd(P_0) \xra{\sigma} \cC|_{P_0})$ on $x \in \sd(P_1)$ is computed as the limit of the functor
\[ J_x \xlongra{f_x} \cC_x \:, \qquad [a_0 \xra{\alpha_0} ... \xra{\alpha_{n-1}} a_n \xra{\alpha_n} x] \longmapsto (\alpha_n)_! ... (\alpha_0)_! \sigma(a_0) \:, \]
provided these limits exist and are preserved by pushforward along all monodromy functors $\cC_x \xra{\beta_!} \cC_y$ over morphisms $x \xra{\beta} y$ in $P_1$. For example, if we suppose that $P$ is finite, all fibers of $\cC \to P$ admit finite limits, and all pushforward functors are left-exact, then the lax right Kan extension $\pi_*^{\lax}(\cC)$ exists for any functor $P \xra{\pi} [1]$. However, \cite[\S 2.2]{QS} did not establish the functoriality of $\pi_{\ast}^{\lax}$ in \emph{lax} morphisms over $\cB$, which is essential for the passage to $\cO$-monoids.

The situation for us becomes more favorable when $\cB_0 = \pt$ (so that $\cB \simeq \cB_1^{\lhd}$ and $\pi$ is the structure map of the join), since then the relevant indexing categories are all contractible and thus $\pi_*^{\lax}(\cC)$ exists \emph{unconditionally}. For our application to the proof of \Cref{thm:GeneralDualizability}, this case suffices; in fact, we will only need to construct $\pi_*^{\lax}$ when $\cB = [n]$ and $\cB_0 = \{ 0 \}$. We now outline our strategy for doing this as follows.

\begin{observation} \label{obs:strategy}
Let $\cB \xra{\pi} \cB'$ be a functor of categories. This determines an adjunction
\[
\adjunct{\pi \circ (-) \eqqcolon \pi_\circ }{\Cat_{/\cB}}{\Cat_{/\cB'}}{ \pi^\circ \coloneqq \cB \times_{\cB'} (-) }
~,
\]
which is opposite to an adjunction
\[
\adjunct{(\pi^\circ)^\op}{\Cat_{/\cB'}^\op}{\Cat_{/\cB}^\op}{(\pi_\circ)^\op}
~.
\]
Applying $\Fun(-,\Cat)$, we obtain an adjunction
\[
\adjunctbig{(-) \circ (\pi_\circ)^\op \eqqcolon \pi^*}{\Fun(\Cat_{/\cB'}^\op,\Cat)}{\Fun(\Cat_{/\cB}^\op,\Cat)}{\pi_* \coloneqq (-) \circ (\pi^\circ)^\op}
~.
\]
So explicitly, given a presheaf $ \Cat^\op_{/\cB} \xra{\cF} \Cat$, the presheaf $\Cat^\op_{/\cB'} \xra{\pi_* \cF} \Cat$ is given by the formula
\[
(\pi_* \cF)(K)
\coloneqq
\cF ( \cB \times_\cB' K)
~.
\]

We would like to restrict to the full subcategories $\Delta_{/\cB} \xhookrightarrow{i_\cB} \Cat_{/\cB}$ and $\Delta_{/\cB'} \xhookrightarrow{i_{\cB'}} \Cat_{/\cB'}$. Note the factorization
\[ \begin{tikzcd}
\Cat_{/\cB}
\arrow{r}{\pi_\circ}
&
\Cat_{/\cB'}
\\
\Delta_{/\cB}
\arrow[hook]{u}
\arrow[dashed]{r}[swap]{\pi_\circ'}
&
\Delta_{/\cB'}
\arrow[hook]{u}
\end{tikzcd}~. \]
Then, we obtain an adjunction
\[
\adjunctbigger{ (-) \circ (\pi_\circ')^\op \simeq (i_\cB)^* \circ \pi^* \circ (i_{\cB'})_* \eqqcolon \pi^*}{\Fun(\Delta_{/\cB'}^\op,\Cat)}{\Fun(\Delta_{/\cB}^\op,\Cat)}{\pi_* \coloneqq (i_{\cB'})^* \circ \pi_* \circ (i_\cB)_* }
~.
\]
In particular, we have that
\[ \pi_{\ast} \Yo^\lax(\cC \to \cB)([m] \xra{\varphi} \cB') = \Fun^{\cocart}_{/[m] \times_{\cB'} \cB}(\sd([m] \times_{\cB'} \cB),  \overline{\varphi}^\ast \cC)~, \]
where $\overline{\varphi}$ denotes the pullback of $\varphi$ by $\pi$.\footnote{Here we use that $\Yo^\lax(\cC \to \cB)$ extends over $\Cat_{/\cB}^\op$ by the same formula
\[ (\cC \to \cB) \longmapsto \left( (K \xra{\varphi} \cB) \mapsto \Fun^{\cocart}_{/K}(\sd(K), \varphi^{\ast} \cC) \right) \]
(see \cite[Observation~A.5.5]{AMGR}).
Note that in this generality, we use \cite[Observation~A.3.9]{AMGR} to see that $\sd([n]) \xra{\max} [n]$ extends to a locally cocartesian fibration for all categories, but we will only be concerned with the case where $\pi$ is a functor of posets, in which case the usual definition of the subdivision poset is consistent with the abstract definition via left Kan extension by \cite[Lemma~A.3.7]{AMGR}.}

Next, we note that by definition we have a commutative square
\[ \begin{tikzcd}
\locCocart_{\cB'} \ar{r}{\Yo^\lax} \ar{d}[swap]{\pi^\ast} & \Fun(\Delta^\op_{/\cB'}, \Cat) \ar{d}{\pi^\ast}  \\
\locCocart_{\cB} \ar{r}[swap]{\Yo^\lax} & \Fun(\Delta^\op_{/\cB}, \Cat)
\end{tikzcd} \]
and hence $\pi^\ast$ extends to a functor $\locCocart^{\lax}_{\cB'} \xra{\pi^*} \locCocart^{\lax}_{\cB}$. Let $(\locCocart^\lax_{\cB})' \subseteq \locCocart^\lax_{\cB}$ be the full subcategory on those locally cocartesian fibrations $\cC \to \cB$ for which $\pi_{\ast} \Yo^\lax(\cC \to \cB)$ is equivalent to $\Yo^\lax(\cC' \to \cB')$ for some locally cocartesian fibration $\cC' \to \cB'$. Then $\pi^\ast$ admits a partially defined right adjoint
\[ (\locCocart^\lax_{\cB})' \xra{\pi^\lax_{\ast}} \locCocart^\lax_{\cB'} \]
computed by $\pi_{\ast}$ under the image of $\Yo^\lax$.
\end{observation}

We also record an adjointability property of the functors in \Cref{obs:strategy}.

\begin{lemma} \label{lem:Adjointability}
Suppose we have a pullback square of categories
\[ \begin{tikzcd}
\cA \ar{r}{f} \ar{d}[swap]{\pi'} & \cB \ar{d}{\pi} \\
\cA' \ar{r}[swap]{f} & \cB'.
\end{tikzcd} \]
Then the exchange transformation $f^\ast \pi_\ast \to \pi_{\ast} f^\ast$ is a natural equivalence.
\end{lemma}
\begin{proof}
In the statement, we may regard these as functors
\[ \Fun(\Cat^\op_{/\cB}, \Cat) \longra \Fun(\Cat^\op_{/\cA'}, \Cat) \qquad \textrm{or} \qquad \Fun(\Delta^\op_{/\cB}, \Cat) \longra \Fun(\Delta^\op_{/\cA'}, \Cat), \]
and it suffices to check in the former case. For this, given a presheaf $\Cat^\op_{/\cB} \xra{\cF} \Cat$ we observe that the morphism
\[ (f^\ast \pi_\ast \cF)(K \to \cA') = \cF(K \times_{\cB'} \cB) \longra \cF(K \times_{\cA'} \cA) = (\pi_{\ast} f^\ast \cF)(K \to \cA') \]
is an equivalence since $K \times_{\cB'} \cB \simeq K \times_{\cA'} \cA$.
\end{proof}

\subsection{An unconditional lax right Kan extension} \label{sec:2.3}

We now let $\cB = [n]$ and consider a map of posets for which $\pi_{\ast}^\lax$ will exist unconditionally.

\begin{notation}\label{notation:map_zeta}
For $n \geq 1$, let $[n] \xra{\zeta} [1]$ denote the map of totally ordered sets given by $\zeta(0) = 0$ and $\zeta(i) = 1$ for $i \geq 1$.
\end{notation}

Our goal is to show that $\zeta^\lax_{\ast}$ is defined on all of $\locCocart^\lax_{[n]}$. First, we give an alternative construction of what will be the dual cartesian fibration of $\zeta^{\lax}_{\ast}(\cC)$.

\begin{observation} \label{obs:ExtraCocartesianFunctoriality}
The composite functor $\sd[n] \xra{\min} [n]^\op \xra{\zeta^\op} [1]^\op$ is a cocartesian fibration, such that a morphism given by a proper string inclusion $\sigma \subset \sigma'$ is $\zeta^\op \min$-cocartesian if and only if $\sigma' = \sigma \cup \{ 0 \}$. This follows by noting that for any map of posets $ P \xra{f} [1]^\op$, a morphism $x < y$ over $(1^\circ \ra 0^\circ) \coloneqq (0 \ra 1)^\circ$ is $f$-cocartesian if and only if for all $z \in P_0$, if $x < z$ then $y \leq z$.

Note then that with respect to the identification of the fibers as $\sd[n]_{0} \cong \sd[n-1]$ (via $\{1,...,n \} \cong [n-1]$) and $\sd[n]_1 = \sd[n]_{\min = 0}$, this cocartesian fibration is classified by the map of posets $\sd[n-1] \to \sd[n]_{\min = 0}$ given by prepending $0$. (So, this functor witnesses $\sd[n]_{\min=0}$ as the left cone on $\sd[n-1]$.)
\end{observation}

% don't do this: Let $\leftnat{\sd[n]}$ denote $\sd[n]$ as a marked simplicial set with the $\zeta^\op \min$-cocartesian morphisms marked. 
\begin{construction} \label{con:DualOfLaxRKE}
Let $\leftnat{\sd[n]}$ denote $\sd[n]$ as a marked simplicial set with the $\max$-locally cocartesian morphisms marked. Consider the span of marked simplicial sets
\[ \begin{tikzcd}[column sep=8ex]
(\left[1 \right]^\op)^\flat & \leftnat{\sd[n]} \ar{r}{\max} \ar{l}[swap]{\zeta^\op \min} & \left[n \right]^\sharp,
\end{tikzcd}   \]
where we note that $\min$ of every $\max$-locally cocartesian morphism is an equality. This induces an adjunction
\[ \adjunctbig{\max_! (\zeta^\op \min)^\ast}{\sSet^+_{/([1]^\op)^{\flat}}}{\sSet^+_{/[n]}}{(\zeta^\op \min)_\ast \max^\ast} \:. \]

Equip $\sSet^+_{/[n]}$ with the locally cocartesian model structure and $\sSet^+_{/([1]^\op)^{\flat}}$ with the model structure of \cite[Appendix~B]{HA} given by the trivial categorical pattern, so that the fibrant objects are inner fibrations with the equivalences marked; these respectively model $\locCocart_{[n]}$ and $\Cat_{/[1]^\op}$ (as inner fibrations over a poset are automatically categorical fibrations). Then since $\zeta^\op \min$ is a flat inner fibration by \cite[Example B.3.11]{HA},\footnote{Beware that $\sd[n] \xra{\min} [n]^\op$ is not a flat fibration for $n>1$: this construction is highly specific to the choice of $\zeta$.} it follows from \cite[Theorem~B.4.2]{HA} that this adjunction is a Quillen adjunction.

Thus for a locally cocartesian fibration $\cC \xra{p} [n]$, we may define
\[ (\zeta^{\lax}_{\ast} \cC)^{\vee} \coloneqq (\zeta^\op \min)_\ast \max^\ast (\leftnat{\cC} \xrightarrow{p} [n])  \]
as an inner fibration over $[1]^\op$. 
\end{construction}

\begin{proposition} \label{prop:GetCartesianFibration}
Let $\cC \xra{p} [n]$ be a locally cocartesian fibration. Then $(\zeta^{\lax}_{\ast} \cC)^{\vee} \xra{q} [1]^\op$ is a cartesian fibration, and a morphism $f$ of $(\zeta^{\lax}_{\ast} \cC)^{\vee}$ over $1^\circ \ra 0^\circ$ is $q$-cartesian if and only if the corresponding functor
\[\leftnat{\sd[n]} \xlongra{F} \leftnat{\cC} \]
sends $\zeta^\op \min$-cocartesian morphisms to equivalences.
\end{proposition}
\begin{proof}
Since $q$ was already shown to be an inner fibration in the course of its construction, it suffices to exhibit a sufficient supply of $q$-cartesian morphisms to prove that $q$ is cartesian. Recall from \cite[Example 4.3.1.4]{HTT} that $f$ is $q$-cartesian if and only if $f$ is a $q$-limit diagram. Note then that $f$ is a $q$-limit diagram if in the corresponding square
\[ \begin{tikzcd}
\sd\left[n \right]_{\min = 0} \ar{r}{F_0} \ar{d}[swap]{\iota} & \cC \ar{d}{p} \\
\sd\left[n \right] \ar{r}[swap]{\max} \ar{ru}[sloped]{F} & \left[n \right],
\end{tikzcd} \]
$F$ is a $p$-right Kan extension (\cite[Definition 4.3.2.2]{HTT}) of its restriction $F_0$. We now apply \cite[Proposition~2.27]{QS} with $\pi = \zeta$ (noting that $\sd([n])_0$ there is our $\sd([n])_{\min=0}$). In the notation there, we have that for all $0<i \leq n$, $J_i = \{0<i\}$ (cf. \cite[Remark 2.28]{QS}), so conditions (1) and (2) there are necessarily satisfied and hence the $p$-right Kan extension exists as a functor
\[ \Fun_{/[n]}(\leftnat{\sd[n]_{\min=0}}, \leftnat{\cC}) \xlongra{\iota_{\ast}} \Fun_{/[n]}(\leftnat{\sd[n]}, \leftnat{\cC}) \]
right adjoint to restriction along $\iota$. Moreover, by the explicit formula of \cite[Proposition~2.27]{QS}, $\iota_{\ast} F_0$ inverts all $\zeta^\op \min$-cocartesian morphisms, using their identification as those morphisms given by prepending $0$ (\Cref{obs:ExtraCocartesianFunctoriality}). Conversely, if $F$ inverts all $\zeta^\op \min$-cocartesian morphisms, then it follows that the counit map $F \ra \iota_{\ast} F_0$ is an equivalence by the two-out-of-three property of equivalences. We conclude both the description of the $q$-cartesian morphisms and that $q$ is a cartesian fibration.
\end{proof}

\begin{example}
Suppose $n=1$. Then \Cref{con:DualOfLaxRKE} applied to a cocartesian fibration $\cC \to [1]$ yields the dual cartesian fibration $\cC^{\vee} \to [1]^\op$ studied in \cite{BGN}.
\end{example}

\begin{remark} \label{rem:UnwindingDual}
Let $\cC \xra{p} [n]$ be a locally cocartesian fibration. Then
$$(\zeta^{\lax}_{\ast} \cC)^{\vee}_0 = \Fun_{/[n]}(\leftnat{\sd[n-1]}, \leftnat{\cC})$$
is by definition the right-lax limit of $\cC|_{[n-1]}$, and the morphism
\[ \begin{tikzcd}
(\zeta^{\lax}_{\ast} \cC)^{\vee}_1 = \Fun_{/[n]}(\leftnat{\sd[n]_{\min=0}}, \leftnat{\cC}) \ar{r}{\ev_0} & \cC_0
\end{tikzcd} \]
is an equivalence (in fact a trivial fibration) by \cite[Proposition~2.21(2)]{QS}.
\end{remark}

The cartesian monodromy of $(\zeta^{\lax}_{\ast} \cC)^{\vee} \ra [1]^\op$ is then given by the functor
\[ \cC_0 \longra \Fun_{/[n]}(\leftnat{\sd[n-1]}, \leftnat{\cC}) \]
that sends an object $x$ to the collection $\{ \varphi_{0j}(x) : 0<j \leq n \}$ together with the structure of the various canonical maps, homotopies, etc. implicitly encoded by the condition that $p$ is a locally cocartesian fibration (where $\cC_i \xra{\varphi_{ij}} \cC_j$ denotes the locally cocartesian monodromy functor of $\cC$ for $0 \leq i < j \leq n$).

By \Cref{prop:GetCartesianFibration}, we get a functor
\[ \locCocart_{[n]} \xra{\zeta^{\lax}_{\ast}(-)^{\vee}} \Cart_{[1]^\op}. \]
Writing $\Cart_{[1]^\op} \xra{(-)^{\vee}} \Cocart_{[1]}$ for the dual cocartesian fibration functor, we define
\[ \locCocart_{[n]} \xra{\zeta^{\lax}_{\ast}(-) \coloneqq (\zeta^{\lax}_{\ast}(-)^{\vee})^{\vee}} \Cocart_{[1]}. \]

The following result demonstrates the consistency of this definition of $\zeta^{\lax}_{\ast}(-)$ with that of \Cref{obs:strategy}.

\begin{theorem} \label{thm:ExistenceRKE}
Let $\cC \to [n]$ be a locally cocartesian fibration. Then there is a canonical equivalence
$$\zeta_{\ast} \Yo^{\lax}(\cC \to [n]) \simeq \Yo^{\lax}(\zeta^{\lax}_{\ast}(\cC) \to [1]).$$
Consequently, $\zeta^{\lax}_{\ast}$ as defined with domain $\locCocart_{[n]}$ extends to a right adjoint
$$\locCocart^{\lax}_{[n]} \xra{\zeta^{\lax}_{\ast}} \Cocart^{\lax}_{[1]}$$
given as the restriction of the right adjoint $\Fun(\Delta^\op_{/[n]}, \Cat) \xra{\zeta_{\ast}} \Fun(\Delta^\op_{/[1]}, \Cat)$.
\end{theorem}

\begin{remark}
Let $\cB$ be a category and let $\cB \xra{\pi} \ast$ be the unique functor. We already have that the composition
\[ \Fun(\Delta^\op_{/\cB}, \Cat) \xlongra{\pi_\ast} \Fun(\Delta^\op, \Cat) \xra{\ev_{[0]}} \Cat \]
restricts on the full subcategory $\locCocart^{\lax}_{\cB}$ as $\lim^{\lax}_{\cB}$. By an argument that is similar to (but easier than) the proof of \Cref{thm:ExistenceRKE}, one may also see that $\lim^{\lax}_{\cB}$ is the restriction of $\pi_{\ast}$ under $\Yo^{\lax}$.
\end{remark}

We will give the proof of \Cref{thm:ExistenceRKE} after first proving two combinatorial lemmas.

\begin{lemma} \label{lem:LocalizationSdToTw}
Let $[m] \xra{\rho} [p]$ and $[n] \xra{\pi} [p]$ be maps of totally ordered sets, form the fiber product $\Tw[m] \times_{\Tw[p]} \sd[n]$ with respect to the obvious induced map $\Tw[m] \rightarrow \Tw[p]$ and the `minmax' map $\sd[n] \rightarrow \Tw[p]$ given by $\sigma \mapsto (\pi \min \sigma \leq \pi \max \sigma)$, and let
\[ \sd([m] \times_{[p]} [n]) \xlongra{F} \Tw[m] \times_{\Tw[p]} \sd[n] \]
be given by $((a_1,x_1),...,(a_k,x_k)) \mapsto ((a_1 \leq a_k),(x_1,...,x_k))$, where $(x_1,...,x_k)$ denotes the image as a string in $[n]$ so that we implicitly remove duplicated entries. Let $W \subseteq \sd([m] \times_{[p]} [n])_1$ be the subset of morphisms sent to identities under $F$. Then the map of marked simplicial sets
\[ (\sd([m] \times_{[p]} [n]), W) \xlongra{F} (\Tw[m] \times_{\Tw[p]} \sd[n])^{\flat} \]
is a weak equivalence in the marked model structure on $\sSet^+$.
\end{lemma}
\begin{proof}
By \cite[Theorem~3.8]{MG}, it suffices to show that for every $r \geq 0$, the induced map
\[ \Fun([r],\sd([m] \times_{[p]} [n]))^W \xlongra{F_r} \Map([r], \Tw[m] \times_{\Tw[p]} \sd[n]) \]
is a weak homotopy equivalence, where we restrict on the lefthand side to natural transformations lying in $W$ and on the righthand side to natural isomorphisms (so that we obtain the discrete set on the objects of $\Fun([r], \Tw[m] \times_{\Tw[p]} \sd[n])$). We first consider the case $r = 0$. Given an object
\[ (a \leq b, \sigma \coloneqq (x_1,...,x_k)) \in (\Tw[m] \times_{\Tw[p]} \sd[n])~, \]
let $P_{ab,\sigma} = F^{-1}(a\leq b ,\sigma) \subseteq \sd([m] \times_{[p]} [n])$ be the fiber, which we must show is weakly contractible.

We proceed by induction on the length $k$. For the base case $k=1$, we have that $P_{ab,\sigma}$ has an initial object $((a,x_1),(b,x_1))$, hence it is weakly contractible. Now suppose that we have shown $P_{cd,\tau}$ is weakly contractible for all pairs $(c \leq d, \tau)$ with $\tau$ of length $<k$. Let $P' \subseteq P = P_{ab,\sigma}$ be the subposet on strings $((a,x_1),...,(b,x_k))$ such that the only entry containing $x_1$ in its second factor is the initial $(a,x_1)$. We define a map of posets $P \xra{r} P'$ as follows: if $(c,x_1)$ is the last entry in a string $\lambda \in P$ with $x_1$ in its second factor, so that we may write $$\lambda = ((a,x_1),...,(c,x_1),(d,x_2),...,(b,x_k)),$$ then we let $r(\lambda) \coloneqq ((a,x_1),(d,x_2),...,(b,x_k))$. Then since $r(\lambda) \leq \lambda$ and $r|_{P'} = \id$, we have that $r$ is right adjoint to the inclusion $P' \subseteq P$, so that $r$ induces a weak homotopy equivalence $|P| \xrightarrow{\sim} |P'|$. Next let $a' \in [m]$ be the smallest number such that $a \leq a'$ and $\rho(a') = \pi(x_2)$, and let $P'' \subseteq P'$ be the subposet on strings $((a,x_1),(a',x_2),...,(b,x_k))$. Define a map of posets $P' \xra{r'} P''$ by sending a string $$\lambda  = ((a,x_1),(d,x_2),...,(b,x_k)) \in P'$$ to the image of $((a,x_1),(a',x_2),(d,x_2),...,(b,x_k))$ as a string in $[m] \times_{[p]} [n]$. Then $\lambda \leq r'(\lambda)$ and $r|_{P''} = \id$, so $r'$ is left adjoint to the inclusion $P'' \subseteq P'$ and hence induces a weak homotopy equivalence $|P'| \xrightarrow{\sim} |P''|$. Now note that we have an isomorphism $P'' \cong P_{a'b,\sigma'}$ for $\sigma' = (x_2,...,x_k)$, given by $\lambda \mapsto \lambda - \{(a,x_1)\}$. By the inductive hypothesis, we deduce that $|P''| \simeq \ast$, which concludes the proof if $r=0$.

Now suppose $r>0$ and let
$$\theta = [(a_0 \leq b_0, \sigma_0) \rightarrow ... \rightarrow (a_r \leq b_r, \sigma_r)] \in \ob (\Fun([r],\Tw[m] \times_{\Tw[p]} \sd[n]))~,$$
so that $\theta$ is the data of $[a_r \leq ... \leq a_0 \leq b_0 \leq ... \leq b_r]$ in $[m]$ and $\sigma_0 \subseteq \ldots \subseteq \sigma_r$ in $\sd([n])$, with constraints $\rho(a_i) = \pi(\min \sigma_i)$ and $\rho(b_i) = \pi(\max \sigma_i)$ for all $0 \leq i \leq r$. For $0 \leq j \leq r$, let $S_j \subseteq P_{a_j b_j, \sigma_j}$ be the subposet on those strings $\lambda$ such that for all $0 \leq i \leq j$, the pairs $(a_i, \min \sigma_i)$ and $(b_i, \max \sigma_i)$ are in $\lambda$. We then have `restriction' functors
\[ S_r \xlongra{\res} S_{r-1} \xlongra{\res} \cdots \xlongra{\res} S_0 \]
which send $\lambda \in S_{j+1}$ to the largest substring $((a_j, \min \sigma_j), ..., (b_j, \max \sigma_j))$ of $\lambda$ in $S_j$, such that if $\lambda' \subseteq \lambda$ is any substring with $\lambda' \in S_j$, then $\lambda' \subseteq \res(\lambda)$. Then let $M \rightarrow [r]$ be the cartesian fibration classified by this sequence of functors, which we may suppose is a poset, and note that we have an equivalence of posets
\[ \Fun_{/[r]}([r], M) \simeq F_r^{-1}(\theta) \]
as a nested sequence of strings $\lambda_0 \subseteq ... \subseteq \lambda_r$ in $F_r^{-1}(\theta)$ is the same data as a section of $M$. For \emph{any} cartesian fibration $\cC \rightarrow [r]$ classified by $\cC_r \xrightarrow{f_r} ... \xrightarrow{f_1} \cC_0$, we have an adjunction
\[ \adjunct{j^{\ast}}{\Fun_{/[r]}([r], \cC)}{\cC_r}{j_{\ast}} \]
where $j^{\ast}$ is evaluation at $r$ and $j_{\ast}(c) = [f_1 \circ ... \circ f_r(c) \rightarrow  ... \rightarrow f_r(c) \rightarrow c ]$. Therefore, to show that $F_r^{-1}(\theta)$ is weakly contractible, it suffices to show that $S_r$ is weakly contractible. Now for $0 < i \leq r$ define convex substrings $\sigma_r^{-i} = (\min \sigma_i, ..., \min \sigma_{i-1})$ and $\sigma_r^i = (\max \sigma_{i-1}, ..., \max \sigma_{i})$ of $\sigma_r$, and also let $\sigma_r^0 = (\min \sigma_0, ..., \max \sigma_0)$ be a convex substring, so that $\sigma_r = \bigcup_{i = -r}^r \sigma_r^i$. Observe that we have a product decomposition
\[ S_r \cong \left( \prod_{i=1}^{r} P_{a_i a_{i-1}, \sigma_r^{-i}} \right) \times P_{a_0 b_0, \sigma_r^0} \times \left( \prod_{i=1}^r P_{b_{i-1} b_i, \sigma_r^i} \right)  \]
which, given $\lambda \in S_r$, sends it on the factor $P_{a_0 b_0, \sigma_r^0}$ to the convex substring bounded by $(a_0, \min \sigma_0)$ and $(b_0, \max \sigma_0)$, and similarly for the other factors. Since we showed that each factor on the righthand side is weakly contractible, so is their product, and this concludes the proof.
\end{proof}

\begin{lemma} \label{lem:LocalizationTwMinToBase}
Let $[m] \xra{\varphi} [1]$ be a functor such that $\varphi^{-1}(1) \neq \es$, and let $W \subseteq \Tw[m]_1$ denote the set of morphisms $[i \leq j] \to [i \leq j']$ such that $\varphi(j) = \varphi(j')$. Form the fiber product $\Tw[m] \times_{\Tw[1]} \sd[n]$ as in \Cref{lem:LocalizationSdToTw} and the fiber product $[m]^\op \times_{[1]^\op} \sd[n]$ with respect to $\varphi^\op$ and $\zeta \min$. Then the map of marked simplicial sets
\[ (\Tw[m],W) \times_{\Tw[1]^{\flat}} \sd[n]^{\flat} \xra{G \coloneqq (\min,\id)} ([m]^\op \times_{[1]^\op} \sd[n])^{\flat} \]
is a weak equivalence in the marked model structure on $\sSet^+$.
\end{lemma}
\begin{proof}
We adopt the same strategy as in the proof of \Cref{lem:LocalizationSdToTw}. By \cite[Theorem~3.8]{MG}, it suffices to show that for every $k \geq 0$, the induced map
\[ \Fun([k], \Tw[m] \times_{\Tw[1]} \sd[n])^W \xlongra{G_k} \Map([k], [m]^\op \times_{[1]^\op} \sd[n]) \]
is a weak homotopy equivalence, where the lefthand side indicates the full subcategory on natural transformations through marked morphisms in $(\Tw[m],W) \times_{\Tw[1]^{\flat}} \sd[n]^{\flat}$. Since the righthand side is discrete, it suffices to show that the fibers of $G_k$ are weakly contractible.
Given a functor $[k] \xra{h=(f,g)} [m]^\op \times_{[1]^\op} \sd[n]$, we define a lift
\[ [k] \xra{\widetilde{h} = (\widetilde{f},g)} \Tw[m] \times_{\Tw[1]} \sd[n] \]
as follows.
\begin{enumerate}[(a)]
\item If $\zeta \max g(k) = 0$ or $\varphi f(0)=1$, let $\widetilde{f} = \left( [f(0)=f(0)] \to [f(1) \leq f(0)] \to ... \to [f(k) \leq f(0)] \right)$.
\item Otherwise, if $\varphi f(0) = 0$ (so that $\zeta \min g(0) = 0$) and $0<i \leq k$ is minimal such that $\zeta \max g(i) = 1$, let
\[ \widetilde{f} = \left( [f(0) = f(0)] \to ... \to [f(i-1) \leq f(0)] \to [f(i) \leq m_0] \to ... \to [f(k) \leq m_0] \right)\]
where $m_0$ is minimal such that $\varphi(m_0) = 1$.
\end{enumerate}
We claim that $\widetilde{h}$ is an initial object in $G_k^{-1}(h)$, which will complete the proof. Suppose that
$$\overline{h} = (f(0)\leq j_0,g(0)) \to ... \to (f(k) \leq j_k, g(k))$$
is another element in the fiber. Then if the projection to $\Tw[1]$ yields the sequence $00 \to ... \to 00$ or $11 \to ... \to 11$, so that we are in the situation of (a), we have for each index $0 \leq l \leq k$ that $[f(l) \leq f(0)] \leq [f(l) \leq j_l]$ in $(\Tw[m], W) \times_{\Tw[1]^\flat} \sd[n]^\flat$, hence $\widetilde{h} \leq \overline{h}$. Otherwise, if the projection to $\Tw[1]$ has $01$ appearing with first instance in the $i$th position so that we are in the situation of (b), then since $\varphi$ is a cocartesian fibration with all morphisms $x \to m_0 \in [m]$ $\varphi$-cocartesian, we likewise have that $\widetilde{h} \leq \overline{h}$.
\end{proof}

\begin{proof}[Proof of \Cref{thm:ExistenceRKE}] We already explained the consequence in \Cref{obs:strategy}. To the task at hand, we first observe that the composite 
\[ \Yo^{\lax} (-)^{\vee}: \Cart_{[1]^\op} \longra \Cocart_{[1]} \longra \Fun(\Delta^\op_{/[1]}, \Cat) \]
is canonically equivalent to the functor given by the assignment
\[ (\cD \to [1]^\op) \longmapsto \left( (\varphi: [m] \to [1]) \mapsto \Fun_{/[1]^\op}([m]^\op, \cD) \right). \]
Indeed, since the ``minmax'' functor $\sd[m] \to \Tw[m]$ given by $\sigma \mapsto (\min \sigma \leq \max \sigma)$ is the unit map of a pointwise left adjoint to the fully faithful inclusion $\locCocart_{[m]} \hookla \Cocart_{[m]}$ \cite[Lemma~2.6.3]{AMGR}, we have a natural equivalence
\[ \Fun^{\cocart}_{/[m]}(\sd[m], \varphi^\ast (\cD^{\vee})) \simeq \Fun^{\cocart}_{/[m]}(\Tw[m], \varphi^\ast (\cD^{\vee}))~. \]
Then by definition of the dual cartesian fibration and its commutativity with base-change, we have a natural equivalence
\[ \Fun^{\cocart}_{/[m]}(\Tw[m], \varphi^\ast (\cD^{\vee})) \simeq \Fun_{/[m]^\op}([m]^\op, (\varphi^\op)^\ast ((\cD^{\vee})^{\vee}) ). \]
Finally, we have a natural equivalence $(\cD^{\vee})^{\vee} \xrightarrow{\sim} \cD$ from the double dual to the identity as in \cite[ Proposition~4.1]{BGN}.

Using \Cref{obs:strategy} to identify $\zeta_{\ast} \Yo^{\lax}$, we thereby reduce to producing a natural equivalence
\[ \Fun^{\cocart}_{/[m] \times_{[1]} [n]}(\sd([m] \times_{[1]} [n]), \overline{\varphi}^{\ast} \cC) \simeq \Fun_{/[1]^\op}([m]^\op, (\zeta^{\lax}_{\ast} \cC)^{\vee} ). \]
Note that the righthand side is by definition $\Fun_{/[n]}(([m]^\op)^\flat \times_{([1]^\op)^\flat} \leftnat{\sd[n]}, \leftnat{\cC})$. In fact, we further have that $[m]^\op \times_{[1]^\op} \sd[n] \to [n]$ is a locally cocartesian fibration with the locally cocartesian morphisms marked as indicated, so we may write this as $\Fun^{\cocart}_{/[n]}([m]^\op \times_{[1]^\op} \sd[n], \cC)$.

We claim that precomposition by the composite functor
\[ \begin{tikzcd}[column sep=8ex]
\sd(\left[m \right] \times_{\left[1 \right]} \left[n \right])
\ar{r}{F}
& \Tw\left[m\right] \times_{\Tw\left[1\right]} \sd\left[n\right]
\ar{r}{G}
& \left[m\right]^\op \times_{\left[1\right]^\op} \sd \left[n \right]
\end{tikzcd} \]
% \[ \sd(\left[m \right] \times_{\left[1 \right]} \left[n \right]) \Tw\left[m\right] \times_{\Tw\left[1\right]} \sd\left[n\right] \left[m\right]^\op \times_{\left[1\right]^\op} \sd \left[n \right] \]
will implement the desired natural equivalence, where $F$ and $G$ are respectively as defined as in Lemmas \ref{lem:LocalizationSdToTw} and \ref{lem:LocalizationTwMinToBase}. We first note that given a commutative square
\[ \begin{tikzcd}
\sd(\left[m \right] \times_{\left[1 \right]} \left[n \right]) \ar{r}{\alpha} \ar{d} & \cC \ar{d} \\ 
\left[m \right] \times_{\left[1 \right]} \left[n \right] \ar{r}[swap]{\overline{\varphi}} & \left[ n \right]
\end{tikzcd} \]
such that $\alpha$ is a morphism of locally cocartesian fibrations, $\alpha$ necessarily sends the class of morphisms $W$ in the statement of \Cref{lem:LocalizationSdToTw} to equivalences in $\cC$ (because the canonical natural transformations of locally cocartesian monodromy functors encoded by morphisms in $W$ sent into $\cC$ are necessarily identities). Therefore, we have the inclusion
\[ \Fun^{\cocart}_{/[m] \times_{[1]} [n]}(\sd([m] \times_{[1]} [n]), \overline{\varphi}^{\ast} \cC) \subseteq \Fun_{/[n]}( (\sd([m] \times_{[1]} [n]) ,W), \cC^{\sim}).  \]
By \Cref{lem:LocalizationSdToTw}, we have an equivalence
\[ \Fun_{/[n]}(\Tw[m] \times_{\Tw[1]} \sd[n], \cC) \xra[\sim]{F^*} \Fun_{/[n]}( (\sd([m] \times_{[1]} [n]) ,W), \cC^{\sim})~, \]
which then restricts to an equivalence
\[ \Fun^{\cocart}_{/[m] \times_{[1]} [n]}(\Tw[m] \times_{\Tw[1]} \sd[n] ,\overline{\varphi}^{\ast} \cC) \xra[\sim]{F^*} \Fun^{\cocart}_{/[m]\times_{[1]} [n]}( \sd( [m] \times_{[1]} [n]), \overline{\varphi}^{\ast} \cC)~. \]
% where we mark $\sd( [m] \times_{[1]} [n])$ as a locally cocartesian fibration over $[m] \times_{[1]} [n]$ via $\max$ and $(\Tw[m] \times_{\Tw[1]} \sd[n])$ as a locally cocartesian fibration over $[m] \times_{[1]} [n]$ via $(\max,\max)$.

Next, we treat two cases for $\varphi$ separately.
\begin{enumerate}
\item Suppose the image of $\varphi$ is $\{0 \} \subset [1]$. Then $[m] \times_{[1]} [n] = [m] \times \{0\}$ and may write $G$ as
\[ \min \times \iota: \Tw[m] \times \{ 0 \} \to [m]^\op \times \sd[n]_{\min=0}. \]
By \cite[Proposition~2.21(2)]{QS}, $\{ 0 \} \to \leftnat{\sd[n]_{\min=0}}$ is an equivalence in the locally cocartesian model structure for $\sSet^+_{/[n]}$. It is also clear that $\Tw[m] \xra{\min} [m]^\op$ exhibits $[m]^\op$ as the localization of $\Tw[m]$ at those morphisms sent to identities under $\min$, i.e., the $\max$-cocartesian morphisms, so $\leftnat{\Tw[m]} \to ([m]^\op)^\flat$ is an equivalence in the marked model structure on $\sSet^+$. By \cite[Remark B.2.5]{HA}, we get that
\[ \leftnat{\Tw[m]} \times \{0\} \longra ([m]^\op)^\flat \times \leftnat{\sd[n]_{\min=0}} \]
is an equivalence in the locally cocartesian model structure on $\sSet^+_{/[n]}$. We conclude that
\[ \Fun^{\cocart}_{/[n]}([m]^\op \times \sd[n]_{\min=0}, \cC) \xrightarrow{G^*} \Fun^{\cocart}_{/[m]}(\Tw[m],\overline{\varphi}^{\ast} \cC) \]
is an equivalence (and both sides are equivalent to $\Fun([m]^\op, \cC_0)$).
\item Suppose that $1 \in [1]$ is in the image of $\varphi$. Then in the notation of \Cref{lem:LocalizationTwMinToBase}, we have an equivalence
\[ \Fun_{/[n]}([m]^\op \times_{[1]^\op} \sd[n], \cC) \xra[\sim]{G^*} \Fun_{/[n]}((\Tw[m],W) \times_{\Tw[1]^\flat} \sd[n]^\flat, \cC^\sim)~. \]
Noting that $W$ is contained in the $\max$-locally cocartesian morphisms, we see that this equivalence restricts to an equivalence
\[ \Fun^{\cocart}_{/[n]}([m]^\op \times_{[1]^\op} \sd[n], \cC) \xra[\sim]{G^*} \Fun^{\cocart}_{/[m] \times_{[1]} [n]}(\Tw[m] \times_{\Tw[1]} \sd[n]),\overline{\varphi}^{\ast} \cC)~. \]
\end{enumerate}
This completes the proof that $\zeta_{\ast} \Yo^{\lax}(\cC \to [n]) \simeq \Yo^{\lax}(\zeta^{\lax}_{\ast}(\cC) \to [1])$.
\end{proof}

\begin{lemma} \label{lem:ZetaAdjointability}
The functor $\locCocart^\lax_{[n]} \xra{\zeta^{\lax}_\ast} \Cocart^\lax_{[1]}$ of \Cref{thm:ExistenceRKE} satisfies the base-change commutativities
\[ \begin{tikzcd}
\Cat \ar{d}{\id} & \locCocart^{\lax}_{[n]} \ar{r}{\iota_1^\ast} \ar{l}[swap]{\iota_0^\ast} \ar{d}{\zeta^{\lax}_{\ast}} & \locCocart^{\lax}_{[n-1]} \ar{d}{\lim^{\lax}_{[n-1]}} \\
\Cat & \Cocart^{\lax}_{[1]} \ar{r}[swap]{\iota^\ast_1} \ar{l}{\iota^\ast_0} & \Cat
\end{tikzcd}~, \]
where $\iota_0$ denotes both inclusions $\{0\} \subset [n]$ and $\{0\} \subset [1]$ and $\iota_1$ denotes both inclusions $\{1,\ldots,n\} \subset [n]$ and $\{1\} \subset [1]$.
\end{lemma}
\begin{proof}
This follows from \Cref{lem:Adjointability} after evaluation at $[0]$ in $\Fun(\Delta^\op,\Cat)$.
\end{proof}

Combining \Cref{thm:ExistenceRKE} and \Cref{lem:ZetaAdjointability}, we arrive at the following result.

%Insert reference to identification of O-monoids in \locCocart^\lax_{[1]} with lax O-monoidal functors of O-monoidal categories
\begin{corollary} \label{cor:InductiveReductionMethod}
Let $\cO^\otimes$ be an operad and let $\cC \to [n]$ be an $\cO$-monoid in $\locCocart^{\lax}_{[n]}$. Let $\cC_0 \xra{\varphi} \lim^{\lax}_{[n-1]} \cC|_{[n-1]}$ be the functor classifying the cartesian fibration $(\zeta^{\lax}_{\ast} \cC)^{\vee}$ of \Cref{con:DualOfLaxRKE}. Then $\varphi$ comes canonically endowed with a laxly $\cO$-monoidal structure such that we have an equivalence
\[ \lim^{\lax}_{[1]} \varphi \simeq \lim^{\lax}_{[n]} \cC \]
of $\cO$-monoidal categories. \qed
\end{corollary}

\subsection{Detecting dualizability on strata and links} \label{sec:2.4}

We now apply \Cref{cor:InductiveReductionMethod} to understand dualizable objects in the lax limit of a locally cocartesian fibration $\cC \to \cB$.

\begin{observation} \label{obs:RestrictionMonoidality}
Let $\cA \xra{f} \cB$ be a functor, and let $\pi_{\cA}$ and $\pi_{\cB}$ denote the projections to a point. Then $\pi_{\cA}^\ast \simeq f^\ast \pi^\ast_{\cB}$ is adjoint to a natural transformation $\lim^\lax_{\cB} \xRightarrow{f^\ast} \lim^\lax_{\cA} f^\ast$ of product-preserving functors from $\locCocart^\lax_{\cB}$ to $\Cat$. For any $\cO$-monoid $\cC \to \cB$ in $\locCocart^\lax_{\cB}$, we thus get an $\cO$-monoidal functor $\lim^\lax_{\cB} \cC \xra{f^*} \lim^\lax_{\cA} f^\ast \cC$.
\end{observation}

%Non-degenerate replaced by non-equivalence?
\begin{theorem} \label{thm:GeneralDualizability}
Let $\cB$ be a category and let $\cC \to \cB$ be a commutative monoid in $\locCocart^\lax_{\cB}$.  Then an object $x \in \lim^\lax_{\cB} \cC$ is dualizable if and only if for every morphism $[1] \xra{f} \cB$ the restriction
\[
f^*x
\in
\lim^\lax_{[1]}(f^* \cC)
\]
is dualizable.
%{\color{red}
%the following two conditions obtain.
%\begin{enumerate}
%\item For every $b \in \cB$, $x_b = b^\ast(x) \in \cC_b$ is dualizable.
%\item For every non-degenerate $f: [1] \to \cB$, $f^\ast(x)$ is dualizable in $\lim^\lax_{[1]} f^\ast \cC$.
%\end{enumerate}}
\end{theorem}
\begin{proof}
%{\color{red} We first note that if $f: [1] \to \cB$ factors through $b \in \cB$, then $x_b$ is dualizable if and only if $f^\ast x$ is dualizable, so we could also combine both conditions into a single one for all morphisms.}
Using that $\sd(\cB) \simeq \underset{[n] \rightarrow \cB}{\colim} \: \sd[n]$ together with \Cref{obs:RestrictionMonoidality}, we have an equivalence
\[ \lim^{\lax}_{\cB} \cC = \Fun^{\cocart}_{/\cB}(\sd(\cB), \cC) \xlongra{\sim}
\lim^\SMC_{([n] \xra{\gamma} \cB) \in \Delta_{/\cB}} \left(
\lim^{\lax}_{[n]} \gamma^\ast \cC\right)
\]
of symmetric monoidal categories. Thus, $x$ is dualizable if and only if $\gamma^{\ast} x$ is dualizable for all $[n] \xra{\gamma} \cB$, which in particular proves the `only if' direction. Furthermore,  it follows that for the `if' direction it suffices to consider the case where $\cB = [n]$.

We proceed by induction on $n \geq 0$. The cases $n=0$ and $n=1$ are evident. So, let us suppose that $n \geq 2$. Suppose we have shown the claim for all $\cB = [m]$ with $m < n$, and assume that for all $[1] \xra{f} [n]$ we have that $f^\ast x \in \lim^\lax_{[1]} (f^\ast \cC)$ is dualizable. In the setup of \Cref{cor:InductiveReductionMethod}, we need to check that $x \in \lim^\lax_{[1]} \varphi$ is dualizable. Let us write $\cX = \lim^\lax_{[n]} \cC$, $\cU = \cC_0$, and $\cZ = \lim^{\lax}_{[n-1]} (\cC|_{[n-1]})$, and let us write $\cX \xra{j^*} \cU$ and $\cX \xra{i^*} \cZ$ for the restrictions. Let $u = j^\ast x$ and $z = i^\ast x$. By the inductive hypothesis, we have that $u$ and $z$ are dualizable. Using \Cref{theorem.algebraic_descr}, it suffices to show that for any $w \in \cU$ the canonical map
\[ z \otimes \varphi(w) \xlongra{\chi} \varphi(u \otimes w) \]
in $\cZ$ is an equivalence. Now note from \Cref{rem:UnwindingDual} and \Cref{obs:RestrictionMonoidality} again that for every $1 \leq k \leq n$, the restriction of $\chi$ over the object $k \in [n]$ is given by $z_k \otimes \varphi_{0k}(x) \to \varphi_{0k}(u \otimes w)$ (where $\cC_0 \xra{\varphi_{0k}} \cC_k$ denotes the monodromy functor). By our assumption applied to each morphism $\{0 < k \} \to [n]$, we see that $\chi_k$ is an equivalence. Hence, by the joint conservativity of restriction to fibers, we deduce that $\chi$ is an equivalence and thus that $x$ is dualizable.
\end{proof}

\begin{remark}
The proof of \Cref{thm:GeneralDualizability} works equally well to characterize left- and right-dualizable objects in the lax limit of a monoid in $\locCocart^\lax_{\cB}$.
\end{remark}

\begin{remark} \label{rem:GeneralInducedRecollementOnModules}
In the setup of \Cref{thm:GeneralDualizability}, if $A \in \cX \coloneqq \lim^\lax_{\cB} \cC$ is a commutative algebra, then we get a symmetric monoidal stratification of $\Mod_A(\cX)$ whose stratum over $b \in \cB$ is given by $\Mod_{A_b}(\cC_b)$, thereby exhibiting $\Mod_A(\cX)$ as a lax limit and furnishing a characterization of its dualizable objects via \Cref{thm:GeneralDualizability} again. To show this, we'll assume known that commutative monoids in $\locCocart^{\lax}_{\cB}$ straighten to oplax functors $\cB \to \SMC^{\lax}$, which follows from $2$-categorical straightening \cite{GoodwillieI}*{\S 3} by an argument similar to the proof of \Cref{prop:StratID}. First note that the $2$-functor $\SMC^{\lax} \xra{\CAlg} \Cat$ is corepresentable, hence $A$ defines a lax cone over the oplax functor $\cB \xra{\cC^\otimes_{\bullet}} \SMC^{\lax}$. Thus, if we let $(\SMC^{\lax})_{\ast \sslash}$ be the lax slice under $\Com^\otimes$, then $A$ defines an oplax functor $\cB \xra{\overline{A}} (\SMC^{\lax})_{\ast \sslash}$. Then since the $2$-functor $(\SMC^{\lax})_{\ast \sslash} \xra{\Mod} \SMC^{\lax}$ is also corepresentable (as a $2$-functor into $\Cat$) and thus preserves lax limits, we obtain the desired symmetric monoidal stratification as $\Mod \circ \overline{A}$. This generalizes \Cref{rem:inducedRecollementOnModules}.
\end{remark}

\section{The \texorpdfstring{$[1]$}{[1]}-stratified bordism hypothesis}\label{section:strat}

In this section, we use \Cref{theorem.intro.alg} to prove \Cref{theorem.intro.bord}. We give the statement (and an outline of the proof) in \Cref{subsection.state.bord}. The main work of the proof consists of three parts, which are accomplished in \S \S \ref{subsec.prove.bord.first.part}, \ref{subsec.prove.bord.second.part}, and \ref{subsec.prove.bord.third.part}.

\begin{notation}
Throughout this section, we fix an object
$$
\varphi \coloneqq (\cU \xra{\varphi} \cZ) \in \Strat,
$$
i.e., a laxly symmetric monoidal functor between symmetric monoidal categories.\footnote{We recall that the morphisms in $\Strat$ are described in \Cref{subsection.intro.bord}; in terms of \Cref{def_strat} (which also introduces $\Stratlax$), this description is justified by \Cref{prop:StratID}.} Additionally, we write
\[
\cX \coloneqq \lim^\lax(\varphi)
\simeq
\lim
\left(
\begin{tikzcd}
&
\Ar(\cZ)
\arrow{d}{\ev_1}
\\
\cU
\arrow{r}[swap]{\varphi}
&
\cZ
\end{tikzcd}
\right)
\]
for the lax limit of $\varphi$.
\end{notation}

For the reader's convenience, we restate the universal property and functoriality of the lax limit as established in \Cref{sec:lax.limit} (cf. the adjunction \Cref{lim.lax.adjunction.monoidal} as well as \Cref{prop:MonoidalLaxLimitID}).

\begin{observation} \label{obs.univ.property.of.lax.limit.of.a.morphism}
For any symmetric monoidal category $\cA$, a morphism $\cA \rightarrow \lim^{\lax}(\varphi)$ in $\SMC$ is equivalent data to a lax commutative diagram
\[ \begin{tikzcd}[ampersand replacement=\&]
\cA
\arrow{d}
\arrow{rd}[sloped, swap, xshift=-0.1cm, yshift=-0.1cm]{\Downarrow}
\\
\cU
\arrow{r}[swap]{\varphi}
\&
\cZ
\end{tikzcd} \]
in $\SMC^{\lax}$ such that the vertical functors are strictly symmetric monoidal. Moreover, a lax commutative square
\[ \begin{tikzcd}[ampersand replacement=\&]
\cU
\arrow{r}{\varphi}[swap, yshift=-0.4cm]{\rotatebox{45}{$\Leftarrow$}}
\arrow{d}[swap]{F_0}
\&
\cZ
\arrow{d}{F_1}
\\
\cU'
\arrow{r}[swap]{\varphi'}
\&
\cZ'
\end{tikzcd} \]
in $\SMC^\lax$ with vertical functors strictly symmetric monoidal extends to a lax commutative diagram
\[ \begin{tikzcd}[ampersand replacement=\&]
\lim^\lax(\varphi)
\arrow{d}
\arrow{rd}[sloped, swap, xshift=-0.2cm, yshift=-0.2cm]{\Downarrow}
\\
\cU
\arrow{r}{\varphi}[swap, yshift=-0.4cm]{\rotatebox{45}{$\Leftarrow$}}
\arrow{d}[swap]{F_0}
\&
\cZ
\arrow{d}{F_1}
\\
\cU'
\arrow{r}[swap]{\varphi'}
\&
\cZ'
\end{tikzcd} \]
and so induces a morphism
\[
\lim^\lax(\varphi)
\longra
\lim^\lax(\varphi')
\]
in $\SMC$ by the universal property of the target. In other words, passage to lax limits defines the second functor in the composite
\[
\Strat
\longhookra
\Stratlax
\xra{\lim^\lax}
\SMC
~.
\]
\end{observation}

% \begin{remark}\label{maps.between.lim.lax}[Also see \Cref{rem.map.gives.map.on.lax.limits}]
% In particular, a lax commutative square
% \[ \begin{tikzcd}[ampersand replacement=\&]
% \cU
% \arrow{r}{\varphi}[swap, yshift=-0.4cm]{\rotatebox{45}{$\Leftarrow$}}
% \arrow{d}[swap]{F_0}
% \&
% \cZ
% \arrow{d}{F_1}
% \\
% \cU'
% \arrow{r}[swap]{\varphi'}
% \&
% \cZ'
% \end{tikzcd} \]
% in $\SMC$ extends to a lax commutative diagram
% \[ \begin{tikzcd}[ampersand replacement=\&]
% \lim^\lax(\varphi)
% \arrow{d}
% \arrow{rd}[sloped, swap, xshift=-0.2cm, yshift=-0.2cm]{\Downarrow}
% \\
% \cU
% \arrow{r}{\varphi}[swap, yshift=-0.4cm]{\rotatebox{45}{$\Leftarrow$}}
% \arrow{d}[swap]{F_0}
% \&
% \cZ
% \arrow{d}{F_1}
% \\
% \cU'
% \arrow{r}[swap]{\varphi'}
% \&
% \cZ'
% \end{tikzcd} \]
% and so induces a morphism
% \[
% \lim^\lax(\varphi)
% \longra
% \lim^\lax(\varphi')
% \]
% in $\SMC$ by the universal property of the target.
% \end{remark}

% {\color{blue} sort out notation below: currently just using $\Strat$, but may eventually be something like $CAlg.\Strat_{[1]}$.}

\begin{remark} \label{rmk.why.defined.Strat.as.such}
If we work in $\Stratlax$ as opposed to $\Strat$ (see \Cref{def_strat} and \Cref{rmk.describe.Strat.informally}), then the adjunction \Cref{lim.lax.adjunction.monoidal} shows that the object $(\Bord \xra{\id_\Bord} \Bord) \in \Strat^\lax$ corepresents dualizable objects in the lax limit. However, we find this to be an unsatisfactory version of a stratified bordism hypothesis: we want morphisms out of our corepresenting object to be as simple to describe as possible. In particular, in $\Stratlax$ we would then end up asking about laxly symmetric monoidal natural transformations between laxly symmetric monoidal functors out of $\Bord$, and we know of no simple universal characterization of these.
\end{remark}
% {\color{cyan} as illustrated in \Cref{obs.univ.property.of.lax.limit.of.a.morphism}, there is slightly more functoriality for right-lax limits: if we allow a map $(F_0,F_1) : \varphi \ra \varphi'$ to just come equipped with a 2-morphism $F_1 \varphi \ra \varphi' F_0$, then we get a symmetric monoidal functor $\lim^\lax(\varphi) \ra \lim^\lax(\varphi')$. call the resulting category $\widetilde{\Strat}$.

% in fact, there's now an easy left adjoint $\SMC \ra \widetilde{\Strat}$, just $\cC \mapsto \id_\cC$, simply due to the universal property of the right-lax limit. in paticular, $\id_\Bord$ corepresents dualizable objects in the lax limit.

% on the other hand, that category is not so good for our purposes, because it is more complicated, and we want morphisms out of our corepresenting object to be as simple to describe as possible. in particular, there we'd end up asking about laxly symmetric monoidal natural transformations between laxly symmetric monoidal functors out of $\Bord$, and (as stated above?) we know of no universal characterization of these. (alternatively, our result can be read as an interpretation of laxly symmetric monoidal natural transformations out of $\Bord$.)}

\subsection{The statement of the \texorpdfstring{$[1]$}{[1]}-stratified bordism hypothesis}
\label{subsection.state.bord}

\begin{definition}
The (1-dimensional framed) \bit{bordism category}, denoted by $\Bord$, is the free symmetric monoidal category on a dualizable object.

By the 1-dimensional bordism hypothesis \cite{Baez-Dolan,TQFT} (see also \cite{Harpaz}), the objects of $\Bord$ are compact signed 0-manifolds,\footnote{More precisely, a signed 0-manifold is a 1-framed 0-manifold.} the morphisms are compact framed 1-dimensional bordisms, and the symmetric monoidal structure is given by disjoint union. We respectively write
\[
+ \in \Bord
\qquad
\text{and}
\qquad
- \in \Bord
\]
for the positively and negatively signed points; these objects are dual to one another, and we consider $+ \in \Bord$ as the free dualizable object.\footnote{These two objects are interchanged by the canonical ${\sf O}(1)$-action on $\Bord$.} We write
\[
\varnothing \in \Bord
\]
for the empty 0-manifold (the symmetric monoidal unit), and we use the notation
\[
M,N, \ldots \in \Bord
\]
to denote generic objects.
\end{definition}

We now define a stratified version of $\Bord$:

\begin{definition}
We define the \bit{stratified bordism category} to be the object
\[
\BBord
\coloneqq
\BBord_{[1]}
\coloneqq
\left( \Bord \xrightarrow{(\id_\Bord , \const_\pt)} \Bord \times \Fin \right)
\in
\Strat
~.
\]
\end{definition}

\begin{remark} Observe that the functor $\Bord \xrightarrow{\id_\Bord} \Bord$ is symmetric monoidal while the functor $\Bord \xrightarrow{\const_\pt} \Fin$ is only laxly symmetric monoidal. Hence, the functor $\BBord$ is only laxly symmetric monoidal.
\end{remark}

\begin{observation}
\label{obs.BBord.is.Ar.Bord.x.Fin}
The lax limit of $\BBord$ can be identified as
\[
\lim^\lax(\BBord)
\simeq
\Ar(\Bord) \times \Fin
\in
\SMC
\]
where the product is of symmetric monoidal categories and $\Ar(\Bord)$ is equipped with its pointwise symmetric monoidal structure. Explicitly, this equivalence is given by the formulas
\[ \begin{tikzcd}[row sep=0cm]
\lim^\lax(\BBord)
\arrow[leftrightarrow]{r}{\sim}
&
\Ar(\Bord) \times \Fin
\\
\rotatebox{90}{$\in$}
&
\rotatebox{90}{$\in$}
\\
\left[M \longmapsto (M,\pt) \longla (N,S) \right]
\arrow[leftrightarrow]{r}
&
(N \longra M , S)
\end{tikzcd}
~.
\]
\end{observation}

\begin{notation}\label{BBord.notations}
For simplicity, we typically refer to objects of $\lim^\lax(\BBord)$ in terms of their corresponding objects in $\Ar(\Bord) \times \Fin$, using the equivalence of \Cref{obs.BBord.is.Ar.Bord.x.Fin}.

We denote by
\[ \begin{tikzcd}[column sep=1.5cm, row sep=0cm]
\Bord
&
\lim^\lax(\BBord)
\arrow{l}[swap]{\ev_0}
\arrow{r}{\ev_1}
&
\Bord \times \Fin
\\
&
\rotatebox{90}{$\simeq$}
\\
&
\Ar(\Bord) \times \Fin
\arrow{luu}[sloped, swap]{t \circ \pr_{\Ar(\Bord)}}
\arrow{ruu}[sloped, swap]{s \times \id_\Fin}
\end{tikzcd}
~.
\]
the evaluation functors $\lim^\lax(\BBord) \xra{\ev_i} \BBord_i$. In particular, there exists a fully faithful right adjoint $\ev_0 \dashv \ev_0^R$ which is moreover laxly symmetric monoidal (as $\ev_0$ is symmetric monoidal).

We also denote by
\[ \begin{tikzcd}[row sep=0cm, column sep=3cm]
\Ar(\Bord)
\arrow[hook]{r}{(\id_{\Ar(\Bord)} , \const_\varnothing)}
\arrow[dashed, hook]{rdd}[sloped, swap]{\iota_{\Ar(\Bord)}}
&
\Ar(\Bord) \times \Fin
\arrow[leftarrow]{r}{(\const_{\id_\varnothing} , \id_\Fin )}
&
\Fin
\arrow[dashed]{ldd}[sloped, swap]{\iota_\Fin}
\\
&
\rotatebox{90}{$\simeq$}
\\
&
\lim^\lax(\BBord)
\end{tikzcd} \]
the indicated morphisms from summands in $\SMC$.
\end{notation}

Similar to the symmetric monoidal category $\Bord$, the $[1]$-stratified symmetric monoidal category $\BBord$ admits a universal dualizable object.

\begin{definition}
We define the \bit{tautological object} of $\lim^\lax(\BBord) \in \SMC$ to be
\[
\tau
\coloneqq
(\id_+,\varnothing)
\in
\lim^\lax(\BBord)
~.
\]
\end{definition}

\begin{observation}
The object $\tau \in \lim^\lax(\BBord)$ is dualizable (as it is dualizable factorwise in $\Ar(\Bord) \times \Fin$), and so it corresponds to a symmetric monoidal functor
\[
\Bord
\xlongra{\tau}
\lim^\lax(\BBord)
~.
\]
\end{observation}

We are now ready to formulate \Cref{theorem.intro.bord}.

\begin{theorem}
\label{thm.stratified.bordism}
The composite morphism of spaces
\[
\Hom_{\Strat}(\BBord,\varphi)
\xra{\lim^\lax}
\Hom_\SMC(\lim^\lax(\BBord),\cX)
\xlongra{\tau^*}
\Hom_\SMC(\Bord,\cX)
\simeq
\cX^\dzbl
\]
is an equivalence.
\end{theorem}

%\begin{remark}
%Informally, one may express \Cref{thm.stratified.bordism} as the assertion that $\BBord$ is the universal $[1]$-stratified symmetric monoidal category with a dualizable object (given by the tautological object $\tau \in \lim^\lax(\BBord)$).
%\end{remark}

\begin{proof}[Proof of \Cref{thm.stratified.bordism}]
We have a commutative diagram
\[ \begin{tikzcd}[row sep=1.5cm]
\Hom_\Strat(\BBord,\varphi)
\arrow{rr}{\lim^\lax}
\arrow[dashed]{rd}[sloped]{\sim}[sloped, swap]{\text{\Cref{lem.hom.in.Strat.to.hom.dagger.an.equivalence}}}
&[-2cm]
&[-2cm]
\Hom_\SMC(\lim^\lax(\BBord) , \cX )
\arrow{rr}{\tau^*}
\arrow[hookleftarrow]{rd}
&[-2cm]
&[-2cm]
\Hom_\SMC(\Bord , \cX)
\\
&
\Hom_\SMC(\lim^\lax(\BBord) , \cX )^\HomforStrat
\arrow[dashed, leftrightarrow]{rr}{\sim}[swap]{\text{\Cref{lemma.subspaces.homforStrat.and.homforCAlg.are.equivalent}}}
\arrow[hook]{ru}
&
&
\Hom_\SMC(\lim^\lax(\BBord) , \cX )^\HomforCAlg
\arrow[dashed]{ru}[sloped]{\sim}[sloped, swap]{\text{\Cref{lem.hom.star.to.hom.in.CAlg.an.equivalence}}}
\end{tikzcd}~, \]
in which
\begin{itemize}

\item the two indicated subspaces of $\Hom_\SMC(\lim^\lax(\BBord),\cX)$ are respectively defined in Notations \ref{notn.for.hom.dagger} and \ref{notn.for.hom.star},

\item the equivalence between them is established as \Cref{lemma.subspaces.homforStrat.and.homforCAlg.are.equivalent},

\item the factorization on the left (exists by \Cref{obs.factorization.of.map.on.hom.spaces} and) is an equivalence by \Cref{lem.hom.in.Strat.to.hom.dagger.an.equivalence}, and

\item the composite on the right is an equivalence by \Cref{lem.hom.star.to.hom.in.CAlg.an.equivalence}. \qedhere
\end{itemize}
\end{proof}

\subsection{The subspace \texorpdfstring{$\Hom^\HomforStrat$}{Hom<} and its equivalence with morphisms in \texorpdfstring{$\Strat$}{Strat}}
\label{subsec.prove.bord.first.part}

In this subsection we define the subspace $\Hom_\SMC(\lim^\lax(\BBord) , \cX )^\HomforStrat \subseteq \Hom_\SMC(\lim^\lax(\BBord) , \cX )$ and prove that it is equivalent to $\Hom_{\Strat}(\BBord,\varphi)$ (see \Cref{lem.hom.in.Strat.to.hom.dagger.an.equivalence}).

\begin{notation}
\label{notn.for.hom.dagger}
We define
\[
\Hom_\SMC(\lim^\lax(\BBord),\cX)^\HomforStrat
\subseteq
\Hom_\SMC(\lim^\lax(\BBord),\cX)
\]
to be the subspace of those symmetric monoidal functors
%\begin{equation}
%\label{typical.s.m.functor.from.limlaxBBord.to.X}
\[
\lim^\lax(\BBord)
\xlongra{F}
\cX
\]
%\end{equation}
satisfying the following conditions.
\begin{enumerate}[label=(\alph*)]

\item\label{define.hom.dagger.factorizations.exist}

There exist factorizations
\begin{equation}
\label{factorizations.of.functor.from.limlaxBBord.to.X.evaluated}
\begin{tikzcd}[column sep=1.5cm]
\Bord
\arrow[dashed]{d}[swap]{F_0}
&
\lim^\lax(\BBord)
\arrow{l}[swap]{\ev_0}
\arrow{r}{\ev_1}
\arrow{d}{F}
&
\Bord \times \Fin
\arrow[dashed]{d}{F_1}
\\
\cU
&
\cX
\arrow{l}{j^*}
\arrow{r}[swap]{i^*}
&
\cZ
\end{tikzcd} 
\end{equation}
in $\Cat$ (see \Cref{BBord.notations}).

\item\label{define.hom.dagger.beck.chevalley}

The resulting lax commutative square
\begin{equation}
\label{lax.comm.square.obtained.by.passing.to.radjts.on.left.side}
\begin{tikzcd}[column sep=2cm, row sep=1.5cm]
\Bord
\arrow[hook]{r}[swap]{\ev_0^R}[xshift=2.1cm, yshift=-1.3cm]{\rotatebox{45}{$\Leftarrow$}}
\arrow{d}[swap]{F_0}
\arrow[bend left]{rr}{(\id_\Bord , \const_\pt)}
&
\lim^\lax(\BBord)
%\arrow{d}{F}
\arrow{r}[swap]{\ev_1}
&
\Bord \times \Fin
\arrow{d}{F_1}
\\
\cU
\arrow[hook]{r}{j_*}
%{c \longmapsto ( c \longmapsto \varphi(c) \xlongla{=} \varphi(c) )}
\arrow[bend right]{rr}[swap]{\varphi}
&
\cX
\arrow{r}{i^*}
&
\cZ
\end{tikzcd}
\end{equation}
in $\Cat$ obtained by passing to horizontal right adjoints in the left commutative square of diagram \Cref{factorizations.of.functor.from.limlaxBBord.to.X.evaluated}  commutes.

\end{enumerate}
\end{notation}

\begin{remark}
Using the joint conservativity of the functors $\cU\xla{j^*} \cX \xra{i^*} \cZ$, it is not hard to see that condition \Cref{define.hom.dagger.beck.chevalley} of \Cref{notn.for.hom.dagger} is equivalent to the following apparently stronger condition.
\begin{itemize}

\item[(b$'$)] The resulting lax commutative square
\[ \begin{tikzcd}[column sep=1.5cm]
\Bord
\arrow[hook]{r}{\ev_0^R}[swap, xshift=0.3cm, yshift=-0.4cm]{\rotatebox{45}{$\Leftarrow$}}
\arrow{d}[swap]{F_0}
&
\lim^\lax(\BBord)
\arrow{d}{F}
\\
\cU
\arrow[hook]{r}[swap]{j^*}
&
\cX
\end{tikzcd} \]
in $\Cat$ obtained by passing to horizontal right adjoints in the left commutative square of diagram \Cref{factorizations.of.functor.from.limlaxBBord.to.X.evaluated}  commutes.

\end{itemize}
\end{remark}

\begin{observation}
\label{obs.factorization.of.map.on.hom.spaces}
By construction, the functor $\Hom_\Strat(\BBord,\varphi) \xra{\lim^{\lax}} \Hom_\SMC(\lim^\lax(\BBord),\cX)$ admits a factorization
\[
\begin{tikzcd}
\Hom_\Strat(\BBord,\varphi)
\arrow{r}{\lim^\lax}
\arrow[dashed]{rd}[sloped, swap]{\lim^\lax}
&
\Hom_\SMC(\lim^\lax(\BBord),\cX)
\\
&
\Hom_\SMC(\lim^\lax(\BBord),\cX)^\HomforStrat
\arrow[hook]{u}
\end{tikzcd}
~.
\]
\end{observation}

\begin{lemma}
\label{lem.hom.in.Strat.to.hom.dagger.an.equivalence}
The morphism
\[
\Hom_\Strat(\BBord,\varphi)
\xra{\lim^\lax}
\Hom_\SMC(\lim^\lax(\BBord),\cX)^\HomforStrat
\]
of spaces from \Cref{obs.factorization.of.map.on.hom.spaces} is an equivalence.
\end{lemma}

To prove \Cref{lem.hom.in.Strat.to.hom.dagger.an.equivalence} we will use the following general result.

\begin{lemma}\label{factorization.in.cat.iff.in.calg} Let 
\[ \begin{tikzcd}
\cB
\arrow[leftarrow, yshift=0.9ex]{r}{L}
\arrow[hook, yshift=-0.9ex]{r}[yshift=-0.2ex]{\bot}[swap]{j}
& \cA
\end{tikzcd} \]
be a symmetric monoidal localization of categories (that is, the functor $j$ is fully faithful, $L$ is left adjoint to $j$, $\cA,\cB$ are symmetric monoidal and the functor $L$ is symmetric monoidal). Then for any morphism $\cA \xra{F} \cE$ in $\SMC$, any factorization 
\[ \begin{tikzcd}
\cA \arrow{r}{L} \arrow{d}[swap]{F} & \cB \arrow[dashed]{dl}
\\
\cE
\end{tikzcd} \]
in $\Cat$ uniquely lifts to a factorization in $\SMC$.
\begin{proof}
Since $L$ is a localization and the functor $j$ is automatically laxly symmetric monoidal (as the right right to symmetric monoidal functor $F$), the dotted functor as above exists if and only if the laxly symmetric monoidal natural transformation $F \rightarrow F j L$ of laxly symmetric monoidal functors is an equivalence, in which case it is given by $F j$. In particular, we see that the above factorization in $\Cat$ can be uniquely lifted to factorization in $\SMC^{\lax}$. To see that the functor $F j$ is actually strictly symmetric monoidal, note that for any $b_1, b_2 \in \cB$ we have
\[
Fj(b_1 \otimes b_2) \simeq Fj(Lj(b_1) \otimes Lj(b_2)) \simeq FjL(j(b_1) \otimes j(b_2)) \xlongla{\sim} F(j(b_1) \otimes j(b_2)) \simeq Fj(b_1) \otimes Fj(b_2)
\]
where the arrow is the equivalence $F \xra{\sim} F j L$.
\end{proof}
\end{lemma}

\begin{observation}
\label{obs.factorizations.are.in.fact.in.CAlg} Since both $\Bord$ and $\Bord \times \Fin$ are symmetric monoidal localizations of $\lim^{\lax}(\BBord) \simeq \Ar(\Bord) \times \Fin$, we see that in condition \Cref{define.hom.dagger.factorizations.exist} of \Cref{notn.for.hom.dagger}, if the factorizations in the diagram \Cref{factorizations.of.functor.from.limlaxBBord.to.X.evaluated} exist then they are unique, and, moreover, they admit unique lifts to $\SMC$ such that the entire commutative diagram \Cref{factorizations.of.functor.from.limlaxBBord.to.X.evaluated} lifts to $\SMC$.
\end{observation}

%\item\label{obs.evaluation.for.BBord.ladjt.is.symmetric.monoidal} The left adjoint $\id_{(-)} \times \id_\Fin$ is symmetric monoidal.

%\item\label{obs.evaluation.for.BBord.radjt.is.laxly.symmetric.monoidal} The right adjoint $(\id_{(-)} , \const_\pt)$ is laxly symmetric monoidal.

\begin{proof}[Proof of \Cref{lem.hom.in.Strat.to.hom.dagger.an.equivalence}]
Fix a point
\[
F
\in
\Hom_\SMC(\lim^\lax(\BBord),\cX)^\HomforStrat
\subseteq
\Hom_\SMC(\lim^\lax(\BBord),\cX)
~.
\]
By \Cref{obs.factorizations.are.in.fact.in.CAlg}, we obtain a lift of the lax commutative square \Cref{lax.comm.square.obtained.by.passing.to.radjts.on.left.side} (which commutes by condition \Cref{define.hom.dagger.beck.chevalley}) from $\Cat$ to $\SMC^\lax$ and hence a morphism $\BBord \ra \varphi$ in $\Strat$. Altogether, this construction assembles into a morphism backwards
\[ \begin{tikzcd}
\Hom_\Strat(\BBord,\varphi)
\arrow{r}{\lim^\lax}
&
\Hom_\SMC(\lim^\lax(\BBord) , \cX)^\HomforStrat
\arrow[bend left, dashed]{l}
\end{tikzcd} \]
of spaces. On the one hand, the dotted arrow is a section by \Cref{obs.factorizations.are.in.fact.in.CAlg}. On the other hand, it is a retraction by \Cref{obs.univ.property.of.lax.limit.of.a.morphism}. 
\end{proof}

\subsection{The subspace \texorpdfstring{$\Hom^\HomforCAlg$}{Hom>} and its equivalence with the subspace \texorpdfstring{$\Hom^\HomforStrat$}{Hom<}}
\label{subsec.prove.bord.second.part}

In this subsection we define the subspace $\Hom_\SMC(\lim^\lax(\BBord) , \cX )^\HomforCAlg \subseteq \Hom_\SMC(\lim^\lax(\BBord) , \cX )$ and prove that it is equivalent to $\Hom_\SMC(\lim^\lax(\BBord) , \cX )^\HomforStrat$ (see \Cref{lemma.subspaces.homforStrat.and.homforCAlg.are.equivalent}).

\begin{notation}
\label{notn.for.hom.star}
We define
\[
\Hom_\SMC(\lim^\lax(\BBord),\cX)^\HomforCAlg
\subseteq
\Hom_\SMC(\lim^\lax(\BBord),\cX)
\]
to be the subspace of those symmetric monoidal functors
\[
\lim^\lax(\BBord)
\xlongra{F}
\cX
\]
satisfying the following conditions.
\begin{enumerate}[label=(\greek*)]

\item\label{define.hom.star.factorizations.exist}

There exist factorizations
\begin{equation}
\label{factorizations.of.functor.from.limlaxidBord.to.X.evaluated}
\begin{tikzcd}[column sep=1.5cm]
\Bord
\arrow[dashed]{dd}[swap]{G_0}
&
\Ar(\Bord)
\arrow{l}[swap]{t}
\arrow{r}{s}
\arrow[hook]{d}{\iota_{\Ar(\Bord)}}
&
\Bord
\arrow[dashed]{dd}{G_1}
\\
%\Bord
&
\lim^\lax(\BBord)
\arrow{d}{F}
%&
%\Bord \times \Fin
%\arrow[dashed]{d}{F_1}
\\
\cU
&
\cX
\arrow{l}{j^*}
\arrow{r}[swap]{i^*}
&
\cZ
\end{tikzcd} 
\end{equation}
in $\Cat$.

\item\label{define.hom.star.Fin.selects.ev.zero.R.of.unit.of.C}

The composite
\[
\Fin
\xhookra{\iota_\Fin}
\lim^\lax(\BBord)
\xlongra{F}
\cX
\]
in $\SMC$ selects the object $j_*(\uno_\cU) \in \CAlg(\cX) \simeq \Hom_\SMC(\Fin , \cX)$.\footnote{This is indeed merely a condition, because the functor $\cU \xra{j_*} \cX$ is fully faithful and laxly symmetric monoidal, so it induces a fully faithful functor $\CAlg(\cU) \xra{j_*} \CAlg(\cX)$, and the object $\uno_\cU \in \CAlg(\cU)$ is initial.}

\end{enumerate}
\end{notation}

\begin{lemma}
\label{lemma.subspaces.homforStrat.and.homforCAlg.are.equivalent}
The subspaces
\[
\Hom_\SMC(\lim^\lax(\BBord),\cX)^\HomforStrat
\subseteq
\Hom_\SMC(\lim^\lax(\BBord),\cX)
\supseteq
\Hom_\SMC(\lim^\lax(\BBord),\cX)^\HomforCAlg
\]
are equal.
\end{lemma}

\begin{notation}
Given a symmetric monoidal category $\cA$ and a dualizable object $a \in \cA^\dzbl$, we simply write
\[
\Bord
\xlongra{a}
\cA
\]
for the corresponding symmetric monoidal functor. We also use the notation
\[
M
\longmapsto
a^{\otimes M}
\]
to indicate its values. So for instance, we have canonical identifications
\[
a^{\otimes +} \simeq a
~,
\qquad
a^{\otimes \varnothing} \simeq \uno_\cA
~,~
\qquad
\text{and}
\qquad
a^{\otimes -} \simeq a^\vee
~.
\]
\end{notation}

\begin{proof}[Proof of \Cref{lemma.subspaces.homforStrat.and.homforCAlg.are.equivalent}]
Choose a morphism
\[
\lim^\lax(\BBord)
\xlongra{F}
\cX
\]
in $\SMC$. We must show that it satisfies conditions \Cref{define.hom.dagger.factorizations.exist} and \Cref{define.hom.dagger.beck.chevalley} of \Cref{notn.for.hom.dagger} if and only if it satisfies conditions \Cref{define.hom.star.factorizations.exist} and \Cref{define.hom.star.Fin.selects.ev.zero.R.of.unit.of.C} of \Cref{notn.for.hom.star}.

Suppose first that the morphism $F$ satisfies conditions \Cref{define.hom.dagger.factorizations.exist} and \Cref{define.hom.dagger.beck.chevalley}.
\begin{itemize}

\item

Using the commutative diagram
\begin{equation}
\label{morphism.over.brax.one.from.Ar.Bord.to.lim.lax.BBord}
\begin{tikzcd}
\Bord
\arrow{d}[swap]{\id_\Bord}
&
\Ar(\Bord)
\arrow{l}[swap]{t}
\arrow{r}{s}
\arrow[hook]{d}{\iota_{\Ar(\Bord)}}
&
\Bord
\arrow{d}{(\id_\Bord,\const_\varnothing )}
\\
\Bord
&
\lim^\lax(\BBord)
\arrow{l}{\ev_0}
\arrow{r}[swap]{\ev_1}
&
\Bord \times \Fin
\end{tikzcd}
\end{equation}
in $\Cat$, we immediately see that condition \Cref{define.hom.star.factorizations.exist} must hold (in fact, it is implied by condition \Cref{define.hom.dagger.factorizations.exist} alone).

\item

Observe that the commutative square \Cref{lax.comm.square.obtained.by.passing.to.radjts.on.left.side} of condition \Cref{define.hom.dagger.beck.chevalley} gives an equivalence
\[
F_1(\varnothing,\pt)
\simeq
\varphi(\uno_\cU)
\]
in $\cZ$. Combining this with the formula
\[ \begin{tikzcd}[row sep=0cm]
\Fin
\arrow{r}{\iota_\Fin}
&
\lim^\lax(\BBord)
\\
\rotatebox{90}{$\in$}
&
\rotatebox{90}{$\in$}
\\
S
\arrow[maps to]{r}
&
( \varnothing
\longmapsto (\varnothing , \pt)
\longla
(\varnothing , S ) )
\end{tikzcd}
~,
\]
we obtain the equivalences
\begin{align*}
F(\iota_\Fin(\pt))
& \simeq
F(\varnothing \longmapsto (\varnothing , \pt) \xlongla{\sim} (\varnothing , \pt) )
\\
& \simeq
(F_0(\varnothing) \longmapsto \varphi(F_0(\varnothing)) \simeq F_1(\varnothing,\pt) \xlongla{\sim} F_1(\varnothing,\pt))
\\
& \simeq
(\uno_\cU \longmapsto \varphi(\uno_\cU) \xlongla{\sim} \varphi(\uno_\cU) )
\\
& \simeq
j_*(\uno_\cU)
\end{align*}
in $\cX$. To upgrade this to an equivalence in $\CAlg(\cX)$, it suffices to observe the pullback square
\[ \begin{tikzcd}
\CAlg(\cU)
\arrow[hook]{r}{j_*}
\arrow{d}[swap]{\fgt}
&
\CAlg(\cX)
\arrow{d}{\fgt}
\\
\cU
\arrow[hook]{r}[swap]{j_*}
&
\cX
\end{tikzcd} \]
in $\Cat$, the initiality of the object $\uno_\cU \in \CAlg(\cU)$, and the fact that the functor
\[
\CAlg(\cU)
\xra{\fgt}
\cU
\]
is conservative. So indeed, condition \Cref{define.hom.star.Fin.selects.ev.zero.R.of.unit.of.C} holds.

\end{itemize}

So indeed, conditions \Cref{define.hom.dagger.factorizations.exist} and \Cref{define.hom.dagger.beck.chevalley} together imply conditions \Cref{define.hom.star.factorizations.exist} and \Cref{define.hom.star.Fin.selects.ev.zero.R.of.unit.of.C}.

Now suppose that the morphism $F$ satisfies conditions \Cref{define.hom.star.factorizations.exist} and \Cref{define.hom.star.Fin.selects.ev.zero.R.of.unit.of.C}.
\begin{itemize}

\item

Note that the equivalence
\[
\lim^\lax(\BBord)
\simeq
\Ar(\Bord) \times \Fin
\]
in $\SMC$ of \Cref{obs.BBord.is.Ar.Bord.x.Fin} is a coproduct decomposition in $\SMC$. Moreover, the factorizations $G_0$ and $G_1$ of condition \Cref{define.hom.star.factorizations.exist} automatically lie in $\SMC$, because in diagram \Cref{factorizations.of.functor.from.limlaxidBord.to.X.evaluated} the composites $j^* \circ F \circ \iota_{\Ar(\Bord)}$ and $i^* \circ F \circ \iota_{\Ar(\Bord)}$ lie in $\SMC$ and the functors $\Bord \xla{t} \Ar(\Bord) \xra{s} \Bord$ are symmetric monoidal localizations. Hence, to verify condition \Cref{define.hom.dagger.factorizations.exist} it suffices to obtain factorizations
\[ \begin{tikzcd}
\pt
\arrow[dashed]{dd}
&
\Fin
\arrow{l}[swap]{\const_\pt}
\arrow{r}{\id_\Fin}[swap]{\sim}
\arrow{d}{\iota_\Fin}
&
\Fin
\arrow[dashed]{dd}
\\
&
\Ar(\Bord) \times \Fin
\arrow{d}{F}
\\
\cU
&
\cX
\arrow{l}{j^*}
\arrow{r}[swap]{i^*}
&
\cZ
\end{tikzcd} \]
in $\SMC$, of which the one on the right is tautological and the one on the left follows from condition \Cref{define.hom.star.Fin.selects.ev.zero.R.of.unit.of.C}.

\item Consider the morphism
\begin{equation}\label{morphism.want.to.check.by.showing.eq.of.star.and.diamond}
F_1 \circ (\id_{\Bord}, \const_{\pt})
\longra
\varphi \circ F_0
\end{equation}
of diagram \Cref{lax.comm.square.obtained.by.passing.to.radjts.on.left.side}. Since $\Bord \times \Fin$ is the coproduct of $\Bord$ and $\Fin$ in $\SMC$, using condition \Cref{define.hom.star.Fin.selects.ev.zero.R.of.unit.of.C} and the commutative diagram \Cref{morphism.over.brax.one.from.Ar.Bord.to.lim.lax.BBord} we have an identification
\[
F_1 \circ (\id_{\Bord}, \const_{\pt}) \simeq \varphi(\uno_{\cU}) \otimes G_1
~.
\]
Furthermore, using the commutative diagram
\begin{equation}
\label{comm.pentagon.from.Bord.down.to.X}
\begin{tikzcd}
&& \Bord \arrow[sloped]{drr}{\id_{\Bord}} \arrow{dll}[sloped]{\id_{\Bord}} \arrow{d}{\id_{(-)}}
\\
\Bord \arrow{d}[swap]{G_0} && \Ar(\Bord) \arrow{d}{F \circ \iota_{\Ar(\Bord)}} \arrow{rr}{s} \arrow{ll}[swap]{t} && \Bord \arrow{d}{G_1}
\\
\cU && \cX \arrow{rr}[swap]{i^*} \arrow{ll}{j^*} && \cZ
\end{tikzcd}~,
\end{equation}
we obtain the identifications
\[
\varphi(\uno_{\cU}) \otimes G_1 \simeq \varphi(\uno_{\cU}) \otimes (i^* \circ F \circ \iota_{\Ar(\Bord)} \circ \id_{(-)})
\]
and
\[
\varphi \circ F_0
=
\varphi \circ G_0
\simeq
\varphi \circ j^* \circ F \circ \iota_{\Ar(\Bord)} \circ \id_{(-)}
~.
\]
Unwinding the definitions, we see that under these identifications the morphism \Cref{morphism.want.to.check.by.showing.eq.of.star.and.diamond} is identified with the composite morphism
\begin{equation}
\label{composite.nat.trans.of.fctrs.from.Bord.to.D}
\varphi(\uno_{\cU}) \otimes G_1
\longra
\varphi(\uno_\cU) \otimes \varphi G_0
\xra{\can}
\varphi(\uno_\cU \otimes G_0)
\simeq
\varphi G_0
~,
\end{equation}
in which the first morphism arises by passing to horizontal right adjoints in the left square of diagram \Cref{comm.pentagon.from.Bord.down.to.X}. Now, consider the vertical composite in diagram \Cref{comm.pentagon.from.Bord.down.to.X}, a symmetric monoidal functor $\Bord \ra \cX$. Let us denote by $[u \mapsto \varphi(u) \xla{\alpha} z] \in \cX$ the dualizable object that it selects. Then, evaluating the composite natural transformation \Cref{composite.nat.trans.of.fctrs.from.Bord.to.D} at an object $M \in \Bord$, we obtain the composite morphism
\[
\varphi(\uno_\cU) \otimes z{\otimes M}
\longra
\varphi(\uno_\cU) \otimes \varphi(u^{\otimes M})
\longra
\varphi(\uno_\cU \otimes u^{\otimes M})
\simeq
\varphi(u^{\otimes M})
\]
in $\cZ$, which is an equivalence by \Cref{lemma.genzd.projection.formula.for.pos.and.neg.tensor.powers} (applied to the case that $w=\uno_{\cU}$).

\end{itemize}
So indeed, conditions \Cref{define.hom.star.factorizations.exist} and \Cref{define.hom.star.Fin.selects.ev.zero.R.of.unit.of.C} together imply conditions \Cref{define.hom.dagger.factorizations.exist} and \Cref{define.hom.dagger.beck.chevalley}.
\end{proof}

\begin{lemma}
\label{lemma.genzd.projection.formula.for.pos.and.neg.tensor.powers}
Suppose that the object $[u \mapsto \varphi(u) \xla{\alpha} z] \in \cX$ is dualizable. Choose any sequence $\varepsilon_1,\ldots,\varepsilon_k \in \{ + , - \}$, and write $M \coloneqq \coprod_{i=1}^k \varepsilon_i \in \Bord$. Then for any $w \in \cU$, the composite morphism
\begin{align}
\label{label.morphism.in.genzd.proj.form}
z^{\otimes M}
\otimes \varphi(w)
\coloneqq &
z^{\otimes \varepsilon_1} \otimes \cdots \otimes z^{\otimes \varepsilon_k}
\otimes \varphi(w)
\longra
\varphi(u^{\otimes \varepsilon_1}) \otimes \cdots \otimes \varphi(u^{\otimes \varepsilon_k})
\otimes \varphi(w) \nonumber \\
& \xra{\can}
\varphi( u^{\otimes \varepsilon_1} \otimes \cdots \otimes u^{\otimes \varepsilon_k} \otimes w)
\eqqcolon
\varphi ( u^{\otimes M} \otimes w)
\end{align}
is an equivalence.
\end{lemma}

\begin{proof}
The claim holds trivially when $k=0$, and when $k=1$ it holds by \Cref{theorem.algebraic_descr} (applied to $x$ if $\varepsilon_1 = +1$ and to $x^\vee$ if $\varepsilon_1 = -1$). If $k \geq 2$, then writing $N = ( \varepsilon_2,\ldots,\varepsilon_k)$, 
%and denoting the morphism \Cref{label.morphism.in.genzd.proj.form} by $u_{M,w}$ (i.e.\! recording its dependence on $M$ and $w$),
we may factor the morphism \Cref{label.morphism.in.genzd.proj.form} as the composite
\[
z^{\otimes M} \otimes \varphi(w)
\simeq
z^{\otimes \varepsilon_1} \otimes z^{\otimes N} \otimes \varphi(w)
\xlongra{\sim}
%\xra{\id_{z^{\otimes N}} \otimes u_{(\varepsilon_k),w}}
z^{\otimes \varepsilon_1} \otimes \varphi(u^{\otimes N} \otimes w)
\xlongra{\sim}
\varphi(u^{\otimes \varepsilon_1} \otimes u^{\otimes N} \otimes w)
\simeq
\varphi(u^{\otimes M} \otimes w)
~,
\]
in which the first morphism is an equivalence by the case where $k=1$ and the second morphism is an equivalence by induction.\footnote{The fact that this is indeed a factorization of the morphism \Cref{label.morphism.in.genzd.proj.form} follows from the compatibility of the various canonical morphisms associated to the laxly symmetric monoidality of $\cU \xra{\varphi} \cZ$.}
\end{proof}

\begin{remark}
\Cref{lemma.genzd.projection.formula.for.pos.and.neg.tensor.powers} applies equally well (with the same proof) in the monoidal case under the assumption that the chosen object is right-dualizable.
\end{remark}

\subsection{The equivalence between the subspace \texorpdfstring{$\Hom^\HomforCAlg$}{Hom>} and morphisms in \texorpdfstring{$\SMC$}{SMC}}
\label{subsec.prove.bord.third.part}

The following result will conclude the proof of \Cref{thm.stratified.bordism}.

\begin{lemma}
\label{lem.hom.star.to.hom.in.CAlg.an.equivalence}
The composite
\begin{equation} 
\label{morphism.in.eq.of.star.with.calg}
\Hom_\SMC(\lim^\lax(\BBord),\cX)^\HomforCAlg
\longra
\Hom_\SMC(\lim^\lax(\BBord),\cX)
\xlongra{\tau^*}
\Hom_\SMC(\Bord,\cX)
\end{equation}
is an equivalence.
\end{lemma}

\begin{notation}
We write
\[
\Hom_{\SMC}(\Ar(\Bord),\cX)^\HomforCAlg \subseteq \Hom_{\SMC}(\Ar(\Bord),\cX)
\]
for the subspace consisting of those morphisms such that there exist factorizations
\begin{equation}
\label{factorizns.defining.hom.star.from.Ar.Bord}
\begin{tikzcd}[column sep=1.5cm]
\Bord
\arrow[dashed]{d}%[swap]{G_0}
&
\Ar(\Bord)
\arrow{l}[swap]{t}
\arrow{r}{s}
\arrow{d}{}
&
\Bord 
\arrow[dashed]{d}%{G_1}
\\
\cU
&
\cX
\arrow{l}{j^*}
\arrow{r}[swap]{i^*}
&
\cZ
\end{tikzcd}
\end{equation}
in $\Cat$.
\end{notation}

\begin{proof}[Proof of \Cref{lem.hom.star.to.hom.in.CAlg.an.equivalence}]
The product decomposition
\[
\Hom_\SMC(\lim^\lax(\BBord),\cX) \simeq \Hom_{\SMC}(\Ar(\Bord),\cX) \times \Hom_{\SMC}(\Fin, \cX)
\]
resulting from \Cref{obs.BBord.is.Ar.Bord.x.Fin} restricts to a product decomposition
\[
\Hom_\SMC(\lim^\lax(\BBord),\cX)^\HomforCAlg
\simeq
\Hom_\SMC(\Ar(\Bord),\cX)^\HomforCAlg \times \{j_*(\uno_\cU) \}
\simeq
\Hom_\SMC(\Ar(\Bord),\cX)^\HomforCAlg
~,
\]
under which the composite \Cref{morphism.in.eq.of.star.with.calg} is identified with the composite
%\begin{equation}
%\label{composite.from.smaller.homforCAlg}
\[
c
:
\Hom_\SMC(\Ar(\Bord),\cX)^\HomforCAlg
\longhookra
\Hom_\SMC(\Ar(\Bord),\cX)
\xra{(\id_{(-)})^*}
\Hom_\SMC(\Bord,\cX)
\]
%\end{equation}
where the latter morphism is precomposition by the morphism $\Bord \xra{\id_{(-)}} \Ar(\Bord)$ in $\SMC$. We claim that the morphism $c$ has an inverse, given by the evident factorization in the diagram (see \Cref{obs.univ.property.of.lax.limit.of.a.morphism})
% \Cref{maps.between.lim.lax}
\[ \begin{tikzcd}
&
&
\Hom_\SMC(\Ar(\Bord) , \cX)^\HomforCAlg
\arrow[hook]{d}
\\
\Hom_\SMC(\Bord , \cX)
\arrow{r}[swap]{\sim}
\arrow[dashed]{rru}[sloped]{f}
&
\Hom_{\Stratlax}(\id_\Bord , \varphi)
\arrow{r}[swap]{\lim^\lax}
&
\Hom_\SMC(\Ar(\Bord) , \cX)
\end{tikzcd}.
\]
Indeed, it follows from \Cref{obs.univ.property.of.lax.limit.of.a.morphism} that the composite $c \circ f$ is equivalent to the identity. Conversely, by \Cref{factorization.in.cat.iff.in.calg}, if the factorizations in diagram \Cref{factorizns.defining.hom.star.from.Ar.Bord} exist then they are unique, and moreover they admit unique lifts to $\SMC$ such that the entire commutative diagram \Cref{factorizns.defining.hom.star.from.Ar.Bord} lifts to $\SMC$, and so it follows from \Cref{obs.univ.property.of.lax.limit.of.a.morphism} that the composite $f \circ c$ is equivalent to the identity.
\end{proof}

\section{Applications and examples}
\label{section.applications.and.examples}

In this section we use \Cref{theorem.algebraic_descr} to describe dualizable objects in various categories of interests.

% \Alg_{\EE_n}(\cV)
\begin{example}[overcategories of $\EE_n$-algebras]
\label{item.generalities.on.dzbls.in.overcat}
For an $\EE_n$-monoidal category $\cV$, an $\EE_n$-algebra $A$ is equivalently a laxly $\EE_n$-monoidal functor $\ast \xra{A} \cV$, and so the overcategory $\cV_{/A} \simeq \lim^{\lax}( \ast \xra{A} \cV)$ obtains an $\EE_n$-monoidal structure. By \Cref{theorem.algebraic_descr} we thus see that an object $X \in \cV_{/A}$ is right-dualizable if and only if $X \in \cV$ is right-dualizable and the composite morphism
\[
X \otimes A \xra{\alpha \otimes \id_A} A \otimes A \xra{{\sf mult}_A} A
\]
in $\cV$ is an equivalence. 
\end{example}

\begin{example}[Thom spectra and orientations]
We discuss a relatively simple application of \Cref{theorem.algebraic_descr}, and explain how it relates to a multiplicative form of the Thom isomorphism theorem (following \cite{ACB}). Let $\cV$ be a presentably $\EE_{n+1}$-monoidal category, and let $A$ be an $\EE_{n+1}$-algebra in $\cV$ for some $n \geq 1$.
\begin{enumerate}

\item\label{item.abstract.orientations} 
For an $\EE_{n}$-monoidal functor $\cC \xlongra{E} \cV$, note that the colimit $\colim(E) \in \cV$ is canonically an $\EE_{n}$-algebra. By the universal property of $\cV_{/A} \simeq \lim^{\lax}( \ast \xra{A} \cV)$, we see that $\Hom_{\Alg_{\EE_{n}}(\cV)}(\colim(E),A)$ is equivalent to the space of $\EE_{n}$-monoidal lifts
\begin{equation}
\label{lifts.from.C.to.V.over.A}
\begin{tikzcd}
& \cV_{/A} \arrow{d}
\\
\cC\arrow{r}[swap]{E} \arrow[dashed]{ur} & \cV.
\end{tikzcd}
\end{equation}
Meanwhile, by an $\EE_{n}$-analogue of \Cref{remark.strictness.on.subcat.gend.by.u.and.udual} we have a pullback diagram
\begin{equation}
\label{pullback.diagram.for.rdzbls.in.V.over.A}
\begin{tikzcd}
(\cV_{/A})^{\sf rdzbl}
\arrow{r}
\arrow{d}
&
\ast
\arrow{d}{A}
\\
\cV^{\sf rdzbl}
\arrow{r}[swap]{(-) \otimes A}
&
{\sf RMod}_{A}(\cV)^{\sf rdzbl}
\end{tikzcd}
\end{equation}
of $\EE_{n}$-monoidal categories. Now, let us assume that all objects of $\cC$ are right-dualizable, so that we obtain a factorization\footnote{Note that we require $E$ to be strictly $\EE_n$-monoidal for this.}
\[ \begin{tikzcd}
\cC
\arrow{rr}{E}
\arrow[dashed]{rd}[sloped, swap]{E}
&
&
\cV
\\
&
\cV^{{\sf rdzbl}}
\arrow[hook]{ru}
\end{tikzcd}~ \]
and similarly for $\cV_{/A}$.
Then, using the pullback diagram \Cref{pullback.diagram.for.rdzbls.in.V.over.A}, we find that the space of $\EE_{n}$-monoidal lifts \Cref{lifts.from.C.to.V.over.A} is equivalent to the space of $\EE_{n}$-monoidal lifts
\[ \begin{tikzcd}[column sep=1.5cm]
& \ast \arrow{d}{A}
\\
\cC\arrow{r}[swap]{E \otimes A} \arrow[dashed]{ur} & {\sf RMod}_{A}(\cV),
\end{tikzcd} \]
i.e., the space ${\sf Orient}_A(E)$ of \bit{$A$-orientations of $E$}. Summarizing, we obtain an equivalence
$$
\Hom_{\Alg_{\EE_{n}}(\cV)}(\colim(E),A) \simeq {\sf Equiv}_{\Fun_{\Alg_{\EE_{n}}(\Cat)}(\cC,{\sf RMod}_{A}(\cV))}(\const_A, E \otimes A) \eqqcolon {\sf Orient}_A(E)
$$
of spaces.

\item\label{item.orientations.of.thom.spectra}

We apply \Cref{item.abstract.orientations} in the case $\cV = \Spectra$ is the presentably symmetric monoidal category of spectra and $\cC = K \in \Alg_{\EE_{n}}^{{\sf gp}}(\Spaces)$ is a grouplike $\EE_{n}$-algebra in spaces.\footnote{Since $K$ is a grouplike $\EE_n$-groupoid, the condition of right-dualizability in $K$ is automatic.} Then, the equivalence above yields an equivalence
$$
\Hom_{\Alg_{\EE_{n}}(\Spectra)}(\colim(E),A) \simeq {\sf Orient}_A(E)
$$
of spaces. In particular, each morphism
\[
\colim(E)
\longra
A
\]
in $\Alg_{\EE_{n}}(\Spectra)$ determines an equivalence 
$$
\colim(E) \otimes A \simeq \colim(E \otimes A) \xlongra{\sim} \colim\left(K \xra{\const_A} {\sf RMod}_{A}(\Spectra)\right) \simeq \Sigma^{\infty}_+ K \otimes A
$$
in $\Alg_{\EE_{n}}({\sf RMod}_A(\Spectra))$, which implies that the $A$-homology of $\colim(E)$ is naturally isomorphic to the $A$-homology of $K$.

\item

We illustrate \Cref{item.orientations.of.thom.spectra} in the specific example that $K={\sf BO}$ and $E$ is given by the $J$-homomorphism
$$
J: {\sf BO} \longra {\sf BAut}_{\Spectra}(\SSS) \longhookra \Spectra
~.
$$
In particular, for a space $K' \in \Spaces$ and a stable vector bundle $K' \xra{V} {\sf BO}$ on $K'$, the composite $K' \xra{V} {\sf BO} \xra{J} \Spectra$ classifies bundle of spectra obtained by $1$-point compactifying and then applying $\Sigma^{\infty}$ to the fibers of $V$. Then, we obtain a universal property for ${\sf MO}=\colim(J)$ as an $\EE_{n}$-algebra (for any $n \geq 1$): morphisms
\[
{\sf MO} \simeq \colim(J) \longra A
\]
in $\Alg_{\EE_{n}}(\Spectra)$ are equivalent to the space of $A$-orientations ${\sf Orient}_A(J)$ of the $J$-homomorphism. In particular, a morphism ${\sf MO} \rightarrow A$ of $\EE_{n}$-algebras induces a splitting ${\sf MO} \otimes A \simeq \Sigma^{\infty}_+ {\sf BO} \otimes A$ and implies that the $A$-homology of ${\sf MO}$ is naturally isomorphic to the $A$-homology of ${\sf BO}$.

\end{enumerate}

\end{example}

\begin{example}[Picard groups]
Recall that the Picard spectrum functor $\SMC \xra{\mathfrak{pic}} \Spectra_{\geq 0}$ preserves limits \cite[Proposition~2.2.3]{MS}. Applying this to the pullback square of \Cref{remark.strictness.on.subcat.gend.by.u.and.udual} and writing $\Pic \coloneqq \pi_0 \mathfrak{pic}$ for the Picard group, we get a long exact sequence on homotopy groups
\[ ... \to \pi_1(\mathfrak{pic}(\cU)) \times \pi_1(\mathfrak{pic}(\cZ)) \to \pi_1 \mathfrak{pic}(\Mod_{\varphi(\uno_{\cU})}(\cZ)) \xra{\partial} \Pic(\cX) \to \Pic(\cU) \times \Pic(\cZ) \:.\footnote{Here, we are free to replace $\cU_0$ (as defined in \Cref{remark.strictness.on.subcat.gend.by.u.and.udual}) with $\cU$ since $\Pic(\cU_0) \subseteq \Pic(\cU)$ and $\pi_i \mathfrak{pic}$ for $i>0$ is unchanged.} \]
Under the isomorphism
\[
\pi_1 \mathfrak{pic}(\Mod_{\varphi(\uno_{\cU})}(\cZ))
\cong
\End_{\Mod_{\varphi(\uno_\cU)}(\cZ)}(\varphi(\uno_{\cU}))^\times
~,
\]
one may identify the boundary map $\partial$ as sending an automorphism $\varphi(\uno_{\cU}) \xra{\gamma} \varphi(\uno_{\cU})$ to the invertible object $[\uno_{\cU} \mapsto \varphi(\uno_{\cU}) \xla{\gamma \circ \can} \uno_{\cZ}]$. The map $\partial$ then implements an isomorphism of abelian groups
\[ \partial: \End(\uno_{\cU})^\times \backslash \End(\varphi(\uno_{\cU}))^\times / \End(\uno_{\cZ})^\times \xra{\cong} \ker \left( \Pic(\cX) \to \Pic(\cU) \times \Pic(\cZ) \right) \:. \]
This isomorphism is originally due to Tom Bachmann in a slightly different form \cite[Corollary 3.6]{Bachmann}.
\end{example}

\begin{example}[modules]
\label{example.modules.over.complete.local.cring}
Let $R$ be a (discrete) complete local commutative ring and let us write $\mathfrak{m}$ for its maximal ideal and $K$ for its fraction field. We then have an equivalence
\[
\Mod_R
\xlongra{\sim}
\lim^\lax \left( \Mod_R^{\mathfrak{m}\text{-}{\sf cplt}} \xra{ (-) \otimes_R K} \Mod_K \right)
\]
of categories. Note that an $R$-module is dualizable if and only if it is perfect, and all perfect $R$-modules are $\mathfrak{m}$-complete. In particular, at least when $R$ is regular, these are \textit{all} of the dualizable objects in $\Mod_R^{\mathfrak{m}\text{-}{\sf cplt}}$, i.e., completion determines an equivalence
\[
\Mod_R^\dzbl
\xlongra{\sim}
(\Mod_R^{\mathfrak{m}\text{-}{\sf cplt}})^\dzbl
~.
\]
It follows that the composite
\[
(\Mod_R^{\mathfrak{m}\text{-}{\sf cplt}})^\dzbl
\longhookra
\Mod_R^{\mathfrak{m}\text{-}{\sf cplt}}
\xra{(-) \otimes_R K}
\Mod_K
\]
is strictly symmetric monoidal. This is consistent with \Cref{theorem.algebraic_descr}, in view of \Cref{rmk.strictly.s.m.case}.
\end{example}

The following example was suggested by Tomer Schlank.

\begin{example}[$E(1)$-local spectra]
To see the necessity of the second condition in $(3)$, \Cref{theorem.algebraic_descr}, consider the equivalence (in fact, a recollement)
\[
L_{E(1)} \Spectra
\xlongra{\sim}
\lim^\lax \left( L_{K(1)}\Spectra \xra{ (-) \otimes_\SSS \QQ} \Mod_\QQ \right)
~.
\]
By \cite[15.1]{HS} we know that there exist invertible $K(1)$-local spectra with torsion homotopy groups (in particular, their rationalization is zero). Consider such a Picard element $U \in \Pic(L_{K(1)} \Spectra)$ and the corresponding object $[U \mapsto 0 \xla{=} 0]$ in $L_{K(1)}\Spectra$. Since $U$ is dualizable and the natural morphism $0 \to 0 \otimes_\SSS (L_{K(1)} \SSS \otimes_\SSS \QQ) \rightarrow U \otimes \QQ$ is tautologically an equivalence, we see that $[U \mapsto 0 \xla{=} 0]$ satisfies the first condition in (3) of \Cref{theorem.algebraic_descr}. On the other hand, it doesn't satisfy the second condition since $L_{K(1)} \SSS \otimes_{\SSS} \QQ$ is nonzero.
By \Cref{stable_case} it then follows that such invertible objects of $L_{K(1)}\Spectra$ cannot lie in the thick subcategory generated by the unit (compare \cite{AMN}*{Example~2.8}).
\end{example}

In the next few examples, we consider the category $\Sp^G$ of (genuine) $G$-spectra when $G$ is a finite group. Recall from \cite{AMGR}*{Theorems~A~and~E} that $\Sp^G$ is symmetric monoidally equivalent to the lax limit of its `canonical fracture', which is a locally cocartesian fibration over the subconjugacy poset $P$ of $G$ with stratum over $H \in P$ given by $\Sp^{h W_G H} = \Fun(B W_G H, \Sp)$ for $W_G H$ the Weyl group and monodromy given by the proper Tate construction $(-)^{\tau (-)}$ with residual action, at least when $G$ is abelian (for a general formula, see \cite{AMGRnew}*{Theorem A}). This generalizes the symmetric monoidal equivalence
\[ \Sp^{C_p} \simeq \lim^{\lax} \left( \Sp^{h C_p} \xra{(-)^{tC_p}} \Sp \right) \:. \]
In view of equivariant Atiyah duality, we already know that the dualizable and compact objects of $\Sp^G$ coincide (and are generated by the orbits $\Sigma^\infty_+(G/H)$ as a thick subcategory). In the following examples, we use Theorems \ref{theorem.algebraic_descr} and \ref{thm:GeneralDualizability} to describe some novel features of these dualizable objects.

\begin{example}[$C_p$-spectra] \label{exm:Cp}
Suppose that $G = C_p$. The Segal conjecture implies that $\SSS^{tC_p} \simeq \SSS^{\wedge}_p$. \Cref{theorem.algebraic_descr} then shows that a $C_p$-spectrum $X = [X^{\Phi e} \mapsto X^{tC_p} \xla{\alpha} X^{\Phi C_p}]$ is dualizable if and only if $X, X^{\Phi C_p}$ are finite as spectra and the canonical maps $(X^{\Phi C_p})^{\wedge}_p \to X^{tC_p}$, $X^{\Phi C_p} \otimes (X^{\vee})^{tC_p} \to (X \otimes X^{\vee})^{tC_p}$ are equivalences. Conceptually, we thus think of $X^{\Phi C_p}$ as providing a `decompletion' of $X^{tC_p}$.

Note in addition that by a thick subcategory argument as in \Cref{stable_case} and using that the $C_p$-Tate construction vanishes on $C_p$-induced objects, verifying the projection formula at $X^{\vee}$ may be traded for checking that $X$ is in the thick subcategory of $\Sp^{hC_p}$ generated by the objects $\Sigma^{\infty}_+ (C_p/e)$ and $\Sigma^\infty_+(C_p/C_p)$ (note that this coincides with the thick essential image of $j^*: (\Sp^{C_p})^{\dzbl} \to (\Sp^{hC_p})^{\dzbl}$ and thus leads to a necessary condition; compare \Cref{rem:BeauvilleLaszlo}).
\end{example}

\begin{example}[$C_{p^n}$-spectra] \label{exm:Cpn}
Suppose $G = C_{p^n}$. Then the Tate orbit lemma \cite{NS}*{Lemma~I.2.1} implies that the full subcategory $\Sp^{C_{p^n}}_{+}$ on bounded-below objects (i.e., whose geometric fixed points are all bounded-below; cf. \cite{QS}*{Lemma~7.8}) decomposes as an iterated pullback via geometric fixed points (cf. \cite{QS}*{Proposition~7.38 and 2.52} and also \cite{NS}*{Remark~II.4.8})
\[ \Sp^{C_{p^n}}_{+} \simeq \Sp^{h C_{p^n}}_{+} \times_{\scriptscriptstyle (-)^{tC_p}, \Sp^{h C_{p^{n-1}}}, \ev_1} \Ar'(\Sp^{h C_{p^{n-1}}}) \times_{\scriptscriptstyle (-)^{tC_p} \ev_0, \Sp^{h C_{p^{n-2}}}, \ev_1} \Ar'(\Sp^{h C_{p^{n-2}}}) \times \cdots \times \Ar'(\Sp) \:,   \]
where $\Ar'$ denotes the full subcategory on arrows whose source is bounded-below. In fact, if we let
$$\cX^k \coloneqq \lim^{\lax} \left( \Sp^{h C_{p^k}}_{+} \xra{(-)^{tC_p}} \Sp^{h C_{p^{k-1}}} \right)$$
and $\cX^k_+ \coloneqq \Sp^{h C_{p^{k-1}}}_+ \times_{\Sp^{h C_{p^{k-1}}}, i^*} \cX^k$, then we have a symmetric monoidal equivalence
\[ \Sp^{C_{p^n}}_+ \simeq \cX^n_+ \times_{\scriptstyle i^*,\Sp^{hC_{p^{n-1}}}_{+},j^*} \cX^{n-1}_+ \times \cdots \times \cX^1_+ \:. \]
It follows that the dualizability of a $C_{p^n}$-spectrum $X$ may be checked in terms of dualizability of its geometric fixed points $\{ X^{\Phi C_{p^k}} \}$ and the projection formulas
\[ (X^{\Phi C_{p^{k+1}}})^{\wedge}_p \xlongra{\sim} (X^{\Phi C_{p^k}})^{t C_p} \qquad \text{and} \qquad (X^{\Phi C_{p^{k+1}}}) \otimes ((X^{\Phi C_{p^k}})^{\vee})^{t C_p} \xlongra{\sim} (X^{\Phi C_{p^k}} \otimes (X^{\Phi C_{p^k}})^{\vee})^{t C_p}. \]
\end{example}

\begin{example}[$G$-spectra] \label{exm:Gspectra}
If $G$ is a finite $p$-group, then the Segal conjecture for $G$ and its subgroups together with an isotropy separation argument imply that $\SSS^{\tau G} \simeq \SSS^{\wedge}_p$. In lieu of the simplification afforded by the Tate orbit lemma for the $C_{p^n}$ case, we may instead use \Cref{thm:GeneralDualizability} along with \cite{AMGRnew}*{Theorem A} to describe the dualizable objects in $\Sp^G$. In fact, the $p$-group case is really the interesting one to consider since if two distinct primes divide the order of $G$, then $(-)^{\tau G}$ vanishes \cite{QS}*{Lemma~3.50}.

In any case, as in \Cref{exm:Cp} we may replace the projection formula condition at duals with the condition that $X^{\Phi H} \in \Sp^{h W_G H}$ is in the thick subcategory generated by the $W_G H$-orbits, using that the proper Tate construction annihilates all non-trivial orbits.
\end{example}

\begin{example}[Tate diagonals]
Consider the Hill-Hopkins-Ravenel norm functor $\Sp \xra{N^{C_p}} \Sp^{C_p}$ given by $N^{C_p}(X) = [X^{\otimes p} \mapsto (X^{\otimes p})^{t C_p} \xla{\Delta} X ]$, where $C_p$ acts on $X^{\otimes p}$ by permutation and $\Delta$ is the Tate diagonal. If $X$ is finite, then by \Cref{exm:Cp} $\Delta$ exhibits $(X^{\otimes p})^{t C_p}$ as the $p$-completion of $X$, since $N^{C_p}$ is symmetric monoidal. In fact, this holds more generally when $X$ is bounded-below \cite{NS}*{Theorem~III.1.7}. On the other hand, suppose instead that $G$ is any finite $p$-group. Then for any finite spectrum $X$, by \Cref{exm:Gspectra} we again have that $X \xra{\Delta} (X^{\otimes |G|})^{\tau G}$ is a $p$-completion map. However, in this case we are not aware of a generalization to the case of $X$ bounded-below.
\end{example}

\begin{example}[derived Mackey functors]
Let $G$ be any finite group and $H \ZZ^{\triv} \in \CAlg(\Sp^G)$ be the image of $H \ZZ$ under the unique symmetric monoidal colimit-preserving functor $\Sp \xra{\triv} \Sp^G$. Then $\Mod_{H \ZZ^{\triv}}(\Sp^G)$ is equivalent to the category of derived Mackey functors $\Fun^{\oplus}(\Span(\Fin_G), D(\ZZ))$ \cite{PSW}. More relevantly for us, $\Mod_{H \ZZ^{\triv}}(\Sp^G)$ admits a description as a lax limit like with $\Sp^G$; one may either use \Cref{rem:GeneralInducedRecollementOnModules} or appeal to \cite{AMGRnew}*{Theorem A}. However, in this algebraic setting, one now has that the orbits generate $(D(\ZZ)^{h G})^{\dzbl}$ as a thick subcategory; this follows immediately from the same assertion for the stable module category $(D(\ZZ)^{h G})^{\dzbl}/(D(\ZZ)^{h G})^{\omega}$ established by \cite{Krause}*{Lemma~5.6} and the fact that $(D(\ZZ)^{h G})^{\omega}$ is generated by $\ZZ[G]$. Consequently, when applying \Cref{thm:GeneralDualizability}, the projection formula of \Cref{theorem.algebraic_descr} now only needs to be checked at the unit for all the monodromy functors.
\end{example}

\begin{example}[cyclotomic spectra]
Let $\cC$ be a symmetric monoidal category and $\cC \xra{F} \cC$ a laxly symmetric monoidal endofunctor. The \emph{lax equalizer} $\LEq(\id,F)$ of $\id$ and $F$ is defined to be the pullback $\cC \times_{(\id,F),\cC \times \cC, (\ev_0, \ev_1)} \Ar(\cC)$ \cite{NS}*{Definition~II.1.4}. If we let $\cX$ denote the lax limit of $F$, then we may rewrite this as
\begin{equation} \label{eq:laxequalizer}
\LEq(\id,F) \simeq \cC \times_{(\id,\id), \cC \times \cC, (i^*, j^*)} \cX \:.
\end{equation}
% given on objects by
% $$[\varphi: x \to F(x)] \otimes [\varphi': x' \to F(x')] = [x \otimes x' \xra{\varphi \otimes \varphi'} F(x) \otimes F(x') \xra{\can} F(x \otimes x')] \:,$$
Moreover, $\LEq(\id,F)$ admits a symmetric monoidal structure \cite{NS}*{Construction IV.2.1} such that \Cref{eq:laxequalizer} is an equivalence of symmetric monoidal categories. Consequently, we may apply \Cref{theorem.algebraic_descr} to see that an object $[x \xra{\varphi} F(x)]$ is dualizable in $\LEq(\id,F)$ if and only if $x$ is dualizable in $\cC$ and the canonical maps $x \otimes F(1) \ra F(x)$ and $x \otimes F(x^{\vee}) \ra F(x \otimes x^{\vee})$ are equivalences.

Now suppose $\cC = \Fun(B S^1, \Sp)$ and $F = (-)^{t C_p}$, so that $\LEq(\id,F)$ is the category of $p$-cyclotomic spectra $\CycSp_p$.\footnote{The Nikolaus-Scholze category $\CycSp$ of cyclotomic spectra agrees with the classical category $\CycSp^{\mathrm{gen}}$ on full subcategories of bounded-below objects \cite{NS}*{Theorem~1.4}, so we may as well work with the Nikolaus-Scholze category when studying dualizability.} Then by the Segal conjecture, $F(\SSS) \simeq \SSS^{\wedge}_p$, with trivial residual action since the action is through ring maps. We deduce that the cyclotomic Frobenius of a dualizable $p$-cyclotomic spectrum identifies with the $p$-completion map, and likewise for cyclotomic spectra, for which dualizability may be treated at each prime independently. Recalling that $\THH$ defines a symmetric monoidal functor $\Cat^{\mathrm{perf}} \to \CycSp$, we see that this strongly constrains $\THH$ of dualizable (i.e., smooth and proper) stable idempotent complete categories.

Furthermore, using \Cref{rem:inducedRecollementOnModules} and the lift of $\THH$ to a symmetric monoidal functor $\Cat^{\mathrm{perf}}_k \to \Mod_{\THH(k)}(\CycSp)$, we obtain a similar constraint on $\THH$ of smooth and proper $k$-linear stable idempotent complete categories. For instance, suppose $k$ is a perfect field of characteristic $p$. Then by \Cref{rem:inducedRecollementOnModules}, \Cref{stable_case}, and the fact that every dualizable object in $\Mod_{\THH(k)}(\Sp^{hS^1})$ is perfect \cite{AMN}*{Theorem~3.2}, a $\THH(k)$-module $[X \xra{\alpha} X^{tC_p}]$ in $\CycSp_p$ is dualizable if and only if $X$ is dualizable in $\Mod_{\THH(k)}(\Sp^{hS^1})$ and $\alpha$ factors as base-change along $\THH(k) \xra{\varphi} \THH(k)^{tC_p}$ followed by an equivalence
$$X \otimes_{\THH(k)} \THH(k)^{tC_p} \xlongra{\sim} X^{tC_p} \:.$$
Moreover, we then have for any $Y \in \Mod_{\THH(k)}(\Sp^{hS^1})$, an equivalence
$$X^{tC_p} \otimes_{\THH(k)^{tC_p}} Y^{tC_p} \simeq (X \otimes_{\THH(k)} Y)^{tC_p} \:,$$
which recovers the $H=C_p$ case of \cite{AMN}*{Theorem~3.3} and also shows that it is not logically dependent on \cite{AMN}*{Theorem~3.2} (which was only used here to simplify the criterion of \Cref{theorem.algebraic_descr}).
\end{example}

\end{document}